\documentclass[reqno,11pt]{amsart}
\usepackage{amsmath,amssymb,latexsym,soul,cite,mathrsfs}
\usepackage{color,enumitem,graphicx}
\usepackage[colorlinks=true,urlcolor=blue,
citecolor=red,linkcolor=blue,linktocpage,pdfpagelabels,
bookmarksnumbered,bookmarksopen]{hyperref}
\usepackage[english]{babel}
\usepackage[left=2.6cm,right=2.6cm,top=2.9cm,bottom=2.9cm]{geometry}
\usepackage[hyperpageref]{backref}

\newtheorem{theorem}{Theorem}
\newtheorem{lemma}[theorem]{Lemma}
\newtheorem{corollary}[theorem]{Corollary}
\newtheorem{proposition}[theorem]{Proposition}
\newtheorem{remark}[theorem]{Remark}
\newtheorem{definition}[theorem]{Definition}

\newtheorem{theoremletter}{Theorem}
\newtheorem{propositionletter}{Proposition}

\newenvironment{acknowledgement}{\noindent\textbf{Acknowledgments}}{}
\newtheoremstyle{tttheorem}
{}                
{}                
{\slshape}        
{}                
{\bfseries}       
{'}               
{ }               
{}                
\theoremstyle{tttheorem}
\newtheorem{theoremtio}{Theorem}

\newcommand{\dive}{\mathrm{div}}

\newcommand{\ud}{\mathrm{d}}
\newcommand{\loc}{\mathrm{loc}}
\newcommand{\vspan}{\mathrm{span}}
\newcommand*{\avint}{\mathop{\ooalign{$\int$\cr$-$}}}

\DeclareMathOperator{\Ric}{Ric}
\DeclareMathOperator{\spec}{spec}
\DeclareMathOperator{\trace}{tr}

\title[Asymptotics for singular solutions to fourth order systems]{Asymptotics for singular solutions to conformally invariant fourth order systems in the punctured ball}  
\thanks{Research supported in part by Fulbright Commission in Brazil grant G-1-00001, CNPq grant 305726/2017-0, and the Coordena\c c\~ao de Aperfei\c coamento de Pessoal de N\'ivel Superior - Brasil (CAPES) grant 88882.440505/2019-01}

\author[J.H. Andrade]{Jo\~{a}o Henrique\ Andrade}
\author[J.M. do \'O]{Jo\~ao Marcos do \'O*}

\address[J.H. Andrade]{
	\newline\indent 
	Department of Mathematics,
	Federal University of Para\'{\i}ba 
	\newline\indent 
	58051-900, Jo\~ao Pessoa-PB, Brazil
	\newline\indent
	and
	\newline\indent
	Department of Mathematics,
	Princeton University
	\newline\indent 
	08540, Princeton, New Jersey, USA}
\email{\href{mailto:jhsda@princeton.edu}{jhsda@princeton.edu}}
\email{\href{mailto:andradejh@mat.ufpb.br}{andradejh@mat.ufpb.br}}

\address[J.M. do \'O]{Department of Mathematics,
	Federal University of Para\'{\i}ba
	\newline\indent 
	58051-900, Jo\~ao Pessoa-PB, Brazil}
\email{\href{mailto:jmbo@pq.cnpq.br}{jmbo@pq.cnpq.br}}

\thanks{* Corresponding author.}
\begin{document}
	
	\begin{abstract}
		We study the asymptotic behavior for singular solutions to a critical fourth order system generalizing the constant $Q$-curvature equation. 
		Our main result extends to the case of strongly coupled systems, the celebrated asymptotic classification due to [L. A. Caffarelli, B. Gidas and J. Spruck, Comm. Pure Appl. Math. (1989)] and [N. Korevaar, R. Mazzeo, F. Pacard and R. Schoen, Invent. Math., (1999)].
		On the technical level, we use an involved spectral analysis to study the Jacobi fields' growth properties in the kernel of the linearization of our system around a blow-up limit solution. 
		Besides, we obtain sharp a priori estimates for the decay rate of singular solutions near the origin.
		Consequently, we prove that sufficiently close to the isolated singularity solutions behave like the so-called Emden--Fowler solution. Our main theorem positively answers a question posed by [R. L. Frank and T. K\"onig,  Anal. PDE (2019)] concerning the local behavior close to the isolated singularity for scalar solutions in the punctured ball.
	\end{abstract}
	
	\maketitle
	
	
	\begin{center}
		\footnotesize
		\tableofcontents
	\end{center}
	
	\section{Description of the results}\label{sec:introduction}
	
	In this paper, we study the local behavior for strongly positive singular $p$-map solutions $\mathcal{U}=(u_1,\dots,u_p): B_1^*\rightarrow \mathbb{R}^p$ to the {\it critical  fourth order} system,
	\begin{equation}\label{oursystem}\tag{$\mathcal{S}_p$}
	\Delta^{2} u_{i}=c(n)|\mathcal{U}|^{2^{**}-2}u_{i} \quad {\rm in} \quad B_1^*,
	\end{equation}
	where $B_1^*:=B_1^{n}(0)\setminus\{0\}\subset\mathbb{R}^n$ is the unit {\it punctured ball}, $n\geqslant5$, 
	$\Delta^{2}$ is the bi-Laplacian and $|\mathcal{U}|$ is the Euclidean norm, that is, $|\mathcal{U}|=(\sum_{i=1}^{p} u_i^2)^{1/2}$.
	System \eqref{oursystem} is strongly coupled by the {\it Gross--Pitaevskii nonlinearity}  $f_i(\mathcal{U})=c(n)|\mathcal{U}|^{2^{**}-2}u_i$ with associated potential $F(\mathcal{U})=(f_1(\mathcal{U}),\dots,f_p(\mathcal{U}))$, where $s\in(1,2^{**}-1)$ with $2^{**}=2n/(n-4)$ the {\it critical Sobolev exponent}, and $c(n)=[n(n-4)(n^2-4)]/{16}$ a normalizing constant.
	
	By a {\it classical solution} to \eqref{oursystem}, we mean a $p$-map $\mathcal{U}$ such that each component $u_i \in C^{4,\zeta}(B_1^*)$, for some $\zeta\in(0,1)$, and solves \eqref{oursystem} in the classical sense.
	We say that $\mathcal{U}$ is a {\it singular solution} to \eqref{oursystem}, if the origin is a {\it non-removable singularity} for $|\mathcal{U}|$, that is, $\lim_{|x|\rightarrow 0}|\mathcal{U}(x)|=\infty$; otherwise, it is called a {\it removable singularity} and $\mathcal{U}$ is a {\it non-singular solution} of \eqref{oursystem}. 
	We also call a $p$-map $\mathcal{U}$ solution {\it strongly positive} ({\it nonnegative}) when $u_i>0$ ($u_i\geqslant0$), and $\mathcal{U}$ is {\it superharmonic} in case $-\Delta u_i>0$ for all $i\in I:=\{1,\dots,p\}$.
	Notice that in general, by the maximum principle, superharmonic nonnegative solutions are {\it weakly positive}, that is, for any $i\in I$ either $u_i>0$ or $u_i\equiv 0$. 
	Moreover, by \cite[Corollary~50]{arXiv:2002.12491}, one has that the limit solutions are strongly positive.
	
	Our main result proves that solutions to \eqref{oursystem} have a local asymptotic profile near the isolated singularity given by an Emden--Fowler solution, that is, satisfying the blow-up limit system, 
	\begin{equation}\label{vectlimitequation}
	\Delta^{2} u_{i}=c(n)|\mathcal{U}|^{2^{**}-2}u_{i} \quad {\rm in} \quad \mathbb{R}^n\setminus\{0\}.
	\end{equation}
	
	\begin{theorem}\label{theorem1}
		Let $\mathcal{U}$ be a strongly positive singular solution to \eqref{oursystem}. 
		Assume that $\mathcal{U}$ is superharmonic.
		Then, there exist an Emden--Fowler solution $\mathcal{U}_{a,T}$ to \eqref{vectlimitequation} and $0<\beta^*_0<1$ such that 
		\begin{equation}\label{asymptotics}
		\mathcal{U}(x)=(1+\mathcal{O}(|x|^{\beta^*_0}))\mathcal{U}_{a,T}(|x|) \quad {\rm as} \quad x\rightarrow0.
		\end{equation}
	\end{theorem}
	
	The last convergence result combined with the arguments in \cite[Section~7]{MR1666838}, one can improve the decay of the remainder term in \eqref{asymptotics}, using deformed Emden--Fowler solutions (see Definition \ref{def:deformedfowler}).
	In this spirit, another version of the last theorem is the following refined asymptotics:
	
	\begin{theoremtio}\label{theorem1'}
		{\it Let $\mathcal{U}$ be a strongly positive singular solution to \eqref{oursystem}. 
			Assume that $\mathcal{U}$ is superharmonic.
			Then, there exist a deformed Emden--Fowler solution $\mathcal{U}_{a,T,0}$ to \eqref{vectlimitequation} and $\beta^*_1>1$ such that}
		\begin{equation*}
		\mathcal{U}(x)=(1+\mathcal{O}(|x|^{\beta^*_1}))\mathcal{U}_{a,T,0}(|x|) \quad {\rm as} \quad x\rightarrow0.
		\end{equation*}
	\end{theoremtio}
	
	Notice that when $p=1$, system \eqref{oursystem} reduces to the following fourth order critical equation,
	\begin{equation}\label{ourequation}\tag{$\mathcal{S}_1$}
	\Delta^2 u=c(n)u^{2^{**}-1} \quad {\rm in} \quad B_1^*.
	\end{equation}
	Consequently, the local models in the neighborhood of the isolated singularity are given by the solutions to the limit blow-up equation
	\begin{equation}\label{limitequation}
	\Delta^2 u=c(n)u^{2^{**}-1} \quad {\rm in} \quad \mathbb{R}^n\setminus\{0\}.
	\end{equation} 
	More accurately, we have the scalar version of Theorems~\ref{theorem1} and \ref{theorem1'}\textcolor{blue}{'}
	\begin{corollary}\label{theorem2}
		Let $u$ be a positive superharmonic singular solution to \eqref{ourequation}. Then, there exist an Emden--Fowler solution $u_{a,T}$ to \eqref{limitequation} and $\beta^*_0>0$ such that 
		\begin{equation*}
		u(x)=(1+\mathcal{O}(|x|^{\beta^*_0}))u_{a,T}(|x|) \quad {\rm as} \quad x\rightarrow0.
		\end{equation*}
		Moreover, one can find a deformed Emden--Fowler solution $u_{a,T,0}$ and $\beta^*_1>0$ such that 
		\begin{equation*}
		u(x)=(1+\mathcal{O}(|x|^{\beta^*_1}))u_{a,T}(|x|) \quad {\rm as} \quad x\rightarrow0.
		\end{equation*}
	\end{corollary}
	
	\begin{remark}
		There are a few observations we would like to quote:\\
		\noindent{\rm (i)} Theorems~\ref{theorem1} and \ref{theorem1'}\textcolor{blue}{'} are the first step in describing the local asymptotic behavior of singular solutions to more general fourth order systems in the inhomogeneous setting $($see \eqref{geometricsystem}$)$.\\
		\noindent{\rm (ii)} The higher asymptotic expansion recently proved in \cite{arXiv:1909.10131v1,arXiv:1909.07466v1} is believed to be true in our fourth order setting $($see \eqref{higherorderasymptotics}$)$. This refined asymptotic behavior is the first step to understand the moduli space of the $Q$-curvature equation on $\mathbb{S}^{n-1}\setminus \mathcal{Z}$, where $\mathcal{Z}$ is a finite set \cite{MR1712628}, which can be done by employing a gluing construction.
		As well as, it is possible to prove existence results for singular solutions on $\mathbb{S}^{n-1}\setminus \mathcal{Z}$, where $\mathcal{Z}$, even when the prescribed singular set $\mathcal{Z}$ has Hausdorff dimension strictly larger than zero \cite{arXiv:1911.11891} and less than a threshold, depending only on the dimension.\\
		\noindent{\rm (iii)} The same arguments used to prove the asymptotics in Theorem~\ref{theorem1} also implies the existence of singular solutions to \eqref{oursystem}. In fact, this follows as a by-product of our linear analysis and the implicit function theorem. 
	\end{remark}
	
	\begin{remark}
		The additional superharmonicity condition $-\Delta u_i>0$ for all $i\in I$ cannot be dropped. 
		Indeed, it has a geometric motivation involving the sign of the scalar curvature of the background metric \cite{MR3420504,MR3518237}.
		In the blow-up limit case \eqref{limitequation}, the Emden--Fowler solution satisfies this assumption automatically, for more details, see \cite[Theorem~2.1]{MR1769247}, \cite[Theorem~1.4]{MR4094467}, and \cite[Propositions~11 and 29]{arXiv:2002.12491}. 
		For local solutions, it is not possible to prove this superharmonicity property by using the same methods.
		This hypothesis is also important to start the sliding technique used to prove the lower bound estimate.
	\end{remark}
	
	Since our techniques rely on blow-up analysis, the first step to obtain the asymptotic behavior near the singularity is to classify the solutions to the limit equation, in both non-singular case and singular case. 
	When $p=1$, we summarize the classification obtained by C. S. Lin \cite[Theorem~1.5]{MR1611691} and R. L. Frank and T. K\"onig \cite[Theorem~2]{MR3869387} (see also \cite[Theorem~1.3]{MR4094467}), 
	
	\begin{theoremletter}\label{thm:frank-konig19}
		Let $u$ be a positive solution to \eqref{limitequation} and assume that\\
		\noindent{\rm (i)} the origin is a removable singularity. Then, there exist $x_0\in\mathbb{R}^n$ and $\mu>0$ such that $u$ is radially symmetric about $x_0$ and 
		\begin{equation}\label{sphericalfunctions}
		u_{x_0,\mu}(x)=\left(\frac{2\mu}{1+\mu^{2}|x-x_0|^{2}}\right)^{\frac{n-4}{2}}. 
		\end{equation}
		We call $u_{x_0,\mu}$ a fourth order spherical solution.
		
		\noindent{\rm (ii)} the origin is a non-removable singularity. Then, $u$ is radially symmetric about the origin. Moreover, there exist $a \in (0,a_0]$ and $T\in (0,T_a]$ such that
		\begin{equation}\label{emden-folwersolution}
		u_{a,T}(x)=|x|^{\frac{4-n}{2}}v_{a}(-\ln|x|+T).
		\end{equation}
		Here $a_0=[n(n-4)/(n^2-4)]^{n-4/8}$, $T_a\in\mathbb{R}$ is the fundamental period of the unique $T$-periodic bounded solution $v_a$ to the following fourth order Cauchy problem, 
		\begin{equation}\label{fowler4order}
		\begin{cases}
		v^{(4)}-K_2v^{(2)}+K_0v=c(n)v^{2^{**}-1}\\
		v(0)=a,\ v^{(1)}(0)=0,\ v^{(2)}(0)=b(a),\ v^{(3)}(0)=0,
		\end{cases}
		\end{equation}
		where $K_2,K_0$ are constants depending only on the dimension and $b(a)$ is the critical shooting parameter for any fixed $a \in (0,a_0]$. We call both $u_{a,T}$ and $v_{a,T}$ Emden--Fowler $($or Delaunay-type$)$ solutions and $a\in(0,a_0)$ its Fowler parameter, which can be chosen satisfying $a=\min_{t>0}v_a(t)$.
	\end{theoremletter}
	
	\begin{remark} 
		The cases $a=0$ and $a=a_0$ are called the limit cases. 
		Observe that these values correspond to the equilibrium solutions to \eqref{fowler4order}. 
		When $a=0$ gives rise to the {\it spherical solution} $v_{\rm sph}(t)=(\cosh (t-t_0))^{(n-4)/2}$, for some $t_0\in\mathbb{R}$, whereas for $a=a_0$, $v_{\rm cyl}(t)=a_0$ is called the {\it cylindrical solution}.
	\end{remark}
	
	On the limit blow-up case for $p>1$, the present authors in \cite[Theorems~1 and 2]{arXiv:2002.12491} used sliding techniques and ODE analysis to classify the solutions to \eqref{vectlimitequation}, this results are summarized as follows
	\begin{theoremletter}\label{thm:andrade-doo19}
		Let $\mathcal{U}$ be a nonnegative solution to \eqref{vectlimitequation} and assume that\\	
		\noindent{\rm (i)} the origin is a removable singularity. Then, $\mathcal{U}$ is weakly positive and radially symmetric about some $x_0\in\mathbb{R}^n$. Moreover, there exist $\Lambda \in\mathbb{S}^{p-1}_{+}=\{ x \in \mathbb{S}^{p-1} : x_i \geqslant 0 \}$ and a fourth order spherical solution given by \eqref{sphericalfunctions} such that
		\begin{equation*}
		\mathcal{U}=\Lambda u_{x_0,\mu}.
		\end{equation*}
		\noindent{\rm (ii)} the origin is a non-removable singularity. Then, $\mathcal{U}$ is strongly positive, radially symmetric about the origin and decreasing. Moreover, there exist $\Lambda^*\in\mathbb{S}^{p-1}_{+,*}=\{ x \in \mathbb{S}^{p-1} : x_i > 0 \}$ and an Emden--Fowler solution given by \eqref{emden-folwersolution} such that 
		\begin{equation*}
		\mathcal{U}=\Lambda^* u_{a,T}.
		\end{equation*}
	\end{theoremletter}
	
	For the geometric point of view, the works of R. Schoen and S.-T. Yau \cite{MR931204,MR929283} highlighted the importance of studying singular equations and describing their asymptotic behavior near their singular sets. 
	Indeed, a positive solution $u$ to \eqref{ourequation} produces a conformally flat metric $\bar{g}=u^{4/(n-4)}\delta_0$ such that $\bar{g}$ has constant $Q$-curvature equals $Q_g=n(n^2-4)/8$, where $\delta_0$ is the standard flat metric. Notice that \eqref{ourequation}
	is a particular case of a more general equation appearing in conformal geometry, namely the $Q$-curvature equation on a Riemannian manifold $(M^n,g)$, where $g$ is a smooth metric on $B_1$,
	\begin{equation}\label{qcurvature}
	P_g=c(n)u^{2^{**}-1} \quad {\rm in} \quad (B_1^*,g).
	\end{equation}
	Here 
	\begin{equation*}
	P_g=\Delta_{g}^{2} u-\dive(a_nR_{g}g-b_n\Ric_g)\ud u+\frac{n-4}{2}Q_gu
	\end{equation*}
	is the {\it Paneitz--Brason operator} (for more details, see \cite{MR2393291,MR904819,MR832360}), where
	\begin{equation*}
	Q_g=-\frac{1}{2(n-1)}\Delta R_g-\frac{2}{(n-2)^2}|\Ric_g|^2
	+\frac{n^3-4n^2+16n-16}{8(n-1)^2(n-2)^2}R_g^2,
	\end{equation*}
	and
	\begin{equation*}
	a_n=\frac{(n-2)^2+4}{2(n-1)(n-2)} \quad {\rm and} \quad b_n=\frac{4}{n-2}.
	\end{equation*}
	We emphasize that \eqref{qcurvature} naturally appears as the deformation law for the $Q$-curvature of two conformal metrics on a Riemannian manifold \cite{MR837196}. This operator satisfies the same conformal invariance enjoyed by $L_g$ in the case of the Yamabe equation. 
	In this fashion, the quantity $Q_g$ plays the same role played by the scalar curvature $R_g$ in the conformal Laplacian $L_g$, for this reason this quantity is called the {\it $Q$-curvature} \cite{MR3618119,MR2104700,MR2149088,MR3953754}. 
	To fix some notation, let us the following nonlinear geometric operator,
	\begin{equation*}
	\mathcal{N}_{g}(u)=P_gu-c(n)u^{2^{**}-1}.
	\end{equation*}
	
	\begin{remark}
		The geometrical interpretation for Theorem~\ref{theorem1} is summarized using the conformal metric $\bar{g}=u_{a,T}^{4/(n-4)}\delta_0$  as follows:\\
		\noindent{\rm (i)} If $a=0$, then $\bar{g}=g_{\rm sph}$ extends to the whole $B_1$ as the round metric;\\
		\noindent{\rm (ii)} If $a=a_0$, then $\bar{g}=g_{\rm cyl}$ is the cylindrical metric given by $g_{\rm cyl}=\ud t^2+\ud\theta^2$;\\
		\noindent{\rm (iii)} If $a\in(0,a_0)$, then $\bar{g}=g_{a,T}$ is the Emden--Fowler metric, interpolating between the spherical and the cylindrical metrics.\\
		In fact, $\lim_{a\rightarrow0}T_{a}=\infty$, in other words $v_{0}=v_{\rm sph}$ is the limit of $v_{a}(t)=\left(t+T_{a}\right)$ as $a\rightarrow0$.
		In the language of conformal geometry, these metrics represent the intermediate states of the evolution from the cylindrical metric to the limit singular metric, which is given by a bead of spheres along an axis.
	\end{remark}
	
	It is a general principle in nonlinear analysis that the existence and local asymptotic behavior for singular PDEs are strongly related to the spectral analysis of the linearized operator around a limit blow-up solution \cite{MR1356375,MR194163}.
	The set of techniques used to study the local behavior for singular solutions to PDEs are sometimes called {\it asymptotic analysis}; this has been a topic of intense study in recent years, motivated by some problems arising in conformal geometry.
	For instance, on the context of the fractional Laplacian \cite{MR3198648,MR3694655,arXiv:1804.00817} (see also \cite{MR2200258} for more general integral equations). 
	In addition, we quote the results in \cite{MR2165306,MR2214582,MR2737708} dealing with fully nonlinear operators related by the $\sigma_k$-curvature problem. We also mention the results in \cite{MR1436822,MR4123335} concerning the subcritical equation associated to \eqref{ourequation}, extending the results in \cite{MR605060,MR615628,MR875297}  to the fourth order setting. 
	In the case of strongly coupled systems, we refer to the works of \cite{arXiv:2002.12491,MR4002167}. 
	Finally, the three-dimensional \cite{MR4029726,MR3669775,MR3465087,MR3962197}, and four-dimensional \cite{MR1611691,MR3604948,MR3615163,arXiv:1804.09261v2,hyder-mancini-martinazzi} cases would be also worth studying (see Section~\ref{problemssection}).
	
	Now we compare our results with the existing literature on the second order case.
	The singular Yamabe problem, reduces to the study of a second order geometric PDE that is similar in spirit to \eqref{qcurvature},
	\begin{equation}\label{yamabeeq}
	L_g u=u^{2^{*}-1}, \quad \mbox{where} \quad L_g=-\Delta_g+\frac{n-2}{4(n-1)}R_g
	\end{equation}
	is the conformal Laplacian, and $2^{*}=2n/(n-2)$ for $n\geqslant3$. In the conformally flat case, we have
	\begin{equation}\label{secondorderequation}
	-\Delta u=u^{2^{*}-1} \quad \mbox{in} \quad B_1^*.
	\end{equation}
	In a milestone paper, L. A. Caffarelli et al. \cite[Theorem~1.2]{MR982351} developed a measure-theoretic version of the Alexandrov technique to prove that solutions to \eqref{secondorderequation} defined in the punctured ball are radially symmetry. Moreover, they provided a classification for global singular solutions to \eqref{secondorderequation} in the punctured space and obtained the local behavior in the neighborhood of the isolated singularity, proving that any singular solution converges to an Emden--Fowler solution \cite{lane,emden,fowler,MR0092663}.
	
	Later, N. Korevaar et al. \cite[Theorem~1]{MR1666838} gave a more geometric approach for proving \eqref{asymptotics}, which is based on the growth of the Jacobi fields for the linearized operator around an blow-up limit solution. In fact, they proved the following asymptotics,
	\begin{theoremletter}\label{thm:korevaar-mazzeo-pacard-schoen99}
		Let $u$ be a positive solution to \eqref{secondorderequation}. Then, there exist an Emden--Fowler solution $u_{a,T}$ to the blow-up limit of \eqref{secondorderequation} and $\beta^*_0>0$ such that 
		\begin{equation*}
		u(x)=(1+\mathcal{O}(|x|^{\beta^*_0}))u_{a,T}(|x|) \quad {\rm as} \quad x\rightarrow0.
		\end{equation*}
		Moreover, one can find a deformed Emden--Fowler solution $u_{a,0,x_0}$ and $\beta^*_1>1$ satisfying 
		\begin{equation*}
		u(x)=(1+\mathcal{O}(|x|^{\beta^*_1}))u_{a,0,x_0}(|x|) \quad {\rm as} \quad x\rightarrow0.
		\end{equation*}
	\end{theoremletter}
	
	We should mention that F. C. Marques \cite[Theorem~1]{MR2393072} extended this asymptotics for the inhomogeneous case, at least for lower dimensions. 
	Namely, he showed that for non-flat background metrics the Emden--Fowler solutions are still the local models near the origin.
	On second order strongly coupled systems, R. Caju et al. \cite{MR4002167} (see also \cite{MR4085120}) classified the limit blow-up $p$-map solutions $\mathcal{U}=(u_1,\dots,u_p): B_1^*\rightarrow \mathbb{R}^p$ to the following second order analog to \eqref{oursystem},
	\begin{equation}\label{secondordersystem}
	-\Delta u_{i}=\frac{n(n-2)}{4}|\mathcal{U}|^{2^{*}-2}u_{i} \quad {\rm in} \quad B_1^* \quad {\rm for} \quad i\in I.
	\end{equation}
	They obtained an asymptotic classification, which is similar in spirit to Theorem~\ref{thm:andrade-doo19}. 
	Moreover, the same results were proved for non-flat background metrics with the same dimensions restrictions found in \cite[Theorem~1]{MR2393072}.
	\begin{theoremletter}\label{thm:caju-doo-santos19}
		Let $\mathcal{U}$ be a nonnegative singular solution to \eqref{secondordersystem}. Then, there exists $\beta^*_0>0$ such that 
		\begin{equation*}
		\mathcal{U}(x)=(1+\mathcal{O}(|x|^{\beta^*_0}))\mathcal{U}_{a,T}(|x|) \quad {\rm as} \quad x\rightarrow0,
		\end{equation*}
		where $\mathcal{U}_{a,T}$ satisfies the limit blow-up system,
		\begin{equation}\label{secondorderlimitsystem}
		-\Delta u_{i}=\frac{n(n-2)}{4}|\mathcal{U}|^{2^{*}-2}u_{i} \quad {\rm in} \quad \mathbb{R}^n\setminus\{0\}.
		\end{equation}
	\end{theoremletter}
	This theorem is the vectorial case of some classical asymptotic results due to L. A. Caffarelli et al. \cite[Theorem~1.2]{MR982351} and N. Korevaar et al. \cite[Theorem~1]{MR1666838}. 
	On the classification for non-singular solutions to \eqref{secondorderlimitsystem}, we refer the reader to \cite{MR2603801,MR2558186}.
	
	On the scalar case of \eqref{oursystem}, T. Jin and J. Xiong \cite[Theorems~1.1 and 1.2]{arxiv:1901.01678} used a Green identity for the poly-Laplacian and some localization methods to study an integral equation equivalent to \eqref{ourequation}, proving asymptotic radial symmetry and sharp global estimates for singular solutions to \eqref{ourequation}, which can be stated as follows
	\begin{theoremletter}\label{thm:kin-xiong19} 
		Let $u$ be a positive superharmonic solution to \eqref{ourequation}. Then, there exist $C_1,C_2>0$ such that 
		\begin{equation}\label{sharpasymp}
		C_1|x|^{\frac{4-n}{2}}\leqslant u(x)\leqslant C_2|x|^{\frac{4-n}{2}}.
		\end{equation}
		Moreover, $u$ is asymptotically radially symmetric and 
		\begin{equation*}
		u(x)=(1+\mathcal{O}(|x|))\overline{u}(x) \quad {\rm as} \quad x\rightarrow0,
		\end{equation*}
		where $\overline{u}(r)=\avint_{\partial B_{1}} u(r \theta)\ud\theta$ is the spherical average of $u$.
	\end{theoremletter}
	
	\begin{remark}
		The convergence to the average in Theorem \ref{thm:kin-xiong19} is also true for higher order equations
		\begin{equation}\label{higherequation}
		(-\Delta)^k u=u^{\frac{n+2k}{n-2k}} \quad {\rm in} \quad B_1^*,
		\end{equation}
		where $n\geqslant 2k$, $k\in\mathbb{N}$ and $(-\Delta)^k$ denotes the poly-Laplacian operator, with additional positivity condition $(-\Delta)^j u>0$ for all $j=1,\dots,k-1$. 
		We also emphasize that when $k$ is odd, \eqref{higherequation} becomes an integral equation, and a more involved analysis is required.
	\end{remark}
	
	The strategy to prove Theorem~\ref{theorem1} relies on asymptotic analysis.  
	Roughly speaking, this is a combination of classification results, some a priori estimates, and linear analysis.
	The first step is to show that the Jacobi fields (these are the elements in the kernel of the linearization of \eqref{oursystem} around an Emden--Fowler limit solution) satisfy suitable growth properties
	
	\begin{proposition}\label{growthpropertiessystem}
		The following properties hold for the projected linearized operator (see Lemma~\ref{linearizedoperator}):\\
		\noindent{\rm (i)} For $j=0$, the homogeneous equation $\mathcal{L}^{a}_{0}(\Phi)=0$ has a solutions basis with $2p$ elements, which are either bounded or at most linearly growing as $t\rightarrow\infty$;\\
		\noindent{\rm (ii)} For each $j\geqslant1$, the homogeneous equation $\mathcal{L}^{a}_{j}(\Phi)=0$ has a solutions basis with $4p$ elements, which are exponentially growing/decaying as $t\rightarrow\infty$.
	\end{proposition}
	Inspired by \cite{MR4002167,MR2737708}, we use the spectral analysis of the linearized operator to prove the last proposition.
	The issue is that not all the Jacobi fields are generated by variations of some parameters in the classification of the Emden--Fowler solutions. 
	To overcome this problem, we show that the spectrum of the linearized operator is purely absolutely continuous.
	More precisely, it is the union of spectral bands separated by gaps away from the origin.
	Therefore, the zero frequency deficiency space is generated only by the geometric Jacobi fields.
	We should also mention that fourth order Jacobi fields likewise in Proposition~\ref{growthpropertiessystem} also appear in the study of stability and non-degeneracy properties for Willmore hypersurfaces \cite{MR2785762,MR3780139,MR3612543,MR2119722}.
	
	In addition, we also need to show that solutions to \eqref{oursystem} are radially symmetric, and satisfy upper and lower bounds estimate near the isolated singularity, which can be enunciated as
	\begin{proposition}\label{estimates}
		Let $\mathcal{U}$ be a strongly positive superharmonic solution to \eqref{oursystem}. 
		Then, $\mathcal{U}$ is radially symmetric about the origin. Moreover, either the origin is a removable singularity or there exist $C_1,C_2>0$ satisfying 
		\begin{equation}\label{estimatesorigin}
		C_1|x|^{\frac{4-n}{2}}\leqslant |\mathcal{U}(x)|\leqslant C_2|x|^{\frac{4-n}{2}} \quad {\rm for} \quad 0<|x|<1/2.
		\end{equation}
	\end{proposition}
	
	The main ingredients in the proof of Proposition~\ref{estimates} are the blow-up method based on the classification result for non-singular solutions to \eqref{vectlimitequation} in Theorem~\ref{thm:frank-konig19}, and a removable singularity theorem relying on the sign of the Pohozaev invariant associated to \eqref{oursystem}. 
	The difficulties in our argument are numerous. 
	The lack of maximum principle causes one of those due to the fourth order operator on the left-hand side of \eqref{oursystem}. 
	To handle the problem with the lack of maximum principle, we apply a Green identity to convert \eqref{oursystem} into an integral system \cite{arxiv:1901.01678} (see also \cite{MR3694645,MR3626036,arXiv:1810.02752v6}). 
	Then, we prove that singular solutions satisfy an upper and lower bound near the isolated singularity; these arguments are based on an integral form of the moving spheres technique. 
	We also need to deal with the nonlinear effects imposed by the coupling term on the right-hand side of \eqref{oursystem}.
	This idea is to use Theorem~\ref{thm:andrade-doo19} combined with some decoupling techniques from \cite{MR4085120,MR2558186,MR2603801,MR2237439}, which allows a comparison involving the norm of a $p$-map solution and each component. 
	
	The next proposition is a reformulation of Theorem~\ref{theorem1} in terms of the a priori estimates, which proof is a combination of Propositions~\ref{growthpropertiessystem} and \ref{estimates},
	\begin{proposition}\label{convergence}
		Let $\mathcal{U}$ be a strongly positive superharmonic solution to \eqref{oursystem} satisfying \eqref{estimatesorigin}.
		Then, there exist $\beta^*_0>0$ and an Emden--Fowler solution $\mathcal{U}_{a,T}$ such that
		\begin{equation*}
		\mathcal{U}(x)=(1+\mathcal{O}(|x|^{\beta^*_0}))\mathcal{U}_{a,T}(|x|) \quad {\rm as} \quad x\rightarrow0.
		\end{equation*}
	\end{proposition}
	The proof of Proposition~\ref{convergence} follows by Simon's technique \cite{MR758452,MR963501} (or slide back technique). 
	This method is based on the classification theorem for singular global solutions to \eqref{vectlimitequation} and exemplifies the deep connections between the fields of geometric analysis and nonlinear PDEs. 
	In the context of geometry, it arises in the study of properties for the Jacobi operator of a minimal hypersurface. 
	It is not surprising that, for problems arising in geometry, the low-frequency Jacobi fields of the linearized operator have an explicitly geometric interpretation \cite{MR1371233,MR1147718,MR1010168}. 
	Indeed, exploring this property is the key part of method, since for this geometric Jacobi  fields refined estimates are available.
	Let us also remark that the exponents $\beta^*_0,\beta^*_1,\dots$ are related to the indicial roots of the linearized operator.
	More precisely, $\beta^*_0$ depends on the decaying rate of the Jacobi fields and on the period of the one-dimensional limit solution.
	
	Besides their applications in conformal geometry, strongly coupled fourth-order systems also appear in several mathematical physics branches. 
	For instance, in hydrodynamics, for modeling the behavior of deep-water and Rogue waves in the ocean \cite{dysthe,lo-mei}. 
	It also gleams in the Hartree--Fock theory for Bose--Einstein double condensates \cite{MR2040621,PhysRevLett.78.3594}. 
	
	Here the division for the rest of the paper. 
	In Section~\ref{sec:preliminaries}, we introduce the cylindrical transformation, the Pohozaev-type invariant, and the linearized operator around an Emden--Fowler solution. 
	In section~\ref{sec:linearanalysis}, we use linear analysis to prove Proposition~\ref{growthpropertiessystem}.
	As a consequence, we use Fredholm theory to prove existence of solutions to \eqref{oursystem}.
	In Section \ref{sec:qualitativeproperties}, using the asymptotic symmetry and an upper bound estimate based on a moving sphere method, we prove Proposition~\ref{estimates}.
	Furthermore, we obtain a removable singularity theorem, which implies the lower bound estimate. In Section \ref{sec:convergence}, we apply Simon's technique to proof Proposition~\ref{convergence}, and consequently Theorem \ref{theorem1}. 
	Whence, we obtain a refined asymptotics involving the deformed Emden--Fowler solutions.
	Finally, we complete this manuscript with some final considerations concerning possible extension for Theorems~\ref{theorem1} and \ref{theorem1'}\textcolor{blue}{'}.

	\section{Preliminaries}\label{sec:preliminaries}
	This section aims to introduce some necessary background for developing our asymptotic analysis methods.
	
	\subsection{Kelvin transform}\label{sec:kelvintransform}
	The moving spheres technique we will use later is based on the {\it fourth order Kelvin transform} for a $p$-map. 
	
	For $\Omega\subseteq\mathbb{R}^n$ a domain, before we define the Kelvin transform, we need to establish the concept of {\it inversion through a sphere} $\partial B_{\mu}(x_0)$, which is a map $\mathcal{I}_{x_0,\mu}:\Omega\rightarrow\Omega_{x_0,\mu}$ given by $\mathcal{I}_{x_0,\mu}(x)=x_0+K_{x_0,\mu}(x)^2(x-x_0)$, where $K_{x_0,\mu}(x)=\mu/|x-x_0|$ and $\Omega_{x_0,\mu}:=\mathcal{I}_{x_0,\mu}(\Omega)$ is the domain of the Kelvin transform (for more details, see \cite[Section~2.7]{arXiv:2002.12491}).
	In particular, when $x_0=0$ and $\mu=1$, we denote it simply by $\mathcal{I}_{0,1}(x)=x^{*}$ and $K_{0,1}(x)=x|x|^{-2}$. For easy reference, let us summarize some well-known facts about the inversion map.
	\begin{proposition}
		The map $\mathcal{I}_{x_0,\mu}$ has the properties:\\
		\noindent {\rm (i)} It maps $B_{\mu}(x_0)$ into its complement $\Omega\setminus \bar{B}_{\mu}(x_0)$, such as $x_0$ into $\infty$;\\
		\noindent {\rm (ii)} It let the boundary $\partial B_{\mu}(x_0)$ invariant, that is, $\mathcal{I}_{x_0,\mu}(\partial B_{\mu}(x_0))=\partial B_{\mu}(x_0)$.
	\end{proposition}
	
	The following definition is a generalization of the Kelvin transform for fourth order operators applied to $p$-maps.
	
	\begin{definition}
		For any $\mathcal{U}\in C^4(\Omega,\mathbb{R}^p)$, let us consider the fourth order Kelvin transform through the sphere with center at $x_0\in\mathbb{R}^n$ and radius $\mu>0$  defined on $\mathcal{U}_{x_0,\mu}:\Omega_{x_0,\mu}\rightarrow\mathbb{R}^p$ by 
		\begin{equation*}
		\mathcal{U}_{x_0,\mu}(x)=K_{x_0,\mu}(x)^{n-4}\mathcal{U}\left(\mathcal{I}_{x_0,\mu}(x)\right).
		\end{equation*}
	\end{definition}
	
	\begin{proposition}
		System \eqref{oursystem} is invariant under the action of Kelvin transform, {\it i.e.}, if $\mathcal{U}$ is a non-singular solution to \eqref{oursystem}, then  $\mathcal{U}_{x_0,\mu}$ is a solution to 
		\begin{equation*}
		\Delta^{2}(u_{i})_{x_0,\mu}=c(n)|\mathcal{U}_{x_0,\mu}|^{2^{**}-2}(u_{i})_{x_0,\mu} \quad {\rm in} \quad (B_1^*)_{x_0,\mu} \quad {\rm for} \quad i\in I,
		\end{equation*}
		where $\mathcal{U}_{x_0,\mu}=((u_1)_{x_0,\mu},\dots,(u_p)_{x_0,\mu})$.
	\end{proposition}
	
	\begin{proof}
		It is a direct consequence of the formula
		\begin{equation*}
		\Delta^2 u_{x_0,\mu}(x)=K_{x_0,\mu}(x)^{n+4}\Delta^{2}u\left(\mathcal{I}_{x_0,\mu}(x)\right)=K_{x_0,\mu}(x)^{8}(\Delta^2 u)_{x_0,\mu}(x),
		\end{equation*}
		which is obtained by a simple computation (see \cite[Appendix~A]{arXiv:2002.12491}).
	\end{proof}
	
	\subsection{Cylindrical transformation} 
	This subsection is devoted to convert singular solutions to \eqref{oursystem} into non-singular solutions in a cylinder. 
	Then, the local behavior of singular solutions to \eqref{oursystem} near the origin reduces to understand the asymptotic global behavior for tempered solutions to a fourth order ODE defined on a cylinder. 
	More accurately, they are distributional solutions where the test functions are taken in the Schwartz space (the function space of all infinitely differentiable functions that are rapidly decreasing as $t\rightarrow0$ along with all their partial derivatives); the elements in this space also have a well-defined Fourier transform. 
	Let us introduce the so-called {\it cylindrical transformation} (see also \cite{arXiv:2002.12491}). 
	First, let us consider ${\mathcal{C}}_{0,1}=(0,1)\times\mathbb{S}^{n-1}$ the cylinder and $\Delta^2_{\rm sph}$ the bi-Laplacian written in spherical (polar) coordinates,
	\begin{align*}
	\Delta^2_{\rm sph}&=\partial_r^{(4)}+ \frac{2(n-1)}{r}\partial_r^{(3)}+\frac{(n-1)(n-3)}{r^2}\partial_r^{(2)}-\frac{(n-1)(n-3)}{r^3}\partial_r&\\\nonumber
	&+\frac{1}{r^4}\Delta_{\sigma}^2+\frac{2}{r^2}\partial^{(2)}_r\Delta_{\sigma}+\frac{2(n-3)}{r^3}\partial_r\Delta_{\sigma}-\frac{2(n-4)}{r^4}\Delta_{\sigma},&
	\end{align*}
	where $\Delta_{\sigma}$ denotes the Laplace--Beltrami operator on $\mathbb{S}^{n-1}$.
	Then, we can rewrite \eqref{oursystem} as the nonautonomous nonlinear equation,
	\begin{equation*}
	\Delta^2_{\rm sph}u_i=c(n)|\mathcal{U}|^{2^{**}-2}u_i \quad {\rm in} \quad {\mathcal{C}}_{0,1}.
	\end{equation*}
	In addition, we apply the Emden--Fowler change of variables (or logarithm cylindrical coordinates) given by $\mathcal{V}(t,\theta)=r^{\gamma}\mathcal{U}(r,\sigma)$, where $r=|x|$, $t=-\ln r$, $\theta=x/|x|$ and $\gamma=({n-4})/{2}$, which sends the problem to the (half-)cylinder $\mathcal{C}_0=(0,\infty)\times \mathbb{S}^{n-1}$.  
	Using this coordinate system and performing a lengthy computation (see \cite{MR4094467,MR4123335}), we arrive at the following fourth order nonlinear PDE system on the cylinder,
	\begin{equation}\label{sphevectfowler}
	\Delta^2_{\rm cyl}v_i=c(n)|\mathcal{V}|^{2^{**}-2}v_i \quad {\rm on} \quad {\mathcal{C}}_0.
	\end{equation}
	Here $\Delta^2_{\rm cyl}$ is the bi-Laplacian written in cylindrical coordinates given by
	\begin{equation}\label{cylbi-Laplacian}
	\Delta^2_{\rm cyl}=\partial_t^{(4)}-K_2\partial_t^{(2)}+K_0+\Delta_{\theta}^{2}+2\partial^{(2)}_t\Delta_{\theta}-J_0\Delta_{\theta},
	\end{equation}
	where $K_0,K_2,J_0$ are constants depending only in the dimension, which are defined by
	\begin{equation*}
	K_0=\frac{n^2(n-4)^2}{16}, \quad K_2=\frac{n^2-4n+8}{2} \quad {\rm and} \quad J_0=\frac{n(n-4)}{2}.
	\end{equation*}
	Furthermore, the superharmonicity condition $-\Delta u_i>0$ is equivalent to
	\begin{equation*}
	-\partial_t^{(2)}v_i+2\partial_t v_i+\sqrt{K_0}-\Delta_{\theta}v_i>0.
	\end{equation*}
	Along this lines, let us introduce the cylindrical transformation defined as follows
	\begin{equation*}
	\mathfrak{F}:C_c^{\infty}(B_1^*,\mathbb{R}^p)\rightarrow C_c^{\infty}(\mathcal{C}_0,\mathbb{R}^p) \quad \mbox{given by} \quad
	\mathfrak{F}(\mathcal{U})=r^{\gamma}\mathcal{U}(r,\sigma).
	\end{equation*}
	
	\begin{remark}
		The transformation $\mathfrak{F}$ is a continuous bijection with respect to the Sobolev norms $\|\cdot\|_{\mathcal{D}^{2,2}(B_1^*,\mathbb{R}^p)}$ and $\|\cdot\|_{H^{2}(\mathcal{C}_0,\mathbb{R}^p)}$, respectively. 
		Furthermore, this transformation sends singular solutions to \eqref{ourequation} into solutions to \eqref{sphevectfowler}, and by density, we get that $\mathfrak{F}:\mathcal{D}^{2,2}(B_1^*,\mathbb{R}^p)\rightarrow H^{2}(\mathcal{C}_0,\mathbb{R}^p)$.
	\end{remark}
	
	In the geometric language, this change of variables corresponds to a restriction of the conformal diffeomorphism between the entire cylinder $\mathcal{C}_{\infty}:=\mathbb{R}\times\mathbb{S}^{n-1}$ and the punctured space, namely,  $\varphi:(\mathcal{C}_{\infty},g_{{\rm cyl}})\rightarrow(\mathbb{R}^n\setminus\{0\},\delta_0)$ defined by $\varphi(t,\sigma)=e^{-t}\sigma$. 
	Here $g_{{\rm cyl}}=\ud t^2+\ud\sigma^2$ stands for the cylindrical metric and ${\ud\theta}=e^{-2t}(\ud t^2+\ud\sigma^2)$ for its volume element obtained via the pullback $\varphi^{*}\delta_0$, where $\delta_0$ is the standard flat metric.
	Notice that, in the blow-up limit case, solutions are rotationally symmetric, which transform \eqref{sphevectfowler} into the following ODE system
	\begin{equation}\label{fowlersystem}
	v_i^{(4)}-K_2v_i^{(2)}+K_0v_i=c(n)|\mathcal{V}|^{2^{**}-2}v_i \quad {\rm in} \quad \mathbb{R},
	\end{equation}
	which can be taken with suitable initial conditions in order to become a well-posed Cauchy problem.
	
	\begin{remark}
		Although we know that asymptotic symmetry and the lower and upper bound estimates hold in the higher order setting, it is not trivial to compute the expressions for coefficients to the poly-Laplacian in both polar and logarithm cylindrical coordinates, which is an obvious difficulty to obtain refined asymptotics for the poly-Laplacian system related to \eqref{oursystem}.
	\end{remark}

	\begin{remark}
		Our choice for the symbol $\Delta^2_{\rm cyl}=\Delta^2_{\rm sph}\circ\mathfrak{F}^{-1}$ is an abuse of notation, since the cylindrical background metric is not flat, we should have $P_{\rm cyl}=\Delta_{\rm sph}^2\circ\mathfrak{F}^{-1}$ $($resp. $\widetilde{P}_{\rm cyl}=\Delta_{\rm sph}^2\circ\widetilde{\mathfrak{F}}^{-1}$$)$, where $P_{\rm cyl}$ stands for the Paneitz--Branson operator of this metric in the new logarithmic cylindrical coordinate system.
	\end{remark}
	
	\subsection{Pohozaev invariant}
	In the next step, we define a type homological invariant associated to \eqref{oursystem}. This invariant is the main ingredient to provide a removable singularity theorem and one of the features for developing the convergence method in Section \ref{sec:qualitativeproperties}.
	The existence of a Pohozaev-type invariant is closely related to a conservation law for the Hamiltonian energy of the ODE system \eqref{fowlersystem}, which is based on \cite{MR0192184,MR1666838,MR2393072,MR4123335,arXiv:2002.12491}.
	
	Initially, let us introduce a vectorial energy which is conserved in time for all $p$-map solutions $\mathcal{V}$ to system \eqref{sphevectfowler}, which now also depends on the angular variable (see \cite{MR3869387,MR4094467,MR4123335}). 
	
	\begin{definition} 
		For any $\mathcal{V}$ solution to \eqref{fowlersystem}, let us consider its {\it Hamiltonian Energy} given by
		\begin{equation}\label{vectenergy}
		\mathcal{H}(t,\theta,\mathcal{V}):=\mathcal{H}_{\rm rad}(t,\theta,\mathcal{V})+\mathcal{H}_{\rm ang}(t,\theta,\mathcal{V}),
		\end{equation}
		where
		\begin{align*}
		&\mathcal{H}_{\rm rad}(t,\theta,\mathcal{V})=
		-\langle \mathcal{V}^{(3)}(t,\theta),\mathcal{V}^{(1)}(t)\rangle+\frac{1}{2}|\mathcal{V}^{(2)}(t,\theta)|^{2}+\frac{K_2}{2}|\mathcal{V}^{(1)}(t,\theta)|^{2}-\frac{K_0}{2}|\mathcal{V}(t,\theta)|^2+\widehat{c}(n)|\mathcal{V}(t,\theta)|^{2^{**}}, \\&
		\mathcal{H}_{\rm ang}(t,\theta,\mathcal{V})=|\Delta_{\theta}\mathcal{V}(t,\theta)|^2+2|\partial^{(2)}_t\nabla_{\theta}\mathcal{V}(t,\theta)|^2-J_0|\nabla_{\theta}\mathcal{V}(t,\theta)|^2,\quad \mbox{and} \quad \widehat{c}(n)={2^{**}}^{-1}c(n).
		\end{align*}
		
	\end{definition}
	A standard computation shows that the Hamiltonian energy is invariant on the variable $t$, that is, ${\partial_t}\mathcal{H}(t,\theta,\mathcal{V})=0$ for all solutions $\mathcal{V}$ to \eqref{sphevectfowler}. 
	Hence, we can integrate \eqref{vectenergy} over the cylindrical slice to define another conserved quantity as follows
	
	\begin{definition} 
		For any $\mathcal{V}$ solution to \eqref{sphevectfowler}, let us define its {\it cylindrical Pohozaev integral} by 
		\begin{align*}
		\mathcal{P}_{\rm cyl}(t,\mathcal{V})&=\displaystyle\int_{\mathbb{S}_t^{n-1}}\mathcal{H}(t,\theta,\mathcal{V})\ud\theta.&
		\end{align*}
		Here $\mathbb{S}_t^{n-1}=\{t\}\times\mathbb{S}^{n-1}$ is the cylindrical ball with volume element given by $\ud\theta=e^{-2t}\ud\sigma$, where $\ud\sigma_r$ is the volume element of the Euclidean ball of radius $r>0$. 
	\end{definition}
	
	Since that by definition $\mathcal{P}$ also does not depend on $t$, let us consider the {\it cylindrical Pohozaev invariant} $\mathcal{P}_{\rm cyl}(\mathcal{V}):=\mathcal{P}_{\rm cyl}(t,\mathcal{V})$. Hence, applying the inverse of cylindrical transformation, we recover the classical {\it spherical Pohozaev integral} defined by $\mathcal{P}_{\rm sph}(r,\mathcal{U}):=\left(\mathcal{P}_{\rm cyl}\circ\mathfrak{F}^{-1}\right)\left(t,\mathcal{V}\right)$, which satisfies the following Pohozaev-type identity:
	
	\begin{lemma}\label{lm:pohozaevidentity}
		Let $\mathcal{U},\widetilde{\mathcal{U}}\in C^{4}(B_1^*,\mathbb{R}^p)$ and $0<r_1\leqslant r_2<1$. 
		Then, it follows 
		\begin{align*}
		&\sum_{i=1}^p\int_{B_{r_2}\setminus B_{r_1}}\left[\Delta^2u_i\langle x,\nabla \widetilde{u}_i\rangle+\Delta^2\widetilde{u}_i\langle x,\nabla u_i\rangle-\gamma\left(\widetilde{u}_i\Delta^2u_i+u_i\Delta^2\widetilde{u}_i\right)\right]\ud x&\\\nonumber
		&=\sum_{i=1}^p\left[\int_{\partial B_{r_2}}q(u_i,\widetilde{u}_i)\ud\sigma_{r_2}-\int_{\partial B_{r_1}}q(u_i,\widetilde{u}_i)\ud\sigma_{r_1}\right].&
		\end{align*}
		Here $\gamma=(n-4)/2$ is the Fowler scaling exponent, and
		\begin{equation}\label{pohozaeverrorfunction}
		q(u_i,\widetilde{u}_i)=\sum_{j=1}^{3}\bar{l}_j(x,D^{(j)}u_i,D^{(4-j)}\widetilde{u}_i)+\sum_{j=0}^{3}\tilde{l}_j(D^{(j)}u_i,D^{(3-j)}\widetilde{u}_i),
		\end{equation}
		where both $\bar{l}_j,\tilde{l}_j$ are linear in each component for $j=0,1,2,3$. 
		Moreover,
		\begin{equation*}
		\bar{l}_3(D^{(3)}u_i,\widetilde{u}_i)=\gamma\int_{\partial B_{r_2}}\widetilde{u}_i\partial_{\nu}\Delta u_i\ud\sigma_{r_2}.
		\end{equation*}
		where $\nu$ is the outer normal vector to $\partial B_{r_2}$.
	\end{lemma}
	
	\begin{proof}
		See the proof in \cite[Proposition~3.3]{arXiv:1503.06412} and \cite[Propositon~A.1]{arxiv:1901.01678}.
	\end{proof}
	
	\begin{remark}
		Using the last lemma, we present a explicit formula for the spherical Pohozaev functional
		\begin{equation}\label{pohozaevsphericalautonomous}
		\mathcal{P}_{\rm sph}(r,\mathcal{U})=\sum_{i=1}^p\int_{\partial B_r}\left[q(u_i,u_i)-\frac{1}{s+1}r|\mathcal{U}|^{s+1}\right]\ud\sigma_r,
		\end{equation}
		where $q$ is defined by \eqref{pohozaeverrorfunction}.
	\end{remark}
		
	Consequently, using the Pohozaev identity in Lemma~\ref{lm:pohozaevidentity}, we observe that $\mathcal{P}_{\rm sph}(r,\mathcal{U})$ does not depend on $r$, which can also be stated in cylindrical coordinates.
	\begin{lemma}
		Let $\mathcal{U}$ be a strongly positive solution to \eqref{oursystem} and $0<r_1\leqslant r_2<1$. Then, $\mathcal{P}_{\rm sph}(r_2,\mathcal{U})-\mathcal{P}_{\rm sph}(r_1,\mathcal{U})=0$.
	\end{lemma}
	
	\begin{proof}
		According to \eqref{oursystem}, we have that $\Delta^{2} u_{i}-c(n)|\mathcal{U}|^{2^{**}-2}u_{i}=0$ in $ B_1^*$ for any $i\in I$, which, by
		multiplying by the factor $x\cdot\nabla u_i$, and integrating over $B_{r_2}\setminus B_{r_1}$, implies
		\begin{align*}
		0=\int_{B_{r_2}\setminus B_{r_1}}\left(\Delta^{2} u_{i}-c(n)|\mathcal{U}|^{2^{**}-2}u_{i}\right)(x\cdot\nabla u_i) \ \ud x:=I_{1,i}-I_{2,i}.
		\end{align*}
		Applying Lemma~\ref{lm:pohozaevidentity} on the left-hand side of the last equation, we get $\mathcal{P}_{\rm sph}(r_2,\mathcal{U})-\mathcal{P}_{\rm sph}(r_1,\mathcal{U})=0$ for all $i\in I$, which by summing over $i\in I$ finishes the proof.
	\end{proof}
	
	\begin{definition}
		For any $\mathcal{U}$ solution to \eqref{oursystem}, let us define its {\it spherical Pohozaev invariant} by
		\begin{equation*}
		\mathcal{P}_{\rm sph}(r,\mathcal{U}):=\mathcal{P}_{\rm sph}(\mathcal{U}).
		\end{equation*} 
	\end{definition}
	
	\begin{remark}\label{negativepohozaev}
		For easy reference, we summarize the following properties:\\
		\noindent{\rm (i)} There is a natural relation between these two invariants, $\mathcal{P}_{\rm sph}(\mathcal{U})=\omega_{n-1}\mathcal{P}_{\rm cyl}(\mathcal{V})$, where $\omega_{n-1}$ is the $(n-1)$-dimensional surface area of the unit sphere;\\ 
		\noindent{\rm (ii)} In the scalar case, one can similarly define $\mathcal{P}_{\rm sph}(u)=(\mathcal{P}_{\rm cyl}\circ\mathfrak{F}^{-1})(v)$, where
		\begin{equation*}
		\mathcal{P}_{\rm cyl}(v)=\displaystyle\int_{\mathbb{S}^{n-1}_t}\mathcal{H}(t,\theta,v)\ud\theta.
		\end{equation*}
		Hence, when $R=\infty$ one can check that, if the non-singular solution is ${\mathcal{U}}_{x_0,\mu}=\Lambda u_{x_0,\mu}$ for some $\Lambda\in\mathbb{S}_+^{p-1}$ and $u_{x_0,\mu}$ a spherical solution from Theorem~\ref{thm:frank-konig19}, we obtain $\mathcal{P}_{\rm sph}({\mathcal{U}}_{x_0,\mu})=\mathcal{P}_{\rm sph}(u_{x_0,\mu})=0$. 
		Also, if the singular solution has the form $\mathcal{U}_{a,T}=\Lambda u_{a,T}$ for some $\Lambda\in\mathbb{S}^{p-1}_{+,*}$ and $u_{a,T}$ an Emden--Fowler solution from Theorem~\ref{thm:andrade-doo19}, then, a direct computation shows $\mathcal{P}_{\rm sph}({\mathcal{U}}_{a,T})=\mathcal{P}_{\rm sph}(u_{a,T})<0$.
	\end{remark}
	
	\subsection{Linearized operator}
	Now we study the linearized operator around limit blow-up solution. The heuristics is that when the linearized operator around a blow-up limit is Fredholm, its indicial roots determine the rate in which singular solutions to the nonlinear problem \eqref{oursystem} converge to a blow-up limit solution near the isolated singularity. 
	Here, we borrow some ideas from \cite[Section~2]{MR3663325} (see also \cite{MR1936047,MR1356375}). 
	
	First, let us consider the following nonlinear operator acting on $p$-maps,
	\begin{equation*}
		\mathcal{N}(\mathcal{U}):=\Delta^2u_i-f_i(\mathcal{U}).
	\end{equation*}
	Then, using the cylindrical transformation and the homogeneity of the Gross--Pitaevskii nonlinearity, we obtain
	\begin{equation}\label{nonlinearoperator}
		\mathcal{N}_{\rm cyl}(\mathcal{V}):=\Delta^2_{\rm cyl}v_i-f_i(\mathcal{V}).
	\end{equation}
	In what follows, we drop the subscript since we often will be using the operator written in cylindrical coordinates.
	
	\begin{lemma}\label{linearizedoperator}
		The linearization of $\mathcal{N}: H^{4}(\mathcal{C}_0,\mathbb{R}^p)\rightarrow L^{2}(\mathcal{C}_0,\mathbb{R}^p)$ around an Emden--Fowler solution $\mathcal{V}_{a,T}$ is given by $\mathcal{L}^a(\Phi)=(\mathcal{L}^a_i(\Phi))_{i\in I}$. Here
		\begin{equation}\label{linearization}
			\mathcal{L}_i^a(\Phi)=\Delta^2_{\rm cyl}\phi_i-n\Lambda_i\langle\Lambda,\Phi\rangle v_{a,T}^{2^{**}-2}+\frac{\widetilde{c}(n)}{n}\left(nv_{a,T}^{2^{**}-2}+4-n\right)\phi_i,
		\end{equation}
		where $\widetilde{c}(n)=c(n)(2^{**}-1)$, or $\ud\mathcal{N}_{\rm cyl}\left[\mathcal{V}_{a,T}\right]\left(\Phi\right)=\mathcal{L}^a\left(\Phi\right)$ is the Fr\'echet derivative of $\mathcal{N}$.
	\end{lemma}
	
	\begin{proof} 
		By definition, we have that $\mathcal{L}_i^a(\Phi):=\mathcal{L}_i[\mathcal{V}_{a,T}](\Phi)$, where
		\begin{align}\label{linearizedcylvect}
			\mathcal{L}_i[\mathcal{V}_{a,T}](\Phi)
			&=\frac{\ud}{\ud t}\Big|_{t=0}\mathcal{N}(\mathcal{V}_{a,T}+t\Phi)&\\\nonumber
			&=\Delta^2_{\rm cyl}\phi_i-c(n)\left[\left(2^{**}-2\right)|\mathcal{V}_{a,T}|^{2^{**}-4}\langle\mathcal{V},\Phi\rangle v_i+|\mathcal{V}|^{2^{**}-2}\phi_i\right].&
		\end{align}
		In fact, writing $f_i(\mathcal{V})=c(n)|\mathcal{V}|^{2^{**}-2}v_i$ and using that $f_i$ is $(2^{**}-1)$-homogeneous, we find
		\begin{align*}
			\mathcal{N}(\mathcal{V}_{a,T}+t\Phi)-\mathcal{N}(\mathcal{V}_{a,T})
			&=\Delta^2_{\rm cyl}\mathcal{V}_{a,T}+t\Delta^2_{\rm cyl}\Phi-f_i\left({\mathcal{V}_{a,T}}+t\Phi\right)-\Delta^2_{\rm cyl}\mathcal{V}_{a,T}+f_i\left({\mathcal{V}_{a,T}}\right)&\\
			&=t\Delta^2_{\rm cyl}\Phi+f_i\left({\mathcal{V}_{a,T}}\right)-f_i\left({\mathcal{V}_{a,T}}+t\Phi\right)&\\
			&=t\Delta^2_{\rm cyl}\Phi-tc(n)\left[\left(2^{**}-2\right)|\mathcal{V}_{a,T}|^{2^{**}-4}\langle\mathcal{V},\Phi\rangle v_i+|\mathcal{V}|^{2^{**}-2}\phi_i\right]+\mathcal{O}\left(t^2\right),&
		\end{align*}
		which by taking $t\rightarrow0$ implies \eqref{linearizedcylvect}.
		Finally, using the classification in Theorem~\ref{thm:andrade-doo19}, we can simplify \eqref{linearizedcylvect} and obtain \eqref{linearization}.
	\end{proof}
	
	\subsection{Jacobi fields}\label{subsec:jacobifields}
	Unfortunately, the linearized operator is not Fredholm in general, since it does not have a closed range \cite[Theorem~5.40]{MR1348401}. 
	This issue is caused by the existence of nontrivial elements on its kernel; these are called the Jacobi fields \cite{MR1936047}.  
	Therefore, we need to introduce suitable weighted Sobolev and H\"{o}lder spaces in which the linearized operator has a well-defined right-inverse, up to a discrete set on the complex plane.
	For more details, see also \cite[Section~2]{MR1763040}.
	
	\begin{definition}
		Given $k,p,q\geqslant1$ and $\beta\in\mathbb{R}$, for any $\mathcal{V} \in L_{\loc}^{q}(\mathcal{C}_0,\mathbb{R}^p)$ define the following weighted Lebesgue norm
		\begin{equation*}
			\|\mathcal{V}\|_{L_{\beta}^{q}(\mathcal{C}_0,\mathbb{R}^p)}^{q}=\int_{0}^{\infty} \int_{\mathbb{S}^{n-1}}e^{-2\beta t}|\mathcal{V}(t, \theta)|^{q} \ud\theta\ud t.
		\end{equation*}	
		Also, let us define the weighted Lebesgue space by 
		\begin{equation*}
			L_{\beta}^{q}(\mathcal{C}_0,\mathbb{R}^p)=\left\{\mathcal{V}\in L_{\loc}^{q}(\mathcal{C}_0) : \|\mathcal{V}\|_{L_{\beta}^{q}(\mathcal{C}_0,\mathbb{R}^p)}<\infty\right\}.
		\end{equation*}	
		Similarly consider the Sobolev spaces $W_{\beta}^{k,q}(\mathcal{C}_0,\mathbb{R}^p)$ of $p$-maps with $k$ weak derivatives in $L^{q}$ having finite weighted norms.
		Here we also denote the Hilbert space $W^{k,2}_{\beta}(\mathcal{C}_0,\mathbb{R}^p)=H^{k}_{\beta}(\mathcal{C}_0,\mathbb{R}^p)$ and $W^{k,q}(\mathcal{C}_0)=W^{k,q}(\mathcal{C}_0,\mathbb{R})$. Notice that when $\beta=0$, we recover the classical Sobolev spaces of $p$-maps. 
	\end{definition}
	
	\begin{definition}
		Given $m,p\geqslant1$, $\beta\in\mathbb{R}$ and $\zeta\in(0,1)$,
		for any $u \in \mathcal{C}_{\loc}^{0,\beta}(\mathcal{C}_0,\mathbb{R}^p)$ define the following norm 
		\begin{equation*}
			\|\mathcal{V}\|_{C_{\beta}^{0,\zeta}(\mathcal{C}_0,\mathbb{R}^p)}=\sup_{T>1} \sup \left\{\frac{e^{-\beta t_{1}} |\mathcal{V}\left(t_{1}, \theta_{1}\right)|-e^{-\beta t_{2}}|\mathcal{V}\left(t_{2}, \theta_{2}\right)|}{\ud_{\rm cyl}\left(\left(t_{1}, \theta_{1}\right),\left(t_{2}, \theta_{2}\right)\right)^{\zeta}}:\left(t_{1}, \theta_{1}\right),\left(t_{2}, \theta_{2}\right) \in\mathcal{C}_{T-1,T+1}\right\}.
		\end{equation*}
		Also, let us define the $($zeroth order$)$ weighted Hold\"{e}r space by 
		\begin{equation*}
			C_{\beta}^{0,\zeta}(\mathcal{C}_0,\mathbb{R}^p)=\left\{\mathcal{V}\in \mathcal{C}_{\loc}^{0,\beta}(\mathcal{C}_0,\mathbb{R}^p) : \|\mathcal{V}\|_{\mathcal{C}_{\loc}^{0,\beta}(\mathcal{C}_0,\mathbb{R}^p)}<\infty\right\}.
		\end{equation*}
		One can similarly define higher order weighted Hold\"{e}r spaces $C_{\beta}^{m,\zeta}(\mathcal{C}_0,\mathbb{R}^p)$.
	\end{definition}
	
	\begin{remark}
		The spaces defined above are the suitable functional spaces to obtain the asymptotic results in Theorem~\ref{theorem1}, since $v\in {W^{k,q}_{\beta}(\mathcal{C}_0)}$ is equivalent to $v\in {W^{k,q}(\mathcal{C}_0)} $ together with the decay $v=\mathcal{O}(e^{-\beta t})$ as $t\rightarrow\infty$. Additionally, by regularity theory, we can indistinguishably work with both the Sobolev or the H\"{o}lder spaces. 
	\end{remark}
	
	\begin{definition}\label{def:jacobifield}
		The {\it Jacobi fields} in the kernel of $\mathcal{L}^a:H^{4}_{\beta}(\mathcal{C}_0,\mathbb{R}^p)\rightarrow L^{2}_{\beta}(\mathcal{C}_0,\mathbb{R}^p)$, are the solutions $\Phi\in H^{4}_{\beta}(\mathcal{C}_0,\mathbb{R}^p)$ of the following fourth order linear system,
		\begin{equation}\label{jacobisystem}
			\mathcal{L}^a(\Phi)=0 \quad {\rm on} \quad \mathcal{C}_0.
		\end{equation}
	\end{definition}
	
	\subsection{Fourier eigenmodes}
	We study the kernel of linearized operator around an Emden--Fowler solution by decomposing into its {\it Fourier eigenmodes}, that is, a {\it separation of variables} technique. 
	First, let us consider $\{\lambda_j,\chi_j(\theta)\}_{j\in\mathbb{N}}$ the eigendecomposition of the Laplace--Beltrami operator on $\mathbb{S}^{n-1}$ with the normalized eigenfunctions, 
	\begin{equation}\label{spectrallaplacian}
		\Delta_{\theta}\chi_j(\theta)+\lambda_{j}\chi_j(\theta)=0.
	\end{equation}
	Here the eigenfunctions $\{\chi_j(\theta)\}_{j\in\mathbb{N}}$ are called {\it spherical harmonics} with associated sequence of eigenvalues $\{\lambda_j\}_{j\in\mathbb{N}}$ given by $\lambda_j=j(j+n-2)$ counted with multiplicity $\mathfrak{m}_j$, which are defined by  
	\begin{equation*}
		\mathfrak{m}_0=1 \quad \mbox{and} \quad \mathfrak{m}_j=\frac{(2j+n-2)(j+n-3)!}{(n-2)!j!}.
	\end{equation*}
	In particular, we have $\lambda_0=0, \quad \lambda_1=\dots=\lambda_n=n-1$, $\lambda_{j}\geqslant 2n$, if $j>n$ and $\lambda_j\leqslant\lambda_{j+1}$.
	Moreover, these eigenfunctions are the restrictions to $\mathbb{S}^{n-1}$ of homogeneous harmonic polynomials in $\mathbb{R}^n$. Here we denote by $V_j$ the eigenspace spanned by ${\chi_j(\theta)}$. Using \eqref{spectrallaplacian}, it is easy to observe that the eigendata of the bi-Laplace--Beltrami operator $\Delta^2_{\theta}$ is given by $\{\lambda_j^2,\chi_j(\theta)\}_{j\in\mathbb{N}}$. Namely, for all $j\in\mathbb{N}$, it follows
	\begin{equation}\label{spectralbilaplacian}
		\Delta^2_{\theta}\chi_j(\theta)-\lambda_j^2 \chi_j(\theta)=0.
	\end{equation}
	
	\subsubsection{Scalar case}
	When $p=1$, the nonlinear operator \eqref{nonlinearoperator} becomes
	\begin{equation*}
		\mathcal{N}(v):=\Delta^2_{\rm cyl}v-c(n)v^{2^{**}-1}
		\quad
		\mbox{and}
		\quad
		\mathcal{L}^a(\phi)=\Delta^2_{\rm cyl}\phi-\widetilde{c}(n)v_{a,T}^{2^{**}-2}\phi.
	\end{equation*}
	Furthermore, using the decomposition \eqref{cylbi-Laplacian} combined with \eqref{spectrallaplacian} and \eqref{spectralbilaplacian}, we get
	\begin{equation*}
		\mathcal{L}^a(\phi)=\partial_t^{(4)}\phi-K_2\partial_t^{(2)}\phi+K_0\phi+\Delta_{\theta}^{2}\phi+2\partial^{(2)}_t\Delta_{\theta}\phi-J_0\Delta_{\theta}\phi-\widetilde{c}(n)v_{a,T}^{2^{**}-2}\phi,
	\end{equation*}
	which by projecting on the eigenspaces $V_j$ gives us
	\begin{equation}\label{jacobiequationproj}
		\mathcal{L}^{a}_j(\phi)=\phi^{(4)}-(K_2+2\lambda_j)\phi^{(2)}+ \left[K_0+\lambda_j(\lambda_j+J_0)-\widetilde{c}(n)v_{a,T}^{2^{**}-2}\right]\phi.
	\end{equation}
	Moreover, for any $\phi\in L^{2}(\mathbb{S}^{n-1})$, we write
	\begin{equation*}
		\phi(t,\theta)=\sum_{j=0}^{\infty}\phi_j(t)\chi_j(\theta), \quad \mbox{where} \quad \phi_j(t)=\int_{\mathbb{S}^{n-1}}\phi(t,\theta)\chi_j(\theta)\ud\theta.
	\end{equation*}
	In other terms, $\phi_j$ is the projection of $\phi$ on the eigenspace $V_j$. 
	Thus, to understand the kernel of $\mathcal{L}^{a}$, we consider the induced family of ODEs $\mathcal{L}^{a}_{j}(\phi_j)=0$ for $j\in\mathbb{N}$. More accurately, for each $a\in(0,a_0)$, it follows
	\begin{equation*}
		\ker(\mathcal{L}^a)=\bigoplus_{j\in\mathbb{N}}\ker(\mathcal{L}_{j}^{a}) \quad \mbox{and} \quad \operatorname{spec}(\mathcal{L}^a)=\bigcup_{j\in\mathbb{N}}\overline{\operatorname{spec}(\mathcal{L}_{j}^{a})}.
	\end{equation*} 
	
	\begin{remark}\label{geoemtricjacobifields}
		For $p=1$ and $j=0,1,\dots,n$ some (low-frequency) Jacobi fields are generated by the variation of a two-parameter family of Emden--Fowler solutions. In particular,  when $j=0$, they are given by
		\begin{equation*}
			\phi^+_{a,0}(t)=\partial_T\big|_{T=0} v_{a,T}(t) \quad \mbox{and} \quad \phi^-_{a,0}(t)=\partial_a\big|_{a=0} v_{a,T}(t).
		\end{equation*}
		However, the other two Jacobi fields cannot be directly constructed as variations of some family of solution to the limit equation. 
		One can show that they are not on the basis of this zero-frequency case.
		When $j=1,\dots,n$, the same construction can be performed; in this case we have that $\phi^\pm_{a,1}=\cdots=\phi^\pm_{a,n}$ with exponential growth/decay. More explicitly, let $\{{\bf e}_{j}\}^n_{j=1}$ be the standard basis of $\mathbb{R}^{n}$; thus, $\chi_j(\theta)=\langle{\bf e}_{j},\theta\rangle$ is the eigenfunction associated to $\lambda_j=n-1$. Taking $x_0=\tau{\bf e}_{j}$ in \eqref{asymptoticsjacobifields1} and \eqref{asymptoticsjacobifields2} provides
		\begin{equation*}
			v_{a,\tau{\bf e}_{j}}(t,\theta)=v_{a}(t)+\tau e^{-t}\chi_{j}(\theta)\left(-v_a^{(1)}+\gamma v_a\right)+\mathcal{O}\left(e^{-2 t}\right) \quad \mbox{as} \quad t\rightarrow\infty,
		\end{equation*}
		which by differentiating with respect to $\tau$ implies
		\begin{equation*}
			\phi_{a,j}^{-}(t,\theta)=e^{-t}\left(-v_a^{(1)}+\gamma v_a\right)+\mathcal{O}\left(e^{-2 t}\right) \quad \mbox{as} \quad t\rightarrow\infty.
		\end{equation*}
		Moreover, a direct computation gives us $\mathcal{L}^a_j(\phi^{-}_{a,j})=0$.
		Similarly, we can start differentiating the translation to obtain
		\begin{equation*}
			\phi_{a,j}^{+}(t,\theta)=e^{-t}\left(v_a^{(1)}+\gamma v_a\right)+\mathcal{O}\left(1\right) \quad \mbox{as} \quad t\rightarrow\infty.
		\end{equation*}
		Also notice that by Remark~\ref{signjacobifields}, whenever one has a parameter family of solutions to a nonlinear equation, the derivatives with respect to each parameter provides a solution to its linearized equation. 
	\end{remark}
	
	\subsubsection{System case}
	For $p>$1 and $\Phi\in L^{2}(\mathcal{C}_0,\mathbb{R}^p)$, we write
	\begin{equation}\label{Fourierdecomposition}
		\Phi(t,\theta)=\sum_{j=0}^{\infty}\Phi_j(t)\chi_j(\theta), \quad \mbox{where} \quad \Phi_j(t)=\int_{\mathbb{S}^{n-1}}\Phi(t,\theta)\chi_j(\theta)\ud\theta.
	\end{equation}
	Hence, for all $i\in I$ and $j\in\mathbb{N}$, we decompose \eqref{jacobisystem} as
	\begin{equation}\label{jacobiequhationprojsystem}
		\mathcal{L}^{a}_{ij}(\Phi)=\phi_i^{(4)}-(K_2+2\lambda_j)\phi_i^{(2)}+ \left[K_0+\lambda_j(\lambda_j+J_0)-\widetilde{c}(n)v_{a,T}^{{2^{**}-2}}\right]\phi_i-n\Lambda^*_i\langle\Lambda^*,\Phi\rangle v_{a,T}^{2^{**}-2}.
	\end{equation}
	Whence, to understand the kernel of \eqref{jacobiequhationprojsystem}, we again consider the induced equations $\mathcal{L}^{a}_{ij}(\phi_j)=0$. Therefore, the study of the kernel of $\mathcal{L}^{a}$ reduces to solving infinitely many ODEs. 
	In other terms, for each $a\in(0,a_0)$, it follows
	\begin{equation*}
		\ker(\mathcal{L}^a)=\bigoplus_{i\in I}\bigoplus_{j\in\mathbb{N}}\ker(\mathcal{L}_{ij}^{a}) \quad \mbox{and} \quad \operatorname{spec}(\mathcal{L}^a)=\bigcup_{i\in I}\bigcup_{j\in\mathbb{N}}\overline{\operatorname{spec}(\mathcal{L}_{ij}^{a})}.
	\end{equation*} 
	In Fourier analysis, it is convenient to divide any $\Phi\in L^2(\mathcal{C}_0,\mathbb{R}^p)$ into its frequency modes by
	\begin{equation*}
		\pi_0[\Phi](t,\theta)=\Phi_0(t)\chi_0(\theta), \
		\pi_1[\Phi](t,\theta)=\sum_{j=1}^{\mathfrak{m}_1}\Phi_j(t)\chi_j(\theta), \ \mbox{and} \ \pi_l[\Phi](t,\theta)=\sum_{j=\mathfrak{m}_l+1}^{\mathfrak{m}_{l+1}}\Phi_j(t)\chi_j(\theta).
	\end{equation*}
	In particularly, the projections $\pi_0,\pi_1$ and $\sum_{l=2}^{\infty}\pi_l$ are called respectively the {\it zero-frequency}, {\it low-frequency}, and {\it high-frequency modes}.
	
	\subsection{Indicial roots}
	A general principle in nonlinear analysis states that the asymptotic behavior of the linearized $\mathcal{L}^a$ on the cylinder $\mathcal{C}_0$ is directly related to its indicial roots, at least when the equilibrium point is hyperbolic, that is, zero is not in spectrum of the linearization around a limit solution. Heuristically, this is a version of Poincar\'e--Bendixon theory, or, more generally, Hartman--Grobman theory \cite{MR0171038}. 
	Inspired by \cite[Definition~5.0.1]{pacard2008}, we present the following definition:
	
	\begin{definition}
		A real number $\beta\in\mathbb{R}$ is called an indicial root of the linearized operator $\mathcal{L}_{j}^a: H^4(\mathcal{C}_0,\mathbb{R}^p)\rightarrow L^2(\mathcal{C}_0,\mathbb{R}^p)$, if there exist a non-zero function $\Phi\in C^{4,\zeta}(\mathcal{C}_0,\mathbb{R}^p)$ and $\beta'<\beta$ such that
		\begin{equation*}
			\liminf_{t\rightarrow\infty}\|\Phi(t,\theta)\|_{L^{\infty}\left(\mathbb{S}_t^{n-1},\mathbb{R}^p\right)}>0 \quad \mbox{and} \quad \lim_{t\rightarrow\infty}e^{-\beta't}\mathcal{L}_{j}^a(\Phi(t,\theta)e^{\beta t})=0.
		\end{equation*}
		Denote by $\mathfrak{I}^a_{j}$, the indicial roots of $\mathcal{L}_{j}^a$ for all $j\in\mathbb{N}$.
		Moreover, $\mathfrak{I}^{a}=\cup_{j\in\mathbb{N}}\mathfrak{I}^a_{j}$. 
	\end{definition}
	
	\begin{remark}
		When the linearized operator has constant coefficients, the indicial roots are the solutions to the {\it characteristic equation}. However, if this operator has periodic coefficients, the indicial exponents are the real part of the eigenvalues of the monodromy matrix, and they are called the {\it Floquet exponents}. One could also compare these exponents with the definition of indicial roots for Fuchsian operators.
	\end{remark}
	
	\section{Linear Analysis}\label{sec:linearanalysis}
	The objective of this section is to prove Proposition~\ref{growthpropertiessystem}. 
	More precisely, we show the linear stability for the linearized operator by studying its spectrum. Consequently, we can control the asymptotics to the global solutions by the growth of the Jacobi fields, which can be computed using Floquet theory (or Bloch wave theory). Namely, we prove that $\operatorname{spec}(\mathcal{L}^a)$ is a disjoint union of nondegenerate intervals, and $0\in\mathfrak{I}_a$ is an isolated point.
	The strategy is to use the integral decomposition in Proposition~\ref{integraldecomposition} to study the spectral bands of the Jacobi operator. We proceed by applying the Fourier--Laplace transform together with some results from holomorphic functional analysis \cite{MR194163,MR1356375} (see also \cite{MR1009163,MR2028503}).
	For complex numbers, we denote $\rho=\alpha+i\beta$, where $\Re(\rho)$,$\Im(\rho)$ stands for its real and imaginary parts, respectively.
	
	\subsection{Fourier--Laplace transform}\label{subsection:spectralanalysis}
	Following \cite[Section~4]{MR1356375} (see also \cite[Section~3]{MR4001465}), we consider the Fourier--Laplace transform, which is the suitable transformation to invert the linearized operator in the frequency space, that is, with respect to the Fourier parameter. 
	We can use the real parameter $\alpha=\Re z$ for $\rho\in\mathfrak{R}_a$ to move the weight of the Sobolev space and invert this transform, up to some region in the complex plane. 
	Before, we need to introduce some background notation and tools.
	Here we recall that $T_a\in\mathbb{R}$ is the fundamental period of the Emden--Fowler solution $v_a$ given by \eqref{fowler4order}.
	
	\begin{definition}
		Let $\widetilde{\beta}\in\mathbb{R}$ and $\Phi\in H_{\widetilde{\beta}}^{k}({\mathcal{C}_{\infty}},\mathbb{R}^p)$, where $\Phi$ is extend to be zero on the region ${\mathcal{C}}_{\infty}\setminus\mathcal{C}_0$. We define the Fourier--Laplace transform , given by
		\begin{equation}\label{laplace-Fouriertransform}
			\mathcal{F}_{a}(\Phi)(t,\theta,\rho)=\sum_{l\in\mathbb{Z}}e^{-il\rho}\Phi\left(t+l T_{a},\theta\right),
		\end{equation}
		where $\rho \in \mathfrak{R}_a:=\{\alpha+i\beta\in \mathbb{C} : \beta<-\widetilde{\beta}T_{a}\}\subset\mathbb{C}$ for some $\widetilde{\beta}\in\mathbb{R}$.
		For the sake of simplicity, we fix the notation $\widehat{\Phi}(t,\theta,\rho):=\mathcal{F}_{a}(\Phi)(t,\theta,\rho)$.
	\end{definition}
	
	Due to the periodicity properties of the linearized operator, it makes sense to define the following spaces, which is sometimes referred to as {\it the space of Bohr $\alpha$-quasi-periodic $p$-maps}.
	
	\begin{definition}
		For $\mathcal{C}_{0,T_a}:=[0,T_a]\times\mathbb{S}^{n-1}$ and $\alpha\in\mathbb{R}$, let us introduce $L_{\alpha}^{2}\left(\mathcal{C}_{0,T_a},\mathbb{C}^p\right)$ defined as the $L^{2}$-completion of $\mathcal{C}_{\alpha}^{0}\left(\mathcal{C}_{0,T_a},\mathbb{C}^p\right)$, where
		\begin{equation*}
			{C}_{\alpha}^{0}\left(\mathcal{C}_{0,T_a},\mathbb{C}^p\right):=\left\{\Phi \in {C}^{0}\left(\mathcal{C}_{0,T_a},\mathbb{C}^p\right):\Phi\left(T_{a}, \theta\right)=e^{i\alpha t} \Phi(0,\theta)\right\}.
		\end{equation*}
	\end{definition}
	
	The main proposition of this subsection is the direct integral decomposition of $L^{2}(\mathcal{C}_0,\mathbb{R}^p)$ in terms of the parameter $\alpha\in\mathbb{R}$ in the Fourier--Laplace transform, which will later provide the correct framework to invert the linearized operator.
	
	\begin{proposition}\label{integraldecomposition}
		For any $a\in[0,a_0]$, it follows 
		\begin{equation*}
			L^{2}\left(\mathcal{C}_0,\mathbb{R}^p\right)=\int_{\alpha\in[0,2\pi]}^{\oplus}L_{\alpha}^{2}\left(\mathcal{C}_{0,T_a},\mathbb{C}^p\right)\ud\alpha.
		\end{equation*}
	\end{proposition}
	
	\begin{proof}
		We divide the proof into a sequence of claims:
		
		\noindent{\bf Claim 1:} If $\rho\in\mathfrak{R}_a$, then the Fourier--Laplace $\mathcal{F}_a(\rho):L^2(\mathcal{C}_{\infty},\mathbb{R}^p)\rightarrow L^2(\mathcal{C}_{\infty},\mathbb{C}^p)$ transform is well-defined.
		
		\noindent In fact, since $\Phi\in H_{\widetilde{\beta}}^{k}\left(\mathcal{C}_0,\mathbb{R}^p\right)$ we know that $|\Phi(t, \theta)|=\mathcal{O}(e^{\widetilde{\beta} t})$, which yields
		\begin{align*}
			\left|\mathcal{F}_{a}(\Phi)(t,\theta,\rho)\right|
			=\sum_{l\in\mathbb{Z}}\left|e^{-i(\alpha+i\beta)l}\Phi\left(t+l T_{a},\theta\right)\right|
			=\sum_{l\in\mathbb{Z}}e^{\beta l}\left|\Phi\left(t+l T_{a},\theta\right)\right|\leqslant Ce^{\widetilde{\beta}t}\sum_{l\in\mathbb{Z}}e^{(\alpha+\widetilde{\beta}T_a)l},
		\end{align*}
		where we used that each choice of $\rho\in\mathbb{C}$ only gives finitely many zeros, and since $\rho\in\mathfrak{R}_a$, we use the growth property to conclude that all the exponents in the series are negative. Therefore, the last sum must converge uniformly in $\mathfrak{R}_a$.
		We can rephrase this conclusion like $\mathcal{F}_{a}$ is analytic whenever $\Phi\in H_{\beta}^{k}\left(\mathcal{C}_0,\mathbb{R}^p\right)$.
		
		In the next claim, we invert the Fourier--Laplace transform by using a contour integral in some inversion branch on the complex plane.
		
		\noindent{\bf Claim 2:} Let $\Phi\in H_{\beta}^{k}\left(\mathcal{C}_0,\mathbb{R}^p\right)$ and $\rho\in\mathfrak{R}_a$. 
		For each $t$ choose $\bar{t}\in[0, T_{a})$ such that $t=\bar{t} \mod T_{a}$, that is, there exists $l_0\in\mathbb{Z}$ satisfying $t=\bar{t}+l_0T_a$. Then, we get
		\begin{equation*}
			\Phi(t,\theta)=\frac{1}{2\pi} \int_{\alpha=0}^{2\pi} e^{il_0T_{a}(\alpha+i \beta)}\widehat{\Phi}(\bar{t},\alpha+i\beta, \theta) \ud\alpha.
		\end{equation*}
		\noindent Indeed, since $z=\alpha+i\beta$ and $\rho\in\mathfrak{R}_a$, we obtain
		\begin{align*}
			\frac{1}{2\pi}\int_{\alpha=0}^{2\pi}e^{\widetilde{\beta}} \widehat{\phi_i}(\bar{t},\theta,z)\ud\alpha &=\frac{1}{2\pi}\int_{\alpha=0}^{2\pi}e^{il_0\rho} \sum_{l\in\mathbb{Z}}e^{-il\rho} \phi_i\left(\bar{t}+lT_{a},\theta\right)\ud\alpha&\\
			&=\sum_{l\in\mathbb{Z}}\frac{1}{2\pi} \int_{\alpha=0}^{2\pi}e^{i(\alpha+i\beta)(l_0-l)} \phi_i\left(\bar{t}+lT_{a},\theta\right)\ud\alpha=\phi_i(t, \theta)&
		\end{align*}
		for all $i\in I$, which implies the proof of the claim.
		
		In the sequel, we show in which way the change 
		on the parameter $\beta\in\mathbb{R}$ on this inversion influences the weight norm of the transformed function. In other words, the next claim is a type of Parseval--Plancherel identity for $p$-maps. 
		
		\noindent{\bf Claim 3:} For each $\theta\in \mathbb{S}^{n-1}$ and $\rho=\alpha+i\beta$ for $\alpha,\beta\in\mathbb{R}$, it follows
		\begin{equation}\label{claim3}
			\left\|\widehat{\Phi}(t,\theta,\rho)\right\|_{L^{2}\left(\mathcal{S}_a,\mathbb{C}^p\right)}^{2} \simeq 2 \pi\left\|\widehat{\Phi}(t,\theta)\right\|^2_{L^2_{\beta/T_{a}}(\mathbb{R},\mathbb{R}^p)},
		\end{equation}
		where $\mathcal{S}_a=[0,T_a]\times[0,2\pi]$. 
		
		\noindent As a matter of fact, for all $i\in I$, it holds
		\begin{align*}
			&\int_{0}^{T_{a}}\int_{0}^{2 \pi}\left|\widehat{\phi}_i(t,\theta,\rho)\right|^2\ud\alpha\ud t&\\
			&=\int_{0}^{T_{a}}\int_{0}^{2 \pi}\left(\sum_{l\in\mathbb{Z}}e^{-il\alpha}e^{l\beta} \phi_i\left(t+lT_{a},\theta\right)\right)\left(\sum_{l\in\mathbb{Z}}e^{-il\alpha}e^{l\beta} \phi_i\left(t+lT_{a},\theta\right)\right)\ud\alpha\ud t&\\
			&=\int_{0}^{T_{a}}\int_{0}^{2 \pi}\sum_{l\in\mathbb{Z}}\sum_{\ell=-l}^l {{l}\choose{\ell}}e^{i(\ell-l)\alpha}e^{(\ell+l)\beta}  \phi_i\left(t+lT_{a},\theta\right)\phi_i\left(t+lT_{a},\theta\right)\ud\alpha\ud t&\\
			&=\int_{0}^{T_{a}}\int_{0}^{2 \pi}\sum_{l\in\mathbb{Z}}e^{2\beta l}\left|\phi_i\left(t+lT_{a},\theta\right)\right|^2\ud\alpha\ud t&\\
			&\simeq 2\pi\int_{\mathbb{R}}\left|e^{\beta t/T_a}\phi_i(t,\theta)\right|^2\ud t.&
		\end{align*}
		
		Next, we prove that $\widehat{\Phi}$ is a section of the flat bundle $\mathbb{T}_a^n=\mathbb{S}_a^1\times\mathbb{S}^{n-1}$ with holonomy $\rho\in\mathbb{C}$ around the $\mathbb{S}^{1}$ loop, where we identify $\mathbb{S}_a^1=\mathbb{R}/T_a\mathbb{Z}$.
		
		\noindent{\bf Claim 4:} For each $\theta \in \mathbb{S}^{n-1}$, we have
		\begin{equation}\label{Claim4}
			\left\|\widehat{\Phi}(t,\theta,\rho)\right\|_{L^{2}\left(\mathcal{S}_a,\mathbb{C}^p\right)}^{2} =2\pi\left\|\widehat{\Phi}(t,\theta)\right\|^2_{L^2(\mathbb{R},\mathbb{R}^p)}.
		\end{equation}
		
		\noindent Indeed, by taking $\beta=0$ in \eqref{claim3} and using \eqref{laplace-Fouriertransform}, we get 
		\begin{equation*}
			{\Phi}\left(t+T_{a}, \theta\right)=\mathcal{F}_{a}^{-1}\left(e^{i\rho} \mathcal{F}_{a}\left(\Phi\right)\right)(t,\theta),
		\end{equation*}
		which clearly concludes the proof of the claim.
		
		Finally, the proof of the proposition is a consequence of Claims 2,3, and 4.
	\end{proof} 
	
	\begin{remark}
		We would like to stress that the dependence of the inversion on the parameter $\alpha>0$ will allows us to change the growth rate of the solution you produce later from the Green's function of the twisted operator.
	\end{remark}
	
	\subsection{Spectral analysis}
	Now inspired by \cite[Section~4.2]{MR194163}, we study the geometric structure of the spectrum of the linearized operator around an Emden--Fowler solution.
	The idea is to construct a {\it twisted operator}, which
	captures the periodicity property of this linear operator, it is unitarily equivalent to the linearized operator, and a Fredholm theory is available.  
	In this direction, let us first introduce the suitable domain of definition for this operator.
	
	\begin{definition}
		For each $\alpha\in\mathbb{R}$ and $k\in\mathbb{N}$, let us define the set of {\it quasi-periodic} $p$-maps $H_{\alpha}^{k, 4}\left(\left[0, T_{a}\right],\mathbb{C}^p\right)$ to be the completion of the space of $C^{\infty}\left(\left[0, T_{a}\right],\mathbb{C}^p\right)$ under the  $H^{k}$-norm with boundary conditions given by
		\begin{equation*}
			\Phi^{(j)}\left(T_{a}\right)=e^{iT_{a}\alpha} \Phi^{(j)}(0) \quad \mbox{for} \quad j=0,1, \ldots, k-1.
		\end{equation*}
		We also denote by $\mathcal{L}^a_{ij,\alpha}$ the restriction of $\mathcal{L}^a_{ij}$ to $H_{\alpha}^{m}\left(\left[0, T_{a}\right]\right)$.
	\end{definition}
	
	Initially, for any $\rho\in\mathfrak{R}_a$ we use the inversion of the Fourier--Laplace transform to define $\widehat{\mathcal{L}}^a(\rho)=\mathcal{F}_a\circ\mathcal{L}^a\circ\mathcal{F}_a^{-1}$, or equivalently $\widehat{\mathcal{L}}^a(\rho)(\widehat{\Phi})=\widehat{\mathcal{L}^a(\Phi)}$, which by \eqref{Claim4} yields 
	\begin{equation*}
		\widehat{\mathcal{L}}^a(\rho)(e^{i\rho}\widehat{\Phi})(t,\theta,\rho)=e^{i\rho}\widehat{\mathcal{L}}^a(\rho)(e^{i\rho}\widehat{\Phi})(t,\theta,\rho) \quad \mbox{and} \quad e^{-i\rho}\widehat{\mathcal{L}}^a(\rho)(e^{i\rho}\widehat{\Phi})=\widehat{\mathcal{L}}^a(\rho)(\widehat{\Phi}).
	\end{equation*}
	Using the last relation, we set $\widetilde{\mathcal{L}}^a(\rho):H^{k+4}\left(\mathbb{T}_a^n,\mathbb{C}^p\right)\rightarrow H^{k}\left(\mathbb{T}_a^n,\mathbb{C}^p\right)$, given by
	\begin{equation}\label{twistedoperator}
		\widetilde{\mathcal{L}}^a(\rho)(\widehat{\Phi})=e^{i\rho}\mathcal{F}_a\circ\mathcal{L}^a\circ\mathcal{F}_a^{-1}\left(e^{-i\rho t}\widehat{\Phi}\right).
	\end{equation}
	
	\begin{remark}
		Notice that $\widehat{\mathcal{L}}^{a}$ has the same coordinate expression as $\mathcal{L}^a$. Thus, their Fourier eigenmodes decomposition  $\widehat{\mathcal{L}}_{j}^{a}$ and $\widetilde{\mathcal{L}}_{j}^{a}$ are also unitarily equivalent. 
		Moreover, by Claim 3 of Proposition~\ref{integraldecomposition} one has that ${\mathcal{L}}_{j,\alpha}^{a}$ coincides with the restriction of $\widehat{\mathcal{L}}_{j}^{a}(\alpha)$ to $[0,T_a]$. 
		Furthermore, $\widetilde{\mathcal{L}}^a(\rho)$ acts on the same functional space for all $\rho\in\mathbb{C}$.
	\end{remark}
	
	This motivates the following definition:
	
	\begin{definition}
		For each $a\in(0,a_0)$, $j\in\mathbb{Z}$, and $\alpha\in\mathbb{R}$, let us denote by $\sigma_{k}(a,j,\alpha)$ the eigenvalues of ${\mathcal{L}}_{j,\alpha}^{a}$. 
		In addition, since for each $a\in(0,a_0)$, $j\in\mathbb{Z}$, one has ${\mathcal{L}}_{j,0}^{a}={\mathcal{L}}_{j,2\pi}^{a}$, it follows that $\sigma_{k}(a,j,\cdot): \mathbb{S}^{1}\rightarrow\mathbb{R}$. Therefore, let us define the {\it $k$-th spectral band} of $\mathcal{L}^a_{j}$ by
		\begin{equation*}
			\mathfrak{B}_k(a,j)=\{\sigma_k\in\mathbb{R} : \sigma_k=\sigma_k(a,j,\alpha) \ {\rm for \ some} \ \alpha\in[0,2\pi/T_a]\}.
		\end{equation*}
	\end{definition}
	
	\begin{remark}
		Notice that $\spec(\mathcal{L}^a)=\spec(\widetilde{\mathcal{L}}^a)=\cup_{j,k\in\mathbb{N}}\mathfrak{B}_k(a,j)$.
	\end{remark}
	
	\begin{remark}\label{rem:quasiperiodic}
		The eigenfunction $\Phi_k$ corresponding to the eigenvalue $\sigma_{k}(a, j,\alpha)$ satisfies 
		\begin{equation*}
			\Phi\left(t+2\pi/T_{a}\right)=e^{i\alpha} {\Phi}(t)=e^{(2\pi-\alpha)i}\Phi(t) \quad \mbox{and} \quad  \widehat{\Phi}(t+2\pi)=e^{-i\alpha}\widehat{\Phi}(t). 
		\end{equation*}
		Furthermore, $\sigma_{k}(a, j, 2\pi-\alpha)=\sigma_{k}(a,j,\alpha)$, since
		$\mathcal{L}^a_{j}$ has real coefficients; thus, we can restrict $\sigma_k:\mathbb{S}^{1}\rightarrow\mathbb{R}$ to the half-circle corresponding to $\alpha\in[0,\pi]$.
	\end{remark}
	
	Now we have conditions to state and prove the most important result of this section.
	
	\begin{proposition}\label{isolatedindicialroot}
		For any $a\in[0,a_0]$, $0\in\mathfrak{I}_a$ is an isolated indicial root of $\mathcal{L}^a$.
	\end{proposition}
	
	\begin{proof}
		The proof follows by estimating the end points of the spectral bands of $\mathcal{L}^{a}$, and it will be divided into a sequence of claims:
		
		\noindent{\bf Claim 1:} For any $a\in(0,a_0)$ and $j,k\in\mathbb{N}$, the band $\mathfrak{B}_{k}(a, j)$ is a nondegenerate interval.
		
		\noindent In fact, each $\mathcal{L}^a_{j}$ is a fourth order ordinary differential operator such that the ODE system $\mathcal{L}_{a,j}\Phi=\sigma_{k}(a,j,\alpha)\Phi$ has a $4p$-dimensional solution space. Suppose that 
		$\mathfrak{B}_{k}(a,j)$ reduces to a single point, then $\sigma_{k}$ would be constant on $[0,2\pi]$ and $\mathcal{L}_{a,j}\Phi=\sigma_{k}(a,j,\alpha) \Phi$ would have an infinite dimensional solution space, which is contradiction.
		
		\noindent{\bf Claim 2:} For any $a\in(0,a_0)$ and $j,k\in\mathbb{N}$, it follows that
		\begin{equation*}
			\mathfrak{B}_{2k}(a, j)=\left[\sigma_{2 k}(a, j, 0),\sigma_{2k}(a, j, \pi)\right] \quad \mbox{and} \quad 
			\mathfrak{B}_{2k+1}(a, j)=\left[\sigma_{2k+1}(a, j, \pi), \sigma_{2k+1}(a, j, 0)\right].
		\end{equation*}
		
		\noindent This is a consequence of Floquet theory, since $\mathfrak{B}_{2k}$ are nondecreasing for any $k\in\mathbb{Z}$, whereas $\mathfrak{B}_{2k+1}$ are all nonincreasing. 
		Therefore, we conclude
		\begin{equation*}
			\sigma_{0}(a,j,0) \leqslant \sigma_{0}(a,j,\pi) \leqslant \sigma_{1}(a,j,\pi) \leqslant \sigma_{1}(a,j,0) \leqslant \ldots
		\end{equation*}
		
		\noindent{\bf Claim 3:} For any $a\in(0,a_0)$ and $j,k\in\mathbb{N}$, we find the lower bound
		\begin{equation}\label{spectralestimate}
			\sigma_{k}(a, j, 0)>\sigma_{0}(a, 0,\alpha)+J_0\lambda_{j}+\lambda_{j}^{2}.
		\end{equation}
		
		\noindent As a matter of fact, we can relate  $\mathfrak{B}_{k}(a,0)$ to $\mathfrak{B}_{k}(a,j)$ since
		\begin{equation*}
			\mathcal{L}^a_{j}-\mathcal{L}^a_{0}=-2 \lambda_{j}\partial_t^{(2)}+J_0 \lambda_{j}+\lambda_{j}^{2}, 
		\end{equation*}
		which for $\Phi$ an eigenvalue of $\mathcal{L}^a_{j,\alpha}$ implies
		\begin{equation}\label{bands1}
			\sigma_{k}(a,j,\alpha)\Phi=\mathcal{L}^a_{0}\Phi-2 \lambda_{j}\Phi^{(2)}+\left(J_0 \lambda_{j}+\lambda_{j}^{2}\right)\Phi.
		\end{equation}
		Using the decomposition $\Phi=\sum_{l=0}^{\infty}c_{l}\Phi_{l}$, where $\mathcal{L}^a_{0} \Phi_{l}=\sigma_{l}(a,0,\alpha)\Phi_{l}$ we can reformulate \eqref{bands1} as
		\begin{equation*}
			\sum_{l\in\mathbb{N}}c_{l}\sigma_{k}(a,j,\alpha)\Phi_{l}=\sum_{l\in\mathbb{N}} c_{l}\left[\sigma_{l}(a,0,\alpha)\Phi_{l}-2 \lambda_{j}\Phi^{(2)}_{l}+\left(J_0 \lambda_{j}+\lambda_{j}^{2}\right)\Phi_{l}\right],
		\end{equation*}
		which provides
		\begin{equation*}
			2\lambda_{j} \Phi^{(2)}_{l}=-\left[\sigma_{k}(a,j,\alpha)-\sigma_{l}(a,0,\alpha)-J_0 \lambda_{j}-\lambda_{j}^{2}\right]\Phi_{l}.
		\end{equation*}
		Finally, noticing that the last equation admits quasi-periodic solutions, if, and only if,
		\begin{equation*}
			\sigma_{k}(a, j, 0)>\sigma_{0}(a, 0,\alpha)+J_0\lambda_{j}+\lambda_{j}^{2},
		\end{equation*}
		we conclude the proof of the claim.
		
		\noindent{\bf Claim 4:} For any $a\in(0,a_0)$ and $j,k\in\mathbb{N}$, it follows that $\mathfrak{B}_{k}(a, j) \subset(0, \infty)$.
		
		\noindent This is the most delicate part; thus, we separate the proof into some steps. First, by the classification in Theorem~\ref{thm:andrade-doo19} (ii), we can reduce our analysis to the case $p=1$.
		
		\noindent{\bf Step 1:} For each $a\in\left(0, a_0\right]$, it follows
		\begin{equation*}
			\check{c}(n)\left(\frac{1}{T_{a}} \int_{0}^{T_{a}}v_{a}^{2^{**}}\ud t\right)^{1-2/2^{**}}\leqslant \sigma_{0}(a,0,0)<0,
		\end{equation*}
		where $\check{c}(n)=c(n)-\widetilde{c}(n)=-{n\left(n^{2}-4\right)}/{2}<0$. Moreover, either $\sigma_{1}(a,0,0)=0$ or $\sigma_{2}(a,0,0)=0$.
		
		\noindent  In fact, we start by the upper bound. Using the Rayleigh quotient of $\widetilde{\mathcal{L}}^a_{0}$, we get 
		\begin{equation}\label{rayleighquotient}
			\sigma_{0}(a,0,0)=\inf_{\phi\in H^4(\mathbb{T}_a^n)}\frac{\int_{0}^{T_a}\phi\widetilde{\mathcal{L}}^a_{0}\phi\ud t}{\int_{0}^{T_a}\phi^2\ud t}.
		\end{equation}
		Since $v_{a}$ is a periodic, it can be taken as a test function on the right-hand side of \eqref{rayleighquotient}; this provides
		\begin{align*}
			{\mathcal{L}}^a_{0}v_a
			&=v_a^{(4)}-K_2v_a^{(2)}+ K_0v_a-\widetilde{c}(n)v_{a,T}^{2^{**}-1}&\\
			&=v_a^{(4)}-K_2v_a^{(2)}+ K_0v_a-{c}(n)v_{a,T}^{2^{**}-1}+\check{c}(n)v_{a,T}^{2^{**}-1}&\\
			&=\check{c}(n)v_{a,T}^{2^{**}-1},&
		\end{align*}
		where we used that ${\mathcal{L}}^a_{0}$ and $\widetilde{\mathcal{L}}^a_{0}$ have the same coordinate expression. Hence, since $\check{c}(n)<0$, the estimate \eqref{spectralestimate} is a consequence of \eqref{rayleighquotient}.
		
		To prove the lower bound estimate, we observe that by \cite{MR4094467} combined with the classification in Theorem~\ref{thm:frank-konig19}, provides the variational characterization
		\begin{equation*}
			v_a=\inf_{\phi\in H_0^4([0,T_a])}\frac{\int_{0}^{T_a}\left(\left|\phi^{(2)}\right|^2-K_2\left|\phi^{(1)}\right|^2+K_0\left|\phi\right|^2\right)\ud t}{\left(\int_{0}^{T_a}\phi^{2^{**}}\ud t\right)^{2/2^{**}}}.
		\end{equation*} 
		Moreover, since $v_a$ satisfies \eqref{fowler4order}, we find
		\begin{equation}\label{estimate1}
			\frac{\int_{0}^{T_a}\left(\left|\phi^{(2)}\right|^2-K_2\left|\phi^{(1)}\right|^2+K_0\left|\phi\right|^2\right)\ud t}{\left(\int_{0}^{T_a}\phi^{2^{**}}\ud t\right)^{2/2^{**}}}\geqslant c(n)\left(\int_{0}^{T_a}v_a^{2^{**}}\ud t\right)^{1-2/2^{**}},
		\end{equation} 
		for all $\phi\in H_0^4([0,T_a])$.
		
		On the other hand, using the H\"{o}lder inequality, we get 
		\begin{equation}\label{estimate2}
			\int_{0}^{T_a}\phi^{2}\ud t\leqslant T_a^{1-2/2^{**}}\left(\int_{0}^{T_a}\phi^{2^{**}}\ud t\right)^{2/2^{**}}.
		\end{equation}
		Then, for all $\phi\in H_0^4([0,T_a])$ a combination of \eqref{estimate1} and \eqref{estimate2} yields
		\begin{align*}
			&\int_{0}^{T_a}\phi{\mathcal{L}}^a_{0}\phi\ud t&\\
			&=\int_{0}^{T_a}\left(\phi^{(4)}-K_2\phi^{(2)}+ K_0\phi-\widetilde{c}(n)v_{a,T}^{2^{**}-1}\phi\right)\ud t&\\
			&\geqslant c(n)\left(\int_{0}^{T_a}v_a^{2^{**}}\ud t\right)^{1-2/2^{**}}\left(\int_{0}^{T_a}\phi^{2^{**}}\ud t\right)^{2/2^{**}}-\widetilde{c}(n)\int_{0}^{T_a}\phi^{2^{**}}\ud t&\\
			&=\left(\int_{0}^{T_a}\phi^{2^{**}}\ud t\right)^{2/2^{**}}\left[c(n)\left(\int_{0}^{T_a}v_a^{2^{**}}\ud t\right)^{1-2/2^{**}}-\widetilde{c}(n)\left(\int_{0}^{T_a}\phi^{2^{**}}\right)^{1-2/2^{**}}\ud t\right]&\\
			&\geqslant T_a^{2/2^{**}-1}\left[c(n)\int_{0}^{T_a}\phi^2\ud t\left(\int_{0}^{T_a}v_a^{2^{**}}\ud t\right)^{1-2/2^{**}}-\widetilde{c}(n)\int_{0}^{T_a}\phi^2\ud t\left(\int_{0}^{T_a}\phi^{2^{**}}\right)^{1-2/2^{**}}\ud t\right].
			&
		\end{align*}
		In particular, taking $\phi\in H_0^4([0,T_a])$ such that $\|\phi\|_{H_0^4([0,T_a])}=\|v_a\|_{H_0^4([0,T_a])}$, we obtain
		\begin{equation*}
			\int_{0}^{T_a}\phi{\mathcal{L}}^a_{0}\phi\ud t\geqslant \check{c}(n)T_a^{2/2^{**}-1}\|\phi\|^2_{L^2([0,T_a])}\|v_a\|^{2^{**}-2}_{L^{2^{**}}([0,T_a])},
		\end{equation*}
		which directly implies the lower bound estimate. 
		
		Finally, since $\phi^+_{a,0}=\partial_a v_{a}$ is a periodic solution to $\mathcal{L}^a_{0}(\phi^+_{a,0})=0$, we have that there is an eigenfunction with associated eigenvalue $\lambda=0$ subjected to periodic boundary conditions provided by $\alpha=0$. 
		Besides, this eigenfunction has two nodal domains within the interval $\left[0,T_a\right]$, which is associated either to $\sigma_{1}(a,0)$ or to $\sigma_{2}(a, 0)$. 
		
		In the remaining steps, we provide a more precise localization of the spectral bands of $\mathcal{L}^a$:
		
		\noindent{\bf Step 2:} For any $a\in(0,a_0)$, it follows that $\mathfrak{B}_{k}(a, 0) \subset(0, \infty)$ for each $k\geqslant3$ and $\mathfrak{B}_{k}(a, 0)\subset[0,\infty)$ for each $k\geqslant2$.
		
		\noindent This a direct consequence of Claim 2 and Step 1.
		
		\noindent{\bf Step 3:} For any $a\in(0,a_0)$ and $j,k\in\mathbb{N}$, it follows that $\mathfrak{B}_{k}(a,j) \subset(0, \infty)$.
		
		\noindent In fact, when $j>n$ we have $\lambda_{j}>2n$, which by Claim 3 implies
		\begin{equation*}
			\sigma_{k}(a, j, 0)>\sigma_{0}(a, 0,0)+n^{3} \quad \mbox{for all} \quad k\in\mathbb{N}.
		\end{equation*}
		On the other hand, since $0<v_a(t)<1$, for all $t\in\mathbb{R}$, by the lower bound estimate, we find that $\sigma_{0}(a, 0,0)\geqslant\check{c}(n)$ and
		\begin{equation*}
			\sigma_{k}(a, j, 0)>\sigma_{0}(a, 0,0)+n^{3} \geqslant n^{3}+\check{c}(n)>0.
		\end{equation*}
		When $1\leqslant j \leqslant n$, it follows from the construction for the geometric Jacobi fields constructed in Remark~\ref{geoemtricjacobifields}, since
		\begin{equation*}
			\mathcal{L}^a_{j}\left(\phi_{a,j}^{\pm}\right)=0 \quad \mbox{and} \quad \phi_{a,j}^{\pm}=e^{\pm t}\left(\pm v^{(1)}_{a}+\gamma v_{a}\right)+\mathcal{E}_{\pm},
		\end{equation*} 
		where $\mathcal{E}_{+}(t)=\mathcal{O}(1)$ and $\mathcal{E}_{-}(t)=\mathcal{O}\left(e^{-2t}\right)$ as $t\rightarrow\infty$ are positive periodic solutions to $\mathcal{L}^a_{j}$.
		
		The last claim relates the spectral bands $\mathfrak{B}_{k}(a, j)$ and the indicial roots $\mathfrak{I}_j^a$.
		
		\noindent{\bf Claim 5:} The ODE $\mathcal{L}^a_{j}(\Phi)=0$ admits a quasi-periodic solution, if and only if, for some $k\in\mathbb{N}$, $0\in \mathfrak{B}_{k}(a,j)$  
		
		\noindent Indeed, we have that $\Phi=\mathcal{F}_{a}^{-1}\left(e^{-i\alpha t} \widehat{\Phi}\right)$ solves $\mathcal{L}_{a,j}\Phi=0$ since 
		\begin{equation*}
			0=\mathcal{L}^a_{j,\alpha}\widehat{\Phi}=e^{i \alpha t}\mathcal{F}_{a}\left(\mathcal{L}^a_{j}\left(\mathcal{F}_{a}^{-1}\left(e^{-i \alpha t} \widehat{\Phi}\right)\right)\right) \quad \mbox{and} \quad 
			\mathcal{L}^a_{j}\left(\mathcal{F}_{a}^{-1}\left(e^{-i \alpha t} \widehat{\Phi}\right)\right)=0.
		\end{equation*}
		Therefore, by Remark~\ref{rem:quasiperiodic}, the proof of the claim follows.
	\end{proof}
	
	\subsection{Fredholm theory }\label{subsec:fredholm}
	In this subsection, we investigate the spectrum of the linearized operator. 
	Indeed, our final goal is to conclude that $\mathcal{L}^a$ is Fredholm, which follows by showing that $\mathfrak{I}_a\subset\mathbb{R}$ is a discrete set. 
	The last assessment is not trivial to prove; in fact, we need to use the results concerning the Fourier--Laplace transform from Subsection~\ref{subsection:spectralanalysis} to find a right-inverse for the linearized operator. 
	Our strategy is based on some results from holomorphic functional analysis (see \cite{MR2028503,MR1009163}).
	
	\begin{definition}
		Let $H$ be a Hilbert space and $\mathcal{L}:H\rightarrow H$ be a linear operator. We say that $\mathcal{L}$ is {\it Fredholm} if $\mathcal{L}$ is bounded and has a closed finite-dimensional kernel and closed range.
	\end{definition}
	
	To invert the twisted operator, we use the analytic Fredholm theorem in \cite[Theorem~8.92]{MR2028503}. 
	
	\begin{propositionletter}\label{analyticfredholm}
		Let $H$ be a Hilbert space, $\mathfrak{R}\subseteq\mathbb{C}$ be a domain, and $\mathcal{F}:\mathfrak{R}\rightarrow\mathcal{L}(H)$ be a 
		meromorphic map such that $\mathcal{F}(\rho):H\rightarrow H$. Then, either\\
		\noindent{\rm (i)} $({\rm Id}-\mathcal{F}(\rho))^{-1}$ does not exist for all $\rho\in\mathfrak{R}$, or\\
		\noindent{\rm (ii)} $({\rm Id}-\mathcal{F}(\rho))^{-1}$ exists for $\rho\in\mathfrak{R}\setminus\mathfrak{D}$, where $\mathfrak{D}\subseteq\mathfrak{R}$ is a discrete set. Moreover, the map $\rho\mapsto({\rm Id}-\mathcal{F}(\rho))^{-1}$ is analytic and if $\rho\in\mathfrak{D}$, then $\mathcal{F}(\rho)\phi=\phi$ has a finite-dimensional solution space.
	\end{propositionletter}
	
	\begin{lemma}
		If $k\in\mathbb{N}$, $a\in(0,a_0)$ and $\beta\in\mathbb{R}$, then $\mathcal{L}^a: H_{\beta}^{k+4}(\mathcal{C}_0,\mathbb{R}^p)\rightarrow H_{\beta}^k(\mathcal{C}_0,\mathbb{R}^p)$ is a bounded linear elliptic self-adjoint operator.
	\end{lemma}
	
	\begin{proof}
		It follows since the symbol of $\mathcal{L}^a$ in cylindrical coordinates is given by $\partial^{(4)}_t+\Delta^2_{\theta}$.
	\end{proof}
	
	The main result of this subsection states the invertibility of the linearized operator. 
	Whence, we can use the analytic Fredholm theorem to prove the twisted operator is Fredholm away from a discrete set of poles on the complex plane. 
	
	\begin{proposition}\label{prop:fredholm}
		If $k\in\mathbb{N}$ and $\beta\notin\mathfrak{I}^{a}$, then $\mathcal{L}^a: H_{\beta}^{k+4}(\mathcal{C}_0,\mathbb{R}^p)\rightarrow H_{\beta}^k(\mathcal{C}_0,\mathbb{R}^p)$ is Fredholm.
	\end{proposition}
	
	\begin{proof}
		To apply the analytic Fredholm theorem, we use the twisted operator from \eqref{twistedoperator},
		\begin{equation*}
			\widetilde{\mathcal{L}}^a(\rho):H^{k+4}(\mathbb{T}_a^n,\mathbb{C}^p)\rightarrow H^{k}(\mathbb{T}_a^n,\mathbb{C}^p) \quad \mbox{given by} \quad 
			\widetilde{\mathcal{L}}^a(\rho)(\widehat{\Phi})=e^{i\rho}\mathcal{F}_a\circ\mathcal{L}^a\circ\mathcal{F}_a^{-1}\left(e^{-i\rho t}\widehat{\Phi}\right).
		\end{equation*}
		In what follows, we divide the proof into some claims:
		
		\noindent{\bf Claim 1:} If $\alpha\in(0,2\pi)$, then $\widetilde{\mathcal{L}}^a_{\alpha}$ is Fredholm.
		
		\noindent For each $\alpha\in(0,2\pi)$, the operator $\widetilde{\mathcal{L}}^a(\rho)$ is linear, bounded, elliptic and depends holomorphically on $\rho$. Thus, by Proposition~\ref{analyticfredholm} is either never Fredholm or it is Fredholm for $\rho$ outside a discrete set. 
		We take $z=\alpha\in(0,2 \pi)$ and suppose there exists $\widehat{\Phi}\in H^{k+4}(\mathbb{T}_a^n,\mathbb{R}^p)$ such that $\widetilde{\mathcal{L}}^a(\rho)(\widehat{\Phi})=0$; thus, ${\mathcal{L}}^a(\rho)({\Phi})$, where ${\Phi}=\mathcal{F}_a^{-1}(e^{-i\rho t}\widehat{\Phi})$. 
		Then, $\Phi$ is quasi-periodic, and in particular, $\Phi$ is bounded. However, by Proposition~\ref{isolatedindicialroot}, any bounded Jacobi field is a multiple of $\Phi_{0}^{+}$, which is not quasi-periodic. Hence, $\mathcal{L}^a(\alpha)$ is injective. 
		Finally, since this operator is formally self-adjoint, it follows that $\widetilde{\mathcal{L}}^a(\rho)$ is Fredholm. 
		
		\noindent{\bf Claim 2:} If $a\in(0,a_0)$ and $\beta\in\mathfrak{I}^a$, then there exists $\widetilde{\mathcal{G}}^a(\rho):H^{k}(\mathbb{T}_a^n,\mathbb{C}^p)\rightarrow H^{k+4}(\mathbb{T}_a^n,\mathbb{C}^p)$ such that $\widetilde{\mathcal{G}}^a(\rho)$ is right-inverse for $\widetilde{\mathcal{L}}^a(\rho)$.
		
		\noindent Using Claim 1, we can apply Proposition~\ref{analyticfredholm} to find a discrete set $\mathfrak{D}_a\subset\mathfrak{R}_a$ and a meromorphic operator
		\begin{equation*}
			\widetilde{\mathcal{G}}^a(\rho):H^{k}(\mathbb{T}_a^n,\mathbb{C}^p)\rightarrow H^{k+4}(\mathbb{T}_a^n,\mathbb{C}^p) \quad \mbox{such that} \quad 
			\widehat{\Phi}=\left(\widetilde{\mathcal{G}}^a(\rho)\circ\widetilde{\mathcal{L}}^a(\rho)\right)(\widehat{\Phi}),
		\end{equation*}
		for $\rho\notin\mathfrak{D}_a$. 
		
		\noindent{\bf Claim 3:} If $a\in(0,a_0)$ and $\beta\in\mathfrak{I}^a$, then there exists $\mathcal{G}^a:H_{\beta}^{k}(\mathcal{C}_0,\mathbb{R}^p)\rightarrow H_{\beta}^{k+4}(\mathcal{C}_0,\mathbb{R}^p)$ right-inverse for $\mathcal{L}^a$.
		
		\noindent Indeed, notice that $\mathfrak{I}^a=\{\beta\in\mathbb{R} : \beta=\Im(\rho) \ \mbox{for some} \ \rho\in\mathfrak{D}_a\}$, which provides
		\begin{equation*}
			{\mathcal{G}}^a(\Phi)=\mathcal{F}_{a}^{-1}\left(e^{-i\rho T_a t}\left(\widetilde{\mathcal{G}}^{a}\left(e^{-i\rho T_a t}\left(\mathcal{F}_{a}(\Phi)\right)\right)\right)\right).
		\end{equation*}
		Furthermore, by construction, we obtain that $\widehat{\Phi}={\mathcal{G}}^a(\Phi)\in H_{-\Im(\rho)}^{k+4}\left(\mathcal{C}_0,\mathbb{R}^p\right)$, which by the Fredholm alternative concludes the proof of the claim.
		The last claim proves the proposition.
	\end{proof}
	
	\begin{proposition}
		The set $\mathfrak{I}^{a}$ is discrete.
	\end{proposition}
	
	\begin{proof}
		Note that each element in $\mathfrak{I}^{a}$ is the imaginary part of a pole to $\widetilde{\mathcal{G}}_{a}$, which by the analytic Fredholm theory is a discrete subset of $\mathbb{C}$. On the other hand, the operator $\mathcal{L}^a(\rho)$ is unitarily equivalent to $\mathcal{L}^a(\rho+2\pi l)$ for each $l\in\mathbb{Z}$; thus, $\rho$ is a pole of ${\mathcal{G}}_{a}$, if and only if $\rho+2 \pi l$ also is for any $l\in\mathbb{Z}$. Therefore, ${\mathcal{G}}_{a}$ can only have finitely many poles in each horizontal strip.
	\end{proof}
	
	\begin{corollary}\label{cor:surjectiveness}
		If $k\in\mathbb{N}$ and $\beta\notin(0,1)$,then:\\
		\noindent {\rm (i)} the operator $\mathcal{L}^a: H^{k+4}_{\beta}(\mathcal{C}_0,\mathbb{R}^p)\rightarrow H_{\beta}^k(\mathcal{C}_0,\mathbb{R}^p)$ is surjective;\\
		\noindent {\rm (ii)} the operator $\mathcal{L}^a: H^{k+4}_{-\beta}(\mathcal{C}_0,\mathbb{R}^p)\rightarrow H_{-\beta}^k(\mathcal{C}_0,\mathbb{R}^p)$ is injective.
	\end{corollary}
	
	\begin{proof}
		In fact, it follows from the proof of Proposition~\ref{prop:fredholm} that $\mathcal{L}^a(\rho)$ is injective for each $\rho\in\mathbb{C}$ with $-1<\Im(\rho)<0$, which implies $\mathcal{L}^a: H_{-\beta}^{k+4}\left(\mathcal{C}_0,\mathbb{R}^p\right)\rightarrow H_{-\beta}^{k}\left(\mathcal{C}_0,\mathbb{R}^p\right)$ is injective, and thus (ii) is proved. 
		Besides, since by duality the operator $\mathcal{L}^a: H_{\beta}^{k+4}\left(\mathcal{C}_0,\mathbb{R}^p\right)\rightarrow H_{\beta}^{k}\left(\mathcal{C}_0,\mathbb{R}^p\right)$ is formally self-adjoint; thus, the surjectiveness follows, and (i) is proved.
	\end{proof}
	
	\subsection{Existence of singular solutions}\label{subsec:existence}
	We prove the existence of solutions to \eqref{oursystem}. 
	We proceed by studying the spectral properties of the linearized operator around an Emden--Fowler solution. 
	Let us remark that by the implicit function theorem, the existence of solutions to \eqref{oursystem} can be obtained by showing that the linearized operator $\mathcal{L}^a$ is Fredholm. 
	We already know that in some cases $\mathcal{L}^a$ does not satisfy this property since its kernel is not closed. To overcome this issue, we introduce the following definition:
	\begin{definition}
		For each $v_{a,T}$, let us consider the {\it deficiency space} generated by the Jacobi fields basis of the linearized operator. Then, \\
		\noindent{\rm (i)} for $j=0$, it follows $D_{a,0}(\mathcal{C}_0,\mathbb{R}^p)=\vspan\{\Phi_{a,0}^{+}, \Phi_{a,0}^{-}\}$;\\
		\noindent{\rm (ii)} for $j\geqslant1$, it follows $D_{a,j}(\mathcal{C}_0,\mathbb{R}^p)=\vspan\{\Phi_{a,j}^{+}, \Phi_{a,j}^{-}, \widetilde{\Phi}_{a,j}^{+}, \widetilde{\Phi}_{a,j}^{-}\}$.
	\end{definition}
	
	The fact that there are only two Jacobi fields in (i) of the last definition is a consequence of Proposition~\ref{isolatedindicialroot}.
	Namely, any zero-frequency Jacobi fields with growth less than exponential (tempered) is generated by the ones obtained by variation of geometric parameters in the Emden--Fowler solution.
	
	Now we can present the main result of the subsection:
	
	\begin{proposition}\label{existence}
		Let $\mathcal{V}_{a,T}$ be an Emden--Fowler solution. Then, 
		
		\noindent{\rm (i)} For $\beta\in(\beta_{a,0},\beta_{a,1})$, $\mathcal{L}^a: C_{\beta}^{4,\zeta}(\mathcal{C}_0,\mathbb{R}^p)\oplus D_{a,0}(\mathcal{C}_0,\mathbb{R}^p)\rightarrow C_{\beta}^{0,\zeta}(\mathcal{C}_0,\mathbb{R}^p)$ is a surjective Fredholm mapping with bounded right-inverse, given by
		\begin{equation*}
			\mathcal{G}_0^a:C_{\beta}^{0,\zeta}(\mathcal{C}_0,\mathbb{R}^p)\rightarrow C_{\beta}^{4,\zeta}(\mathcal{C}_0,\mathbb{R}^p)\oplus D_{a,0}(\mathcal{C}_0,\mathbb{R}^p).
		\end{equation*}
		\noindent{\rm (ii)} For $\beta\in (\beta_{a,1},\beta_{a,2})$, $\mathcal{L}^a: C_{\beta}^{4,\zeta}(\mathcal{C}_0,\mathbb{R}^p)\oplus D_{a,0}(\mathcal{C}_0,\mathbb{R}^p)\oplus D_{a,1}(\mathcal{C}_0)\rightarrow C_{\beta}^{0,\zeta}(\mathcal{C}_0,\mathbb{R}^p)$ is a surjective Fredholm mapping with bounded right-inverse, given by
		\begin{equation*}
			\mathcal{G}_1^a:=C_{\beta}^{0,\zeta}(\mathcal{C}_0,\mathbb{R}^p)\rightarrow C_{\beta}^{4,\zeta}(\mathcal{C}_0,\mathbb{R}^p)\oplus D_{a,0}(\mathcal{C}_0,\mathbb{R}^p)\oplus D_{a,1}(\mathcal{C}_0,\mathbb{R}^p).
		\end{equation*}
	\end{proposition}
	
	\begin{proof}
		We proceed as in Proposition~\ref{prop:fredholm}. First, we decompose the linearized operator in Fourier modes, and we apply the Laplace--Fourier transform. 
		Then, by conjugation, let us define a family of transformations satisfying the assumptions of Proposition~\ref{analyticfredholm}. We can therefore invert the conjugated operator $\widetilde{\mathcal{L}}_j^a(\rho)$. Afterwards, we reconstruct the functional by undoing the inverse of Fourier--Laplace transform. In other terms, for all $j\in\mathbb{N}$, we take the right-inverse,
		\begin{equation*}
			\mathcal{G}^a_j(\Phi)=\mathcal{F}_{a}^{-1}\left(e^{-i\rho t}\left(\widetilde{\mathcal{G}_j}^{a}\left(e^{-i\rho t}\left(\mathcal{F}_{a}(\Phi)\right)\right)\right)\right);
		\end{equation*}
		this provides the proof of the proposition.
	\end{proof}
	
	\begin{remark}
		The necessity of adding the deficiency spaces $D_{a,j}(\mathcal{C}_0,\mathbb{R}^p)$ comes from a simple form of the linear regularity theorem from R. Mazzeo et al. \cite[Lemma~4.18]{MR1356375} and some ODE theory. In addition, note that if $\beta=\pm\beta_{a,j}$, then $\mathcal{L}^{a}$ does not have closed range. Moreover, we have Schauder estimates in the sense of weighted spaces. More precisely, if $\mathcal{V}$ is solution to the nonhomogeneous problem $\mathcal{L}^{a}\left(\mathcal{V}\right)=\Psi$, then $\mathcal{V}\in C_{\beta}^{4,\zeta}(\mathcal{C}_0,\mathbb{R}^p)$  whenever $\Psi\in C_{\beta}^{0,\zeta}(\mathcal{C}_0,\mathbb{R}^p)$. More generally, it should be possible to find a inverse like
		\begin{equation*}
			\mathcal{G}_j^a:=C_{\beta}^{0,\zeta}(\mathcal{C}_0,\mathbb{R}^p)\rightarrow C_{\beta}^{4,\zeta}(\mathcal{C}_0,\mathbb{R}^p)\bigoplus_{l=0}^{j} D_{a,l}(\mathcal{C}_0,\mathbb{R}^p),
		\end{equation*}
		which would give us the refined information in the sense of \cite{arXiv:1909.10131v1,arXiv:1909.07466v1} $($See \eqref{higherorderasymptotics}$)$.
	\end{remark}
	
	As a consequence of our results, we present the main result of the subsection,
	
	\begin{corollary}
		There exists at least one strongly positive solution $\mathcal{V}$ of \eqref{sphevectfowler}.
	\end{corollary}
	
	Another application is the following improved regularity theorem for solutions to \eqref{oursystem} in cylindrical coordinates:
	
	\begin{corollary}\label{iteration}
		Let $\mathcal{V}$ be a strongly positive solution to $\mathcal{L}^{a}\left(\mathcal{V}\right)=\Psi$. Assume that $\mathcal{V}\in C_{\widetilde{\beta}}^{4,\zeta}(\mathcal{C}_0,\mathbb{R}^p)$ and $\Psi\in C_{\widehat{\beta}}^{0,\zeta}(\mathcal{C}_0,\mathbb{R}^p)$. Hence, it holds\\
		\noindent{\rm (i)} If $0<\widetilde{\beta}<\widehat{\beta}<1$, then $\mathcal{V}\in C_{\beta_2}^{4,\zeta}(\mathcal{C}_0,\mathbb{R}^p)$;\\
		\noindent{\rm (ii)} If $0<\widetilde{\beta}<1<\widehat{\beta}<\beta_{a,2}$, then $\mathcal{V}\in C_{\widehat{\beta}}^{4,\zeta}(\mathcal{C}_0,\mathbb{R}^p)\oplus D_{a,1}(\mathcal{C}_0,\mathbb{R}^p)$.
	\end{corollary}
	
	\begin{proof}
		First, we use the right-inverse operator $\mathcal{G}_0^a$ in Proposition~\ref{existence} to obtain 
		\begin{equation*}
			\widetilde{\mathcal{V}}+c\Phi^{+}_{a,0}=\mathcal{G}_0^a\left(\Psi\right)\in C^{4,\zeta}_{\beta}(\mathcal{C}_0,\mathbb{R}^p)\oplus D_{a,1}(\mathcal{C}_0,\mathbb{R}^p)
		\end{equation*}
		is also solution to $\mathcal{G}_0^a(\mathcal{V})=\Psi$, which implies that $\widehat{\mathcal{V}}=\mathcal{V}-\widetilde{\mathcal{V}}$ satisfies $\mathcal{G}_0^a(\widehat{\mathcal{V}})=0$. Then, $\widehat{\mathcal{V}}$ is exponentially decaying, that is, $\widehat{\mathcal{V}}\in C^{4,\zeta}_{1}(\mathcal{C}_0,\mathbb{R}^p)$. Finally, $\mathcal{V}\in C^{4,\zeta}_{\beta_2}(\mathcal{C}_0,\mathbb{R}^p)$ since $\mathcal{V}=\widehat{\mathcal{V}}+\widetilde{\mathcal{V}}$, which finishes the proof of (i). The proof of (ii) follows by the same argument, so that we omit it here.
	\end{proof}
	
	\subsection{Growth properties for the Jacobi fields}\label{subsec:growthproperties}
	In this part, we apply the spectral analysis developed before to investigate the growth/decay rate in which the Jacobi fields on the kernel of the linearized operator grow/decay. 
	This analysis is a fundamental part of the convergence technique we perform in the next section. 
	
	Let us begin with some considerations concerning the scalar case $p=1$.
	First, by Theorem \ref{thm:frank-konig19}, the operator \eqref{jacobiequationproj} has periodic coefficients. 
	Second, by Proposition~\ref{isolatedindicialroot}, we can use classical Floquet theory (or Boch wave theory) to study the asymptotic behavior of the Jacobi fields on the projection over $V_j$ (for more details, see \cite[Theorem~5.1]{MR0069338}). 
	To apply Floquet theory, we transform the fourth order operator \eqref{jacobiequationproj} into a first order operator on $\mathbb{R}^4$. 
	More precisely, defining $X=(\phi,\phi^{(1)},\phi^{(2)},\phi^{(3)})$, we conclude that the fourth order equation $\mathcal{L}_j^{a}(\phi)=0$ is equivalent to the first order system $X'(t)=N_{a,j,n}(t)X(t)$. Here
	\begin{equation*}
		N_{a,j,n}(t)=
		\left[{\begin{array}{cccc}
				0 & 1 & 0 & 0 \\
				0 & 0 & 1 & 0 \\
				0 & 0 & 0 & 1 \\
				0 & 0 & -B_{j,n} & +C_{a,j,n}(t)\\
		\end{array} } \right],
	\end{equation*}
	where 
	\begin{equation*}
		B_{j,n}:=K_2(n)+2\lambda\quad \mbox{and} \quad C_{a,j,n}(t)=K_0(n)+\lambda_j(\lambda_j+J_0)-\widetilde{c}(n)v_{a,T}(t)^{2^{**}-2}.
	\end{equation*}
	Notice that $N_{a,j,n}(t)$ is a $T_a$-periodic matrix.
	Hence, the {\it monodromy matrix} associate to this ODE system with periodic coefficients is given by $M_{a,j,n}(t)=\exp{\int_{0}^{t}{A}(\tau)d\tau}$. Finally, we define the {\it Floquet exponents}, denoted by $\widetilde{\mathfrak{I}}_j^a$,
	as the complex frequencies associated to the eigenvectors of $M_{a,j,n}(t)$, which forms a four-dimensional basis for the kernel of $\mathcal{L}^a_j$. 
	Using Abel's identity, we get that $N_{a,j,n}(t)$ is constant, which yields
	\begin{equation*}
		\det(M_{a,j,n}(t))=\exp{\int_{0}^{T}\trace{N}_{a,j,n}(\tau)d\tau}=\exp\left(-{\int_{0}^{T}C_{a,j,n}(\tau)d\tau}\right)=1.
	\end{equation*}
	Since $N_{a,j,n}(t)$ has real coefficients, all its eigenvalues are pairs of complex conjugates. 
	Equivalently, $\widetilde{\mathfrak{I}}_j^a=\{\pm\rho_{a,j},\pm\bar{\rho}_{a,j}\}$, where $\rho_{a,j}=\alpha_{a,j}+i\beta_{a,j}$ and $\widetilde{\rho}_{a,j}=\widetilde{\alpha}_{a,j}-i{\beta}_{a,j}$. Then, the set of indicial roots of $\mathcal{L}_j^{a}$ are given by $\mathfrak{I}_j^a=\{-\beta_{a,j},\beta_{a,j}\}$. Moreover, for any $\phi\in \ker(\mathcal{L}_j^{a})$, we have
	\begin{equation*}
		\phi(t)=b_1\phi_{a,j}^{+}(t)+b_2\phi_{a,j}^{-}(t)+b_3\widetilde{\phi}_{a,j}^{+}(t)+b_4\widetilde{\phi}_{a,j}^{-}(t),
	\end{equation*}
	where $\phi^{\pm}_{a,j}(t)=e^{\pm\rho_{a,j} t}$ and $\widetilde{\phi}_{a,j}^{\pm}(t)=e^{\pm\widetilde{\rho}_{a,j} t}$.
	Hence, the exponential decay/growth rate of the Jacobi fields is controlled by $|\beta_{a,j}|$. Therefore, the asymptotic properties of $\mathcal{L}_j^{a}$ are obtained by the study of $\mathfrak{I}^{a}$. Let us remember that in the definition of indicial roots, we are not considering the multiplicity, that is, $\beta_{a,2}$ stands for all the low-frequency ($j=1$) Jacobi fields.
	
	In the next lemma, we give some structure to the set of indicial roots of $\mathcal{L}^a$.
	
	\begin{lemma}\label{indicialset}
		For any $a\in[0,a_0]$, it follows\\
		\noindent{\rm (i)} If $j=0$, then $0\in\mathfrak{I}_{0}^a$.\\
		\noindent{\rm (ii)} If $j=1$, then $\{-1,1\}\subset\mathfrak{I}_{1}^a$.\\
		\noindent{\rm (iii)} If $j>1$, then $\min_{j>1}\mathfrak{I}_{j}^a>1$.\\
		Moreover, $\mathfrak{I}^a$ is a discrete set, namely,
		\begin{equation*}
			\mathfrak{I}^a=\{\dots,-\beta_{a,2},-1,0,1,-\beta_{a,2},\dots\}.
		\end{equation*}
		In particular, the indicial root $0$ is isolated. 
	\end{lemma}
	
	\begin{proof}
		Initially, let us divide the proof into three cases steps, namely, $a=0$, $a=a_0$ and $a\in(0,a_0)$.  In the first two ones, by Theorem~\ref{thm:frank-konig19}, we know that $v_{a,T}$ is constant; thus, the indicial exponents are the solutions to a fourth order characteristic equation. 
		In this fashion, let us also introduce the following notation for the discriminant of this indicial equation,
		\begin{equation*}
			D(a,n,j):=B_{j,n}^2-4C_{a,j,n}.
		\end{equation*}
		
		We also divide each step into three cases with respect the Fourier eigenmodes, namely $j=0$ (zero-frequency), $j=1$ (low-frequency) and $j>1$ (high-frequency). 
		
		\noindent{\bf Step 1:} (spherical solution) $a=0$ .
		
		\noindent When $a=0$, we have that $v_{a,T}\equiv0$, then the linearized operator becomes
		\begin{equation*}
			\mathcal{L}^{0}_j(\phi)=\phi^{(4)}-(K_2+2\lambda_j)\phi^{(2)}+\left(K_0+\lambda_j^2+\lambda_jJ_0\right)\phi.
		\end{equation*}
		Therefore, we shall compute the roots of the fourth order characteristic equation
		\begin{equation*}
			\rho^{4}-B_{j,n}\rho^{2}-C(0,j,n)\rho=0,
		\end{equation*}
		
		\noindent{\bf zero-frequency:} $j=0$, $\mathfrak{m}_0=1$ and $\lambda_0=0$.
		
		\noindent The operator has the following expression
		\begin{equation*}
			\mathcal{L}^{0}_0(\phi)=\phi^{(4)}-K_2\phi^{(2)}+K_0\phi.
		\end{equation*}
		Notice that when $D(a,n,j)>0$, then $\beta_{a,j}=0$. More generally, the sign of $D(a,n,j)$ controls the nature of the complex roots. It is straightforward to check that the indicial roots of this operator are 
		\begin{equation*}
			\rho_{0,0}=\frac{n}{2} \quad \mbox{and} \quad \widetilde{\rho}_{0,0}=\frac{n-4}{2}.
		\end{equation*}
		
		\noindent{\bf low-frequency:} $j=1$, $\mathfrak{m}_1=n$ and $\lambda_1=\cdots\lambda_n=n-1$.
		
		\noindent Here we obtain,
		\begin{equation*}
			\mathcal{L}^{0}_1(\phi)=\phi^{(4)}-B_{1,n}\phi^{(2)}+C(0,1,n)\phi,
		\end{equation*}
		and the indicial roots are given by 
		\begin{equation*}
			\rho_{0,1}=\frac{1}{2}(n+2) \quad \mbox{and} \quad \widetilde{\rho}_{0,1}=\frac{1}{2}(n-2).
		\end{equation*}
		
		\noindent{\bf high-frequency:} $j>1$, $\mathfrak{m}_j>n$ and $\lambda_j=\ell(n-2+\ell)$, for some $\ell\in\mathbb{N}$.
		
		\noindent Here we obtain,
		\begin{equation*}
			\mathcal{L}^{0}_j(\phi)=\phi^{(4)}-B_{j,n}\phi^{(2)}+C(0,j,n)\phi,
		\end{equation*}
		and the indicial roots are given by  
		\begin{equation*}
			\rho_{0,\ell}=\frac{1}{2}\left(2+\sqrt{D(0,\ell,n)}\right) \quad \mbox{and} \quad \frac{1}{2}\left(2-\sqrt{D(0,\ell,n)}\right),
		\end{equation*}
		where $D(0,\ell,n)=n^2-4n+4+4\ell(n+\ell-2)$. Notice that $D(0,\ell,n)>0$ for $\ell>1$. Using a direct argument, we can check that $\Im\rho_{0,j}>\Im\rho_{0,1}$ and $\Im\rho_{0,j}>\Im\widetilde{\rho}_{0,1}$ for all $j>1$, which by the last case concludes the proof of Step 1.
		
		\noindent{\bf Step 2:} (cylindrical solution) $a=a_0$.
		
		\noindent
		Since $v_{a_0,T}\equiv a_0=[n(n-4)/(n^2-4)]^{n-4/8}$, we proceed identically as in the last step. First, we have
		\begin{equation*}
			\mathcal{L}^{a_0}_j(\phi)=\phi^{(4)}-(K_2+2\lambda_j)\phi^{(2)}+\left(K_0+\lambda_j^2+\lambda_jJ_0-\widetilde{c}(n)a_0^{2^{**}-2}\right)\phi.
		\end{equation*}
		As before, we shall divide our approach as follows\\
		\noindent{\bf zero-frequency:} $j=0$, $\mathfrak{m}_0=1$ and $\lambda_0=0$.
		\begin{equation*}
			\rho_{a_0,0}=\frac{1}{2}\sqrt{n^2-4n+8+\sqrt{D(a_0,0,n)}} \quad \mbox{and} \quad \widetilde{\rho}_{a_0,0}=\frac{1}{2}\sqrt{n^2-4n+8-\sqrt{D(a_0,0,n)}},
		\end{equation*}
		where $D(a_0,0,n)=n^4-64n+64$.
		
		\noindent{\bf low-frequency:} $j=1$, $\mathfrak{m}_1=n$ and $\lambda_1=\cdots\lambda_n=n-1$.
		\begin{equation*}
			\rho_{a_0,1}=\sqrt{\frac{n^2+2}{2}} \quad \mbox{and} \quad \widetilde{\rho}_{a_0,1}=1.
		\end{equation*}
		
		\noindent{\bf high-frequency:} $j>1$, $\mathfrak{m}_j>n$ and $\lambda_j=\ell(n-2+\ell)$, for some $\ell\in\mathbb{N}$. 
		\begin{equation*}
			\rho_{a_0,j}=\frac{1}{4}\sqrt{{[(n+2(\ell-1))]^2}+\sqrt{D(a_0,\ell,n)}} \ \mbox{and} \ \widetilde{\rho}_{a_0,j}=\frac{1}{4}\sqrt{{[(n+2(\ell-1))]^2}-\sqrt{D(a_0,\ell,n)}},
		\end{equation*}
		where $D(a_0,\ell,n)=n^4+64(\ell-1)(n+\ell-1)$.
		
		\noindent{\bf Step 3:} (Emden--Fowler solution) $a\in(0,a_0)$.
		
		\noindent This is the most delicate case since $v_{a,T}$ is periodic, so the zeroth-order term in the operator $\mathcal{L}^a_j$ is also $T_a$-periodic. In this case, it is not possible to compute the Floquet exponents explicitly. Nonetheless, we can show that they are strictly bigger than one when $j>1$. 
		
		\noindent{\bf zero-frequency:} $j=0$, $\mathfrak{m}_0=1$ and $\lambda_0=0$ . 
		
		\noindent By Remark~\ref{geoemtricjacobifields}, it follows that $\phi^+_{a,0}$ is bounded and $\phi^-_{a,0}$ is linearly growing, then $0\in\mathfrak{I}^a$ with multiplicity $2$.
		
		\noindent{\bf low-frequency:} $j=1$, $\mathfrak{m}_1=n$ and $\lambda_1=\cdots\lambda_n=n-1$.
		
		\noindent Again using Remark~\ref{geoemtricjacobifields}, it follows  $\phi^\pm_{a,1}=\cdots=\phi^\pm_{a,n}$ is exponentially growing/decaying, then $\{-1,1\}\subset\mathfrak{I}^a$.
		
		\noindent{\bf high-frequency:} More generally,
		note that the indicial roots form an increasing sequence,
		\begin{equation*}
			\beta_{a,0}\leqslant \beta_{a,1}\leqslant\cdots\leqslant \beta_{a,j}\leqslant \beta_{a,j+1}\rightarrow\infty \quad \mbox{as} \quad j\rightarrow\infty,
		\end{equation*}
		which is a consequence of a comparison principle on $a\in(0,a_0)$ for the linearized operator in cylindrical coordinates.
	\end{proof}
	
	\begin{lemma}
		The indicial set $\mathfrak{I}_{j}^a$ is a discrete. Moreover,
		\begin{equation*}
			\mathfrak{I}_{j}^a=\{\dots,-\beta_{a,2},-1,0,1,\beta_{a,2},\dots\}.
		\end{equation*}
		In particular, the indicial root $0$ is isolated. 
	\end{lemma}
	
	\begin{proof}
		It follows directly by Proposition~\ref{isolatedindicialroot}.
	\end{proof}
	
	Notice that Lemma~\ref{indicialset} only provides exponential growth/decay for the Jacobi fields. Nevertheless, to apply Simon's technique, we need something slightly more robust. Namely, for $j=0$, we must show that the Jacobi fields are either periodic (bounded) or linearly growing. For the first two Jacobi fields $\phi^+_{a,0}$ and $\phi^-_{a,0}$, this follows because they arise respectively as the variation of the necksize and translation parameters that appear in the classification for the Emden--Fowler solutions. The difficulty here is to show that they generate the zero-frequency space.
	We overcome this issue, observing that by the direct computation in Lemma~\ref{indicialset}, we know that $0\in\mathfrak{I}^a_0$ with multiplicity two.  
	Next, we proceed as in \cite[Proposition~4.14]{MR1712628} to prove the following asymptotic expansion
	
	\begin{proposition}
		Let $\psi\in\ C_{0}^{\infty}\left(\mathcal{C}_0\right)$, $\beta\in(0,1)$ and $\phi\in H_{-\beta}^{4}\left(\mathcal{C}_0\right)$ satisfying $\mathcal{L}^a(\phi)=\psi$. Then, $\phi$ has an asymptotic expansion $\phi=\sum_{j\in\mathbb{N}}\phi_{j}$ with $\mathcal{L}^a\left(\phi_{j}\right)=0$ and $\phi_{j} \in H_{-\beta}^{4}\left(\mathcal{C}_0\right)$ for any $\beta<\beta_{a,j}$.
	\end{proposition}    
	
	\begin{proof}
		We divide the proof into some steps.
		
		\noindent{\bf Step 1:} For $\beta<1$, it follows $\phi\in H_{-\beta}^{k+4}\left(\mathcal{C}_0\right)$.
		
		\noindent Indeed, take $\rho\in\mathbb{C}$ with $0<\beta<\Im(\rho)$ and consider the transformed equation
		\begin{equation}\label{inversedequation}
			\widetilde{\mathcal{L}}^a(\rho)(\widehat{\phi})=\widetilde{\psi},
		\end{equation}
		where $\widetilde{\phi}=e^{i\rho t}\widehat{\phi}$ and $\widetilde{\psi}=e^{i\rho t}\widehat{\psi}$.
		By applying the inverse operator $\widetilde{\mathcal{G}}^a(\rho)$ in both sides of \eqref{inversedequation}, we get that $\widehat{\phi}=\widetilde{\mathcal{L}}^a(\rho)(\widetilde{\psi})$.
		Then, since $\psi\in C^{\infty}_0(\mathcal{C}_0)$, it follows that $\widetilde{\psi}(\rho)$ is an entire function on $\rho$ and smooth on $(t,\theta)$. Notice that $\widetilde{\phi}$ is analytic on the half-plane $\Im(\rho)>\beta$, since the poles of $\mathcal{G}^a(\rho)$ coincide with the zeros of $\widetilde{\psi}$. Finally, take $\beta^{\prime}\in(\beta,1)$ and since $\mathcal{G}^a(\rho)$ has no poles in $\Im(\rho)\in(\beta^{\prime},\beta)$, by the Cauchy formula, we can define the contour integral $\mathcal{F}_a^{-1}$ up to $\Im(\rho)$.
		
		\noindent{\bf Step 2:} For each $\beta\in(0,1)$, there exist $\beta^{\prime\prime}\in(1,\beta_{a,2})$, $\phi^{\prime\prime}\in H_{-\beta^{\prime\prime}}^{k+4}\left(\mathcal{C}_0\right)$, and $\phi^\prime\in H_{-\beta}^{k+4}\left(\mathcal{C}_0\right)$ with $\mathcal{L}^a(\phi^{\prime})=0$ satisfying $\phi=\phi^{\prime}+\phi^{\prime\prime}$.
		
		\noindent Choose $\beta^{\prime\prime}\in(1,\beta_{a,2})$ and $\rho^{\prime\prime}$ such that $\Im(\rho^{\prime\prime})=\beta^{\prime\prime}$. Now let us define $\widetilde{\phi}^{\prime\prime}=\widetilde{\mathcal{G}}^a(\rho^{\prime\prime})$. Finally, we apply the inverse $\mathcal{F}^{-1}_a$ on the two contour lines $\Im(\rho)=\beta$ and $\Im(\rho)=\beta^{\prime\prime}$, which by periodicity does not takes into account the vertical sides of the rectangle $[\beta,\beta^{\prime\prime}]\times[0,2\pi]\subset\mathbb{C}$. In fact, $\widetilde{\phi}-\widetilde{\phi}^{\prime\prime}=\widetilde{\mathcal{G}}^a(\rho)-\widetilde{\mathcal{G}}^a(\rho^{\prime\prime})$ is the residue of a meromorphic function with pole at $-i$.
		
		We can continue this process by shifting the contour integral to the other poles in the strip.
	\end{proof}
	
	\begin{corollary}\label{cor:asymptoticexpansion}
		Let $\beta\in(0,1)$, $\psi\in C_{c}^{\infty}\left(\mathcal{C}_0\right)\cap L_{-\beta}^{2}\left(\mathcal{C}_0\right)$ and
		$\phi\in H_{-\beta}^{4}\left(\mathcal{C}_0\right)$ satisfying $\mathcal{L}^a(\phi)=\psi$. 
		Then, there exist $\phi^{\prime}\in H_{-\beta}^{4}\left(\mathcal{C}_0\right)$ and
		$\phi^{\prime\prime}\in D_{a,0}\left(\mathcal{C}_0\right)$ such that $\phi=\phi^{\prime}+\phi^{\prime\prime}$.
	\end{corollary}
	
	\begin{corollary}\label{scalargrowthproperties}
		The following properties hold for the projected scalar linearized operator:\\
		\noindent{\rm (i)} Assume $j=0$, then the homogeneous equation $\mathcal{L}^a_0(\phi)=0$ has a solutions basis with two elements, which are either bounded or at most linearly growing as $t\rightarrow\infty$;\\
		\noindent{\rm (ii)} Assume $j\geqslant1$, then the homogeneous equation $\mathcal{L}^a_j(\phi)=0$ has a solutions basis with four elements, which are exponentially growing/decaying as $t\rightarrow\infty$.
	\end{corollary}
	
	\begin{proof}
		For (i), we use Corollary~\ref{cor:asymptoticexpansion}.  
		Notice that (ii) follows directly from Lemma~\ref{indicialset}.
	\end{proof}
	
	When $p>1$, we can use a similar strategy as before to study the solutions to \eqref{jacobisystem}. For $p=1$, we have constructed a Jacobi field basis with four elements (two in the zero-frequency case). Now we must find a base with $4p$ elements ($2p$ in the zero-frequency case), sharing the same growth properties in Corollary~\ref{scalargrowthproperties}. 
	
	\begin{proof}[Proof of Proposition~\ref{growthpropertiessystem}]
		First, notice that by Theorem \ref{thm:andrade-doo19}, there exist $a\in[0,a_0]$, $T\in[0,T_a]$ and $\Lambda^*\in\mathbb{S}_+^{p-1}$ only provides $p+1$ families of solutions given by 
		\begin{equation*}
			T\mapsto\Lambda^* v_{a}(t+T), \quad a\mapsto\Lambda^* v_{a,T}(t), \quad  \mbox{and} \quad \theta\mapsto\Lambda^*(\theta)v_{a,T}(t),
		\end{equation*} 
		which by differentiation gives rise to some elements of the basis. 
		\begin{equation*}
			\Lambda^*\partial_T\big|_{T=0} v_{a,T}(t), \quad \Lambda^*\frac{1}{a}\partial_a\big|_{a=0} v_{a,T}(t), \quad \mbox{and} \quad \partial_{\theta_i}\Lambda^*(\theta)v_{a,T}(t),
		\end{equation*}
		for $i=1,\dots,p-1$. 
		Second, to construct all the Jacobi fields basis, let us consider $\{{\bf e}_i\}_{i\in I}\subset\mathbb{S}^{p-1}$ a linearly independent set in $\mathbb{R}^p$ with ${\bf e}_1=\Lambda^*$, which provides four families with $p$ Jacobi fields each given by
		\begin{equation*}
			\Phi^{+}_{a,j,i}={\bf e}_i\phi^+_{a,j}, \quad \Phi^{-}_{a,j,i}={\bf e}_i\phi^-_{a,j}, \quad \widetilde{\Phi}^{+}_{a,j,i}={\bf e}_i\widetilde{\phi}^+_{a,j}, \quad \mbox{and} \quad \widetilde{\Phi}^-_{a,j,i}={\bf e}_i\widetilde{\phi}^-_{a,j}.
		\end{equation*}
		Then, using Theorem~\ref{thm:andrade-doo19}, it is easy to check that $\mathcal{B}^a_j=\cup_{i\in I}\{\Phi^{\pm}_{a,j,i},\widetilde{\Phi}^{\pm}_{a,j,i}\}$ is a basis to the kernel of $\mathcal{L}^a_j$ with $4p$ elements for each $j\geqslant1$, and $\mathcal{B}^a_0=\cup_{i\in I}\{\Phi^{\pm}_{a,0,i}\}$ a basis with $2p$ elements when $j=0$.
	\end{proof}
	
	\section{Qualitative properties and a priori estimates}\label{sec:qualitativeproperties}
	This section is devoted to prove Proposition~\ref{estimates}. More precisely, we show that solutions to \eqref{oursystem} are asymptotic radially symmetric and satisfy an upper and lower bound estimate near the isolated singularity. 
	Our strategy is to convert our system of differential equations into a system of integral equations.
	Then, we use the Kelvin transform to perform a moving sphere technique and the Pohozaev invariant together with a barrier construction to obtain the a priori estimates.
	Here we are inspired in some techniques from \cite{arxiv:1901.01678} (see also \cite{MR3626036,MR3694645,arXiv:1810.02752v6}).
	The main difference to the proofs in \cite{arxiv:1901.01678} is that we need to deal with many components of System \eqref{oursystem}, coupled by the Gross--Pitaevskii nonlinearity.
	
	\subsection{Integral representation formulas}
	Now we use a Green identity to transform the fourth order differential system  \eqref{oursystem} into an integral system.
	In this way, we can avoid using the classical form of the maximum principle, and a sliding method is available \cite{arxiv:1901.01678,MR2055032,MR3558255}, which will be used to classify solutions. 	
	Besides, in this setting it is also possible to prove regularity through a barrier construction.
	
	For $n\geqslant3$, the following expression for the Green function of the Laplacian in the unit ball is well-known.
	\begin{equation*}
		G_{1}(x, y)=\frac{1}{(n-2) \omega_{n-1}}\left(|x-y|^{2-n}-\left|\frac{x}{|x|}- x|y|\right|^{2-n}\right),
	\end{equation*}
	where $\omega_{n-1}$ is the surface area of the Euclidean unit sphere. 
	In addition, for any $u \in C^{2}\left(B_{1}\right) \cap C\left(\bar{B}_{1}\right)$, the next decomposition holds,
	\begin{equation*}
		u(x)=-\int_{B_{1}} G_{1}(x, y)\Delta u(y) \mathrm{d} y+\int_{\partial B_{1}} H_{1}(x, y) u(y) \mathrm{d} \sigma_{y},
	\end{equation*}
	where
	\begin{equation*}
		H_{1}(x, y)=-{\partial_{v_{y}}} G_{1}(x, y)=\frac{1-|x|^{2}}{\omega_{n-1}|x-y|^{n}} \quad \mbox{for} \quad x \in B_{1} \quad \mbox{and} \quad y \in \partial B_{1},
	\end{equation*}
	with ${v_{y}}$ the outward normal vector at $y$.
	
	Similarly, in the fourth order case with $n\geqslant5$, for any $u\in C^{4}\left(B_{1}\right)\cap C^{2}\left(\bar{B}_{1}\right)$, it follows
	\begin{equation*}
		u(x)=\int_{B_{1}} G_{2}(x, y)\Delta^{2} u(y) \ud y+\int_{\partial B_{1}} H_{1}(x, y)u(y) \ud \sigma_{y}-\int_{\partial B_{1}} H_{2}(x, y)\Delta u(y) \ud \sigma_{y},
	\end{equation*}
	where
	\begin{equation*}
		G_{2}(x, y)=\int_{B_{1} \times B_{1}} G_{1}\left(x, y_{1}\right) G_{1}\left(y_{1}, y\right) \ud y_{1}
	\end{equation*}
	and
	\begin{equation*}
		H_{2}(x, y)=\int_{B_{1} \times B_{1}} G_{1}\left(x, y_{1}\right) H_{1}\left(y_1, y\right) \mathrm{d} y_{1}.
	\end{equation*}
	By a direct computation, we have
	\begin{equation}\label{greenfunction}
		G_{2}(x, y)=C(n, 2)|x-y|^{4-n}-A(x, y),
	\end{equation}
	where $C(n, 2)=\frac{\Gamma(n-4)}{2^{4} \pi^{n / 2} \Gamma(2)}$, $A:B_1\times B_1\rightarrow\mathbb{R}$ is a smooth map and $H_{i}(x, y) \geqslant 0$ for $i=1,2$.
	
	In the next lemma, we provide 
	basic integrability for singular solutions to \eqref{oursystem} when $R<\infty$ and $s>1$, which yields a weaker formulation to \eqref{oursystem}. Here, the proof is similar to the one in \cite[Lemma~3.1]{MR4123335}. Nevertheless, we included the proof for the sake of completeness.
	
	\begin{lemma}\label{lm:integrability}
		Let $s\in(1,\infty)$ and $\mathcal{U}$ be a nonnegative singular solution to \eqref{oursystem}. Then, $\mathcal{U} \in L^{s}\left(B_{1},\mathbb{R}^{p}\right)$. 
		In particular, if $s\in(2_{**},\infty)$, then $\mathcal{U}$ is a
		distribution solution to \begin{equation}\label{subcriticalsystem}
			\Delta^2u_i=f^s_i(\mathcal{U}) \quad {\rm in} \quad B^*_{1},
		\end{equation}
		where $f^s_i(\mathcal{U}):=c(n,s)|\mathcal{U}|^{s-1}u_i$.
	\end{lemma}
	
	\begin{proof}
		For any $0<\varepsilon \ll 1$, let us consider $\eta_{\varepsilon}\in C^{\infty}\left(\mathbb{R}^{n}\right)$ with $0\leqslant\eta_{\varepsilon}\leqslant1$ satisfying
		\begin{equation}\label{cutoff}
			\eta_{\varepsilon}(x)=
			\begin{cases}
				0, & \mbox{if} \ |x| \leqslant \varepsilon\\
				1, & \mbox{if} \ |x| \geqslant 2 \varepsilon,
			\end{cases}
		\end{equation}
		and $|D^{(j)} \eta_{\varepsilon}(x)| \leqslant C \varepsilon^{-j}$ for $j\geqslant1$.
		Define $\xi_{\varepsilon}=\left(\eta_{\varepsilon}\right)^{m}$, where  $m=s\gamma(s)$. Multiplying \eqref{oursystem} by $\xi_{\varepsilon}$, and integrating by parts in $B_{r}$ with $r\in(1/2,1)$, we obtain
		\begin{equation*}
			\int_{B_{r}} |\mathcal{U}|^{s-1}u_i \xi_{\varepsilon}\ud x =\int_{\partial B_{r}} {\partial_\nu} \Delta u_i\ud\sigma_r+\int_{B_{r}} u_i \Delta^{2} \xi_{\varepsilon}\ud x \quad \mbox{for all} \quad i\in I.
		\end{equation*}
		On the other hand, there exists $C>0$ such that
		\begin{equation*}
			\left|\Delta^{2} \xi_{\varepsilon}\right| \leqslant C \varepsilon^{-4}\left(\eta_{\varepsilon}\right)^{m-4} \chi_{\{\varepsilon \leqslant|x| \leqslant 2 \varepsilon\}}=C \varepsilon^{-4}\left(\xi_{\varepsilon}\right)^{1/s} \chi_{\{\varepsilon \leqslant|x| \leqslant 2 \varepsilon\}},
		\end{equation*}
		which, by H\"{o}lder's inequality, gives us
		\begin{align*}
			\left|\int_{B_{r}} u_i \Delta^{2} \xi_{\varepsilon}\ud x\right|
			&\leqslant C \varepsilon^{-4} \int_{\{\varepsilon \leqslant|x| \leqslant 2 \varepsilon\}} u_i\xi_{\varepsilon}^{1/s}\ud x \\
			& \leqslant C \varepsilon^{-4} \varepsilon^{n(1-1/s)}\left(\int_{\{\varepsilon \leqslant|x| \leqslant 2 \varepsilon\}} |\mathcal{U}|^{s-1}u_i \xi_{\varepsilon}\ud x\right)^{1/s} \\
			& \leqslant C\left(\int_{\{\varepsilon \leqslant|x| \leqslant 2 \varepsilon\}} |\mathcal{U}|^{s-1}u_i \xi_{\varepsilon}\ud x\right)^{1/s}.
		\end{align*}
		Thus, it follows
		\begin{equation*}
			\int_{B_{r}} |\mathcal{U}|^{s-1}u_i \xi_{\varepsilon}\ud x \leqslant \int_{\partial B_{r}} {\partial_\nu} \Delta u_i\ud\sigma_r+C\left(\int_{\{\varepsilon \leqslant|x| \leqslant 2 \varepsilon\}} |\mathcal{U}|^{s-1}u_i \xi_{\varepsilon}\ud x\right)^{1/s},
		\end{equation*}
		which provides a constant $C>0$ (independent of $\varepsilon$) such that
		\begin{equation*}
			\int_{B_{r}}|\mathcal{U}|^{s-1}u_i \xi_{\varepsilon}\ud x \leqslant C.
		\end{equation*}
		Now letting $\varepsilon \rightarrow 0$, since $u_i\leqslant|\mathcal{U}|$, we conclude that $u_i \in L^{s}\left(B_{r}\right)$ for all $i\in I$ and the integrability follows.
		
		We are left to show that $\mathcal{U}$ is a distribution solution to \eqref{oursystem}. For any $\Phi\in$ $C_{c}^{\infty}\left(B_{1},\mathbb{R}^p\right)$, we multiply \eqref{oursystem} by $\widetilde{\Phi}=\eta_{\varepsilon} \Phi$, where $\eta_{\varepsilon}$ is given by \eqref{cutoff}. Then, using that $|\mathcal{U}|\in L^{s}\left(B_{r}\right)$ and integrating by parts twice, we get
		\begin{equation}\label{yang1}
			\int_{B_{1}} \langle \mathcal{U},\Delta^{2}\left(\eta_{\varepsilon}\Phi\right)\rangle\ud x=\int_{B_{1}} \langle|\mathcal{U}|^{s-1}\mathcal{U},\eta_{\varepsilon}\Phi\rangle\ud x.
		\end{equation}
		By a direct computation, we find that $\Delta^{2}\left(\eta_{\varepsilon} \phi_i\right)=\eta_{\varepsilon} \Delta^{2}\phi_i+\varsigma^{\varepsilon}_i$, where
		\begin{equation*}
			\varsigma^{\varepsilon}_i=4 \langle\nabla \eta_{\varepsilon},\nabla \Delta\phi_i\rangle+2\Delta \eta_{\varepsilon} \Delta\phi_i+4 \Delta\eta_{\varepsilon}\Delta\phi_i+4\langle\nabla \Delta \eta_{\varepsilon},\nabla \phi_i\rangle+\phi_i\Delta^{2} \eta_{\varepsilon}.
		\end{equation*}
		Furthermore, using H\"{o}lder's inequality again, we find
		\begin{align*}
			\left|\int_{B_{1}} \langle \mathcal{U},\Psi_{\varepsilon}\rangle\ud x\right| \leqslant C\left(\int_{\{\varepsilon\leqslant|x| \leqslant 2 \varepsilon\}} |\mathcal{U}|^{s-1}u_i\ud x\right)^{1/s} \leqslant C\left(\int_{\{\varepsilon\leqslant|x| \leqslant 2 \varepsilon\}} |\mathcal{U}|^{s}\ud x\right)^{1/s} \rightarrow 0 
		\end{align*}
		as $\varepsilon \rightarrow 0$, where $\Psi_{\varepsilon}=(\varsigma^{\varepsilon}_1,\dots,\varsigma^{\varepsilon}_p)\in C_c^{\infty}(B_1,\mathbb{R}^p)$.
		Finally, letting $\varepsilon \rightarrow 0$ in \eqref{yang1}, and applying the dominated convergence theorem the proof follows.
	\end{proof}
	
	In the following lemma, we employ some ideas due to L. Caffarelli et al. \cite{MR982351} and L. Sun and J. Xiong \cite{MR3558255}. 
	
	\begin{lemma}\label{integrability}
		Let $s\in(2_{**},\infty)$ and  $\mathcal{U}\in C^{4}(\bar{B}_1^*,\mathbb{R}^p)\cap L^{1}(B_{1},\mathbb{R}^p)$ be a nonnegative singular
		solution to \eqref{subcriticalsystem}.
		 Then, $|x|^{-q} u_i^{s}\in L^{1}\left(B_{1}\right)$ for any $q<n-{4s}/{(s-1)}$. Moreover,
		\begin{equation*}
			u_i(x)=\int_{B_{1}} G_{2}(x, y)\Delta^{2} u_i(y) \ud y+\int_{\partial B_{1}} H_{1}(x, y)u_i(y) \ud \sigma_{y}-\int_{\partial B_{1}} H_{2}(x, y)\Delta u_i(y) \ud \sigma_{y}.
		\end{equation*}
	\end{lemma}
	
	\begin{proof}
		First, by Lemma~\ref{lm:integrability}, we have that $|\mathcal{U}|\in L^{s}\left(B_{1},\mathbb{R}^{p}\right)$ and $\mathcal{U}$ is a distribuitional solution in $B_1$.
		Let $\eta\in C^{\infty}(\mathbb{R})$ such that $\eta(t)=0$ for $t\leqslant 1$, $\eta(t)=1$ for $t \geqslant 2$ and $0 \leqslant\eta\leqslant 1$ for $1\leqslant t\leqslant 2$. For small $\varepsilon>0$, plugging $\eta\left(\varepsilon^{-1}|x|\right)|x|^{4-q}$ into \eqref{subcriticalsystem}, and using integration by parts, it holds
		\begin{equation}\label{integrality1}
			\int_{B_{1}} f^s_i(\mathcal{U})\eta\left(\varepsilon^{-1}|x|\right)|x|^{4-q}\ud x=\int_{B_{1}} u_i\Delta^{2}\left(\eta\left(\varepsilon^{-1}|x|\right)|x|^{4-q}\right)\ud x+\int_{\partial B_{1}} G(u_i) \ud \sigma_y,
		\end{equation}
		where $G(u_i)$ involves $u_i$ and its derivatives up to third order. Taking $q=q_{0}=0$ in \eqref{integrality1}, we have that  $\int_{B_{1}} u_i^{s}|x|^{4{\left(-q_{1}-4/s\right)}}\ud x<\infty$ as $\varepsilon \rightarrow 0$, since by hypothesis $u_i \in L^{1}\left(B_{1}\right)$ for all $i\in I$. Moreover, by H\"{o}lder inequality, if $0<q_{1}<[n(s-1)-4]/s$, we have
		\begin{align*}
			\int_{B_{1}} u_i|x|^{-q_{1}}\ud x
			&=\int_{B_{1}} u_i|x|^{4/s}|x|^{\left(-q_{1}-4/s\right)} \ud x\leqslant\left(\int_{B_{1}} u_i^{s}|x|^{4}\ud x\right)^{1/s}\left(\int_{B_{1}}|x|^{-\frac{4+s q_{1}}{s-1}}\ud x\right)^{(s-1)/s}<\infty.&
		\end{align*}
		Using that $q=q_{1}$ in \eqref{integrality1} and taking the limit $\varepsilon\rightarrow0$, we obtain $\int_{B_{1}} u^{s}|x|^{4-q_{1}}\ud x<\infty$. By the same argument, we find  $\int_{B_{1}} u|x|^{-q_{2}}\ud x<\infty$, if $q_{2}<[n(s-1)-4](s^{-1}+s^{-2})$.
		Iterating this procedure, we get
		\begin{equation*}
			\int_{B_{1}} u_i|x|^{-q_{k}}\ud x<\infty \quad \mbox{and} \quad \int_{B_{1}} u_i^{s}|x|^{4-q_{k}}\ud x<\infty,
		\end{equation*}
		if $0<q_{k}:=[n(s-1)-4]\left(\sum_{i=1}^{k}s^{-i}\right)$ for $k\in\mathbb{N}$, which proves the first statement.
		
		Next, let us consider
		\begin{equation*}
			\widehat{u}_i(x)=\int_{B_{1}} G_{2}(x, y)\Delta^{2} u_i(y) \ud y+\int_{\partial B_{1}} H_{1}(x, y)u_i(y) \ud \sigma_{y}-\int_{\partial B_{1}} H_{2}(x, y)\Delta u_i(y) \ud \sigma_{y}.
		\end{equation*}
		and $h_i=u_i-\widehat{u}_i$ for $i\in I$ . 
		Thus, $\Delta^{2} h_i=0$ in $B^*_{1}$ because $u_i^{s} \in L^{1}\left(B_{1}\right)$ and $\widehat{u}_i\in L_{{\rm weak}}^{2^{**}/2}\left(B_{1}\right)\cap L^{1}\left(B_{1}\right)$. 
		Furthermore, since the Riesz potential $|x|^{4-n}$ is of weak type $\left(1,2^{**}/2\right)$ (see \cite[Chapter~9]{MR1814364}), for every $0<\varepsilon\ll1$, we can choose $\rho>0$
		such that $\int_{B_{2\rho}} f_i(\mathcal{U}(x))\ud x<\varepsilon$, which implies that for all $M\gg1$,
		\begin{align*}
			&\mathcal{L}^n\left(\{x \in B_{\rho} : |\widehat{u}_i(x)|>M\}\right)&\\ &\leqslant\mathcal{L}^n\left(\left\{x \in B_{\rho} : \int_{B_{2\rho}} G_{2}(x, y) f_i(\mathcal{U}(y))\ud y>\mu / 2\right\}\right)\leqslant\varepsilon M^{-2^{**}/2}.&
		\end{align*}
		Hence, $h_i\in L_{{\rm weak}}^{2^{**}/2}\left(B_{1}\right)\cap L^{1}\left(B_{1}\right)$ for all $i\in I$ and for every $0<\varepsilon\ll1$, there exists $\rho>0$ such that for all
		sufficiently large $\mu\gg1$, it holds
		\begin{align*}
			&\mathcal{L}^n\left(\{x \in B_{\rho}:|h_i(x)|>M\}\right)&\\
			&\leqslant\mathcal{L}^n\left(\{x \in B_{\rho}:|u_i(x)|>M/2|\}\right)+\mathcal{L}^n\left(\{x \in B_{\rho}:|\widehat{u}_i(x)|>M/2\}\right)\leqslant\varepsilon M
			^{-2^{**}/2}.&
		\end{align*}
		Using the B\^ocher theorem for biharmonic functions from \cite{MR1820695}, we obtain that $\Delta^2 h_i=0$ in $B_{1}$. Finally, since $h_i=\Delta h_i=0$ on $\partial B_{1}$, by the maximum principle, we have that $h_i \equiv 0$, and thus $u_i=\widehat{u}_i$, which finishes the proof of the second part.
	\end{proof}
	
	\begin{remark}
		For any $\mathcal{U} \in C^{2}\left(B^*_{1},\mathbb{R}^p\right)$ be a nonnegative superharmonic $p$-map, it follows $\mathcal{U}\in L^{1}\left(B_{1/2},\mathbb{R}^p\right)$. Indeed, for $0<r<1$ and $i\in I$, we have $\left(r^{n-1} \overline{u}_i^{(1)}\right)^{(1)}\leqslant0$,
		where $\overline{u}_i(r)=\avint_{\partial B_{1}} u_i(r \theta)\ud\theta$, which by a direct integration implies that
		\begin{equation*}
			\overline{u}_i(r)\leqslant C\left(r^{2-n}+1\right).
		\end{equation*} 
	\end{remark}
	
	In the next proposition, we use the Green identity to convert \eqref{oursystem} into an integral system, which is the main result of this section.
	Here the superharmonicity condition is assumed.
	
	\begin{proposition}\label{lm:integralrepresentation}
		Let $s\in(1,2^{**}-1]$ and $\mathcal{U}$ be a nonnegative superharmonic solution to \eqref{oursystem}. Then, there exists $r_0>0$ such that
		\begin{equation}\label{integralsystem}
			u_i(x)=\int_{B_{r_0}}|x-y|^{4-n}f^s_i(\mathcal{U}) \ud y+\psi_i(x) \quad {\rm in} \quad B_R^*,
		\end{equation}
		where $\psi_i>0$ satisfies $\Delta^{2} \psi_i=0$ in $B_{r_0}$. Moreover, one can find a constant $C(\widetilde{r})>0$ such that
		\begin{equation}\label{logestimate}
			\|\nabla \ln \psi_i\|_{C^{0}\left(B_{\widetilde{r}}\right)} \leqslant C\left(\widetilde{r}\right) \quad {\rm for \ all} \quad i\in I \quad {\rm and} \quad 0<\widetilde{r}<r_0.
		\end{equation}  
	\end{proposition}
	
	\begin{proof}
		Using that $-\Delta u_i>0$ in $B^*_{1}$ and $u_i>0$ in $\bar{B}_{1}$, it follows from the maximum principle that $c_{1i}:=\inf _{B_{1}} u_i=\min _{\partial B_{1}} u_i>0 $. In addition, by Lemma~\ref{integrability}, we get that $f^s_i(\mathcal{U}) \in L^{1}\left(B_{1}\right)$, which implies that there exists $r_0<1/4$ satisfying the following inequality
		\begin{equation*}
			\int_{B_{r_0}}|A(x, y)|f^s_i(\mathcal{U})\ud y\leqslant\frac{c_{1i}}{2} \quad \mbox{for} \quad x\in B_{r_0},
		\end{equation*}
		where $A(x, y)$ is given by \eqref{greenfunction}. Hence, for $x \in B_{r_0}$, we get
		\begin{align*} 
			\psi_i(x)=-& \int_{B_{r_0}} A(x, y)f^s_i(\mathcal{U})\ud y+\int_{B_{1}\setminus B_{r_0}} G_{2}(x, y) f^s_i(\mathcal{U}) \mathrm{d} y&\\
			&+\int_{\partial B_{1}} H_{1}(x, y)u_i(y)\ud\sigma_{y}-\int_{\partial B_{1}} H_{2}(x, y)\Delta u_i(y)\ud\sigma_{y}&\\ 
			&\geqslant-\frac{c_{1i}}{2}+\int_{\partial B_{1}} H_{1}(x, y) u_i(y) \ud\sigma_{y}&\\
			&\geqslant-\frac{c_{1i}}{2}+\inf _{B_{1}} u_i=\frac{c_{1i}}{2}.
		\end{align*}
		By hypothesis $\psi_i$ is biharmonic, then a removable singularity theorem, and elliptic regularity shows that $\psi_i\in C^{\infty}(B_{r_0})$ for all $i\in I$, which provides that $|\nabla \psi_i|\leqslant C_i(\widetilde{r})$ in $B_{\widetilde{r}}$ for all $\widetilde{r}<r_0$ and $i\in I$, where $C_i\left(\widetilde{r}\right)>0$ depends only on $n,r_0-\widetilde{r}$, and in the $L^{1}$ norm of $f^s_i(\mathcal{U})$. Consequently,
		\begin{equation*}
			\|\nabla \ln \psi_i\|_{C^{0}\left(B_{\widetilde{r}}\right)} \leqslant 2\frac{C\left(\widetilde{r}\right)}{c_{1i}} \quad \mbox{for} \quad i\in I,
		\end{equation*}
		which finishes the proof.
	\end{proof}
	
	\subsection{Moving spheres technique and the upper bound estimate}
	The objective here is to prove the upper bound estimate in Proposition~\ref{estimates}. In this fashion, we use the integral form of the moving spheres technique combined with a classical blow-up argument from \cite{MR2055032}. 
	First, we recall the result from \cite[Lemma 3.1]{arxiv:1901.01678}. 
	
	\begin{lemma}\label{movingpsheresineq}
		Let $s=2^{**}-1$ and $\mathcal{U}$ be a strongly positive solution to \eqref{integralsystem}. 
		For any $x \in B_{1}$, $z \in B_{2}\setminus\left(\{0\}\cup B_{\mu}(x)\right)$ and $\mu<1$, it follows that $u_i(z)-(u_i)_{x, \mu}(z)>0$ for $i\in I$.
	\end{lemma}
	
	\begin{proof}
		If $\mathcal{U}$ is a strongly positive solution to \eqref{integralsystem}, then, replacing $u_i(x)$ by $r^{\gamma}u_i(r x)$ for $r={1}/{2}$, we may consider the equation defined in $B^*_{2}$ for convenience, 
		\begin{equation}\label{rescaledintegralsystem}
			u_i(x)=\int_{B_{2}}|x-y|^{4-n}f_i(\mathcal{U}(y))\ud y+\psi_i(x) \quad \mbox{in} \quad B^*_{2}
		\end{equation}
		such that $u_i\in C\left(B^*_{2}\right) \cap L^{\frac{n+4}{n-4}}\left(B_{2}\right)$ and    $|\nabla \ln \psi_i| \leqslant C_{0} $ in $B_{3/2}$. Extending $u_i$ to be identically zero outside $B_{2}$, we have
		\begin{equation*}
			u_i(x)=\int_{\mathbb{R}^n} |x-y|^{4-n}f_i(\mathcal{U}(y)) \ud y+\psi_i(x) \quad \mbox{in} \quad B^*_{2}.
		\end{equation*}
		Using the identities in \cite[page 162]{MR2055032}, one has 
		\begin{equation}\label{identity1}
			\left(\frac{\mu}{|z-x|}\right)^{n-4} \int_{|y-x| \geqslant\mu} {\left|\mathcal{I}_{x,\mu}(z)-y\right|^{n-4}}{f_i(\mathcal{U}(y))} \ud y=\int_{|y-x| \leqslant \mu} \left|z-y\right|^{n-4}{f_i(\mathcal{U}(z))}\ud y
		\end{equation}
		and
		\begin{equation}\label{identity2}
			\left(\frac{\mu}{|z-x|}\right)^{n-4} \int_{|y-x|\leqslant\mu} {\left|\mathcal{I}_{x,\mu}(z)-y\right|^{n-4}}{f_i(\mathcal{U}(y))} \ud y=\int_{|y-x|\geqslant \mu} \left|z-y\right|^{n-4}{f_i(\mathcal{U}(y))}\ud y,
		\end{equation}
		which implies
		\begin{equation*}
			{(u_i)}_{x, \mu}(z)=\int_{\mathbb{R}^{n}} \left|z-y\right|^{n-4}{f_i(\mathcal{U}(y))}\ud y+{(\psi_i)}_{x,\mu}(z) \quad \mbox{for} \quad z\in \mathcal{I}_{x,\mu}(B_{2}).
		\end{equation*}
		Consequently, for any $x \in B_{1}$ and $\mu<1$ we have 
		\begin{equation*}
			u_i(z)-(u_i)_{x, \mu}(z)=\int_{|y-x|\geqslant\mu} E(x,y,\mu,z)\left[f_i(\mathcal{U})-f_i(\mathcal{U}_{x,\mu})\right]\ud y+\left[(\psi_i)_{x, \mu}(z)-\psi_i(z)\right],
		\end{equation*}
		for $z\in B^*_{2}\cup B_{\mu}(x)$, where
		\begin{equation*}
			E(x,y,z,\mu)={|z-y|^{4-n}}-\left(\frac{|z-x|}{\mu}\right)^{4-n} {\left|\mathcal{I}_{x,\mu}(z)-y\right|^{4-n}}
		\end{equation*}
		is used to estimate the difference between a $p$-map $\mathcal{U}$ and its Kelvin transform $\mathcal{U}_{x_0,\mu}$. Finally, it is straightforward to check that $E(x,y,z,\mu)>0$ for all $|z-x|>\mu>0$, which concludes the proof of the lemma.
	\end{proof}
	
	The last lemma will be useful in the proof of our initial estimate. For this, we use a contradiction argument based on the blow-up classification.
	
	\begin{proposition}\label{integralmovingspheres}
		Let $s=2^{**}-1$, and $\mathcal{U}\in C\left(B^*_{2},\mathbb{R}^p\right)\cap L^{2^{**}-1}\left(B_{2},\mathbb{R}^p\right)$ be a strongly positive solution to \eqref{rescaledintegralsystem}. Suppose that $\psi_i\in C^{1}\left(B_{2}\right)$ is a positive function satisfying \eqref{logestimate} for any $i\in I$. Then,
		\begin{equation*}
			\limsup _{|x| \rightarrow 0}|x|^{-\gamma} |\mathcal{U}(x)|<\infty.
		\end{equation*}
	\end{proposition}
	
	\begin{proof}
		Let us assume by contradiction that there exist $i\in I$ and $\left\{x_{k}\right\}_{k\in\mathbb{N}}\subset B_{2}$ such that $\lim_{\rightarrow\infty}|x_k|=0$ and $\left|x_{k}\right|^{\gamma} u_i\left(x_{k}\right)\rightarrow\infty$ as $k\rightarrow\infty$.
		For $\left|x-x_{k}\right| \leqslant1/2{\left|x_{k}\right|}$, let us define
		\begin{equation*}
			\widetilde{u}_{ki}(x):=\left(\frac{\left|x_{k}\right|}{2}-\left|x-x_{k}\right|\right)^{\gamma} u_i(x).
		\end{equation*}
		Hence, using that $u_i$ is positive and continuous in $\bar{B}_{\left|x_{k}\right|/2}\left(x_{k}\right)$, there exists a maximum point $\bar{x}_{k} \in B_{\left|x_{k}\right|/2}\left(x_{k}\right)$ of     $\widetilde{u}_{ki}$, that is,
		\begin{equation*}
			\widetilde{u}_{ki}\left(\bar{x}_{k}\right)=\max _{\left|x-x_{k}\right|\leqslant{\left|x_{k}\right|/2}}\widetilde{u}_{ki}(x)>0.
		\end{equation*}
		Taking $2\mu_{k}:={\left|x_{k}\right|/2}-\left|\bar{x}_{k}-x_{k}\right|>0$, we get
		\begin{equation}\label{upbound1}
			0<2\mu_{k}\leqslant \frac{\left|x_{k}\right|}{2} \quad \mbox{and} \quad \frac{\left|x_{k}\right|}{2}-\left|x-x_{k}\right|\geqslant \mu_{k} \quad \mbox{for} \left|x-\bar{x}_{k}\right|\leqslant\mu_{k}.
		\end{equation}
		Moreover, it follows that $2^{\gamma} u_i\left(\bar{x}_{k}\right)\geqslant u_i(x)$ for $\left|x-\bar{x}_{k}\right|\leqslant\mu_{k}$ and
		\begin{equation}\label{upbound2}
			\left(2 \mu_{k}\right)^{\gamma} u_i\left(\bar{x}_{k}\right)=\widetilde{u}_{ki}\left(\bar{x}_{k}\right)\geqslant \widetilde{u}_{ki}\left(x_{k}\right)={2}^{-\gamma}{\left|x_{k}\right|^{\gamma}} u_i\left(x_{k}\right) \rightarrow \infty \quad \mbox{as} \quad k\rightarrow \infty.
		\end{equation}
		We consider
		\begin{equation*}
			w_{ki}(y)={u_i\left(\bar{x}_{k}\right)^{-1}} u\left(\bar{x}_{k}+{y}{u_i\left(\bar{x}_{k}\right)^{-\gamma^{-1}}}\right) \ \mbox{and} \ h_{ki}(y)={u_i\left(\bar{x}_{k}\right)^{-1}} \psi_i\left(\bar{x}_{k}+{y}{u_i\left(\bar{x}_{k}\right)^{-\gamma^{-1}}}\right) \ \mbox{in} \ \Omega_{k},
		\end{equation*}
		where
		\begin{equation*}
			\Omega_{k}=\left\{y \in \mathbb{R}^{n}: \bar{x}_{k}+{y}{u_i\left(\bar{x}_{k}\right)^{-\gamma^{-1}}}\in B^*_{2}\right\}.
		\end{equation*}
		Now, extending $w_{ki}$ to be zero outside of $\Omega_{k}$ and using Proposition~\ref{lm:integralrepresentation}, we get
		\begin{equation}\label{auxequation}
			w_{ki}(y)=\int_{\mathbb{R}^{n}} f_i(\mathcal{W}_k)|y-x|^{4-n} \ud x+h_{ki}(y) \quad \mbox{for} \quad y \in \Omega_{k}
		\end{equation}
		and $w_{ki}(0)=1$ for $k\in\mathbb{N}$, where $\mathcal{W}_k={\bf e_1}w_{ki}$. Moreover, from \eqref{upbound1} and \eqref{upbound2}, it holds
		\begin{equation*}
			\left\|h_{ki}\right\|_{C^{1}\left(\Omega_{k}\right)}\rightarrow 0 \quad \mbox{and} \quad w_{ki}(y)\leqslant 2^{\gamma} \quad \mbox{in} \quad B_{R_{ki}},
		\end{equation*}
		where $R_{ki}:=\mu_{ki} u_i\left(\bar{x}_{k}\right)^{\gamma^{-1}} \rightarrow \infty$ as $k\rightarrow\infty$.
		Using the regularity results in \cite{MR2055032}, one can find $w_{0i}>0$ such that $w_{ki} \rightarrow w_{0i}$ as $k\rightarrow\infty$ in $C_{\loc}^{4,\zeta}\left(\mathbb{R}^{n}\right)$ for some $\zeta\in(0,1)$, where $w_{0i}>0$ satisfies
		\begin{equation*}
			w_{0i}(y)=c(n)\int_{\mathbb{R}^{n}} |y-x|^{4-n} w_{0i}^{2^{**}-1}\ud x \quad \mbox{in} \quad  \mathbb{R}^{n}.
		\end{equation*}
		or, equivalently $\Delta^2w_{0i}=f_i(w_{0})$ in $\mathbb{R}^{n}$. Furthermore, by construction, we conclude $w_{0i}(0)=1$, which by Theorem~\ref{thm:andrade-doo19} (i), implies that there exist $\mu>0$ and $y_{0} \in \mathbb{R}^{n}$ such that 
		\begin{equation}\label{buble}
			w_{0i}(y)=\left(\frac{2\mu}{1+\mu^{2}\left|y-y_{0}\right|^{2}}\right)^{\gamma}.
		\end{equation}    
		In the next claim, we use the last classification formula to apply the moving spheres technique.
		
		\noindent{\bf Claim 1:} For any $\mu>0$, it holds that $(w_{0i})_{x,\mu}(y) \leqslant w_{0i}(y)$ for $|y-x|\geqslant\mu$.
		
		\noindent Indeed, for a fixed $\mu_{0}>0$, we have that $0<\mu_{0}<{R_{k}}/{10}$ when $k\gg1$. We also consider
		\begin{equation*}
			\widetilde{\Omega}_{k}:=\left\{y \in \mathbb{R}^{n}: \bar{x}_{k}+{y}{u\left(\bar{x}_{k}\right)^{-\gamma^{-1}}} \in B^*_{1}\right\} \subset \subset {\Omega}_{k}.
		\end{equation*}
		Now let us divide the proof of the claim into three step as follows
		
		\noindent{\bf Step 1:} For $k\gg1$, it holds that $\left(w_{ki}\right)_{x,\mu_{0}}(y) \leqslant w_{ki}(y)$ for $y\in\widetilde{\Omega}_{k}$ such that $|y|\geqslant\mu_{0}$.
		
		\noindent In fact, by Lemma~\ref{movingpsheresineq}, there exists $\bar{r}>0$ such that for all $0<\mu\leqslant\bar{r}$ and $\bar{x}\in B_{1/100}$,
		\begin{equation}\label{upperbound7}
			\left(\frac{\mu}{|y|}\right)^{n-4} \psi_{ki}\left(\mathcal{I}_{0,\mu}(y)+\bar{x}\right) \leqslant \psi_{ki}(y+\bar{x}) \quad \mbox{for} \quad |y|\geqslant \mu \quad \mbox{and} \quad y \in B_{149/100}.
		\end{equation}
		Let $k\gg1$ be sufficiently large such that $\mu_{0} u\left(\bar{x}_{k}\right)^{(4-n)^{-1}}<\bar{r}$. Hence, for any $0<\mu\leqslant\mu_{0}$, it holds
		\begin{equation}\label{inequality}
			\left(\psi_{ki}\right)_{x, \mu}(y) \leqslant \psi_{ki}(y) \quad \mbox{in} \quad \widetilde{\Omega}_{k}\setminus B_{\mu},
		\end{equation}
		which by passing to the limit as $k\rightarrow\infty$ concludes the proof of Step 1.
		
		\noindent{\bf Step 2:}  There exists $\mu_{1}>0$, independent of $k$, such that $\left(w_{ki}\right)_{x, \mu}(y)\leqslant w_{ki}(y)$ in $\widetilde{\Omega}_{k}\setminus B_{\mu}$ for $0<\mu<\mu_{1}$.
		
		\noindent As matter of fact, since $w_{ki}\rightarrow w_{0i}$ as $k\rightarrow\infty$ in $C^4$-topology and $w_{0i}$ is given by \eqref{buble} we get that there exists $C_0>0$ satisfying $w_{ki}\geqslant C_{0}>0$ on $B_{1}$ for $k\gg1$. On the other hand, by \eqref{auxequation} and standard regularity results, it follows that $\left|D^{(j)}w_{ki}\right|\leqslant C_{0}<\infty$ on $B_{1}$ for $j=1,2,3,4$. Using Lemma~\ref{movingpsheresineq}, there exists $r_{0}>0$, not depending on $k\gg1$, such that for all $0<\mu\leqslant r_{0}$, it holds
		\begin{equation}\label{movingspheres1}
			\left(w_{ki}\right)_{x,\mu}(y)<w_{ki}(y) \quad \mbox{for} \quad 0<\mu<|y|\leqslant r_{0}.
		\end{equation}
		Again, since $w_{ki}\geqslant C_{0}>0$ on $B_{1}$ for $k\gg1$, there exists $C>0$ satisfying
		\begin{equation*}
			w_{ki}(x) \geqslant C_{0}^{2^{**}-1} \int_{B_{1}}|x-y|^{4-n}\ud y\geqslant \frac{1}{C}(1+|x|)^{4-n} \quad \mbox{in} \quad \Omega_{k}.
		\end{equation*}
		Therefore, one can find $0<\mu_{1}\leqslant r_{0}\ll1$ sufficiently small such that
		for all $0<\mu<\mu_{1}$, we have
		\begin{equation*}
			\left(w_{ki}\right)_{\mu}(y)\leqslant\left(\frac{\mu_{1}}{|y|}\right)^{n-4} \max _{B_{r_{0}(x)}} w_{ki} \leqslant C\left(\frac{\mu_{1}}{|y|}\right)^{n-4} \leqslant w_{ki}(y) \quad \mbox{for} \quad y\in \Omega_{k} \quad \mbox{and} \quad |y|\geqslant r_{0}, 
		\end{equation*}
		which in combination with \eqref{movingspheres1} proves Step 2.
		
		For $\mu_{0}>0$ fixed before, let us define,
		\begin{equation*}
			\mu^*:=\sup\left\{0<\mu\leqslant\mu_{0} : \left(w_{ki}\right)_{x,\mu}(y)\leqslant w_{ki}(y), \ y \in \widetilde{\Omega}_{k} \ \mbox{with} \ |y-x_{0}|\geqslant\mu \ \mbox{and} \ 0<\mu<\mu_0\right\}.
		\end{equation*}
		
		\noindent{\bf Step 3:} For $k\gg1$, it holds that $\mu^*=\mu_{0}$.
		
		\noindent Indeed, using \eqref{identity1}, \eqref{identity2} and \eqref{inequality}, it follows
		\begin{align}\label{upperbound4}
			\nonumber
			&w_{ki}(y)-\left(w_{ki}\right)_{0,\mu}(y)&\\
			&=\int_{\mathbb{R}^n\setminus B_{\mu}} E(0,y,z,\mu)\left[w_{ki}(z)^{2^{**}-1}-\left(w_{ki}\right)_{0,\mu}(z)^{2^{**}-1}\right]\ud z+\left[\psi_{ki}(y)-\left(\psi_{ki}\right)_{0,\mu}(y)\right]&\\\nonumber
			&\geqslant\int_{\widetilde{\Omega}_{k}\setminus B_{\mu}} E(0,y,z,\mu)\left[w_{ki}(z)^{2^{**}-1}-\left(w_{ki}\right)_{0,\mu}(z)^{2^{**}-1}\right]\ud z+J\left(\mu, w_{ki}, y\right),&
		\end{align}
		for any $\mu^*\leqslant\mu\leqslant\mu^*+{1}/{2}$ and $y\in\widetilde{\Omega}_{k}$ with $|y|>\mu$, where
		\begin{align*}
			J\left(\mu, w_{ki}, y\right)
			&=\int_{\mathbb{R}^{n}\setminus\widetilde{\Omega}_{k}} E(0,y,z,\mu)\left[w_{ki}(z)^{2^{**}-1}-\left(w_{ki}\right)_{0,\mu}(z)^{2^{**}-1}\right]\ud z&\\
			&=\int_{\Omega_{k}\setminus\widetilde{\Omega}_{k}} E(0,y,z,\mu)\left[w_{ki}(z)^{2^{**}-1}-\left(w_{ki}\right)_{0,\mu}(z)^{2^{**}-1}\right]\ud z&\\
			&-\int_{\mathbb{R}^{n}\setminus\Omega_{k}} E(0,y,z,\mu)\left(w_{ki}\right)_{0,\mu}(z)^{2^{**}-1}\ud z.&
		\end{align*}
		For $z\in\mathbb{R}^{n}\setminus\widetilde{\Omega}_{k}$ and $\mu^*\leqslant\mu\leqslant\mu^*+1,$ we obtain that $|z|\geqslant{1}/{2} u\left(\bar{x}_{k}\right)^{-\gamma^{-1}}$ and thus
		\begin{equation*}
			\left(w_{ki}\right)_{0,\mu}(z)\leqslant\left(\frac{\mu}{|z|}\right)^{n-4} \max_{B_{\mu^*+1}} w_{ki}\leqslant C u_i\left(\bar{x}_{k}\right)^{-2}.
		\end{equation*}
		In addition, since $u_i\geqslant C_{0}>0$ in $B_{2}\setminus B_{1/2}$, by the definition of $w_{ki},$ we have
		\begin{equation*}
			w_{ki}(y)\geqslant \frac{C_{0}}{u_i\left(\bar{x}_{k}\right)} \quad \mbox{in} \quad \Omega_{k}\setminus\widetilde{\Omega}_{k},
		\end{equation*}
		which in turns yields
		\begin{align}\label{upperbound2}
			J\left(\mu, w_{ki}, y\right)
			&\geqslant\frac{1}{2}\left(\frac{C_{0}}{u_i\left(\bar{x}_{k}\right)}\right)^{2^{**}-1} \int_{\Omega_{k}\setminus\widetilde{\Omega}_{k}} E(0,y,z,\mu)\ud z-C\int_{\mathbb{R}^n\setminus\Omega_{k}} E(0,y,z,\mu)\left(\frac{\mu}{|z|}\right)^{n+4} \mathrm{d} z&\\\nonumber
			&\geqslant
			\left\{\begin{array}{ll}{\frac{1}{C}(|y|-\mu) u_i\left(\bar{x}_{k}\right)^{-1}}, & {\mbox{if} \quad \mu\leqslant|y|\leqslant\mu^*+1} \\ {\frac{1}{C} u_i\left(\bar{x}_{k}\right)^{-1}}, & {\mbox{if} \quad |y|>\mu^*+1 \quad \mbox{and} \quad y \in \widetilde{\Omega}_{k}}.
			\end{array}\right.
		\end{align}
		Indeed, since $E(0,y,z,\mu)=0$ for $|y|=\mu$, and
		\begin{equation*}
			\left.y\nabla_{y}\cdot E(0,y,z,\mu)\right|_{|y|=\mu}=(n-4)|y-z|^{4-n-2}\left(|z|^{2}-|y|^{2}\right)>0,
		\end{equation*}
		for $|z|\geqslant\mu^*+2$, and using the positivity and smoothness of $E$, there exists $0<\delta_{1}\leqslant\delta_{2}<\infty$ satisfying
		\begin{equation}\label{upperbound3}
			{\delta_{1}}{|y-z|^{4-n}}(|y|-\mu)\leqslant E(0,y,z,\mu)\leqslant{\delta_{2}}{|y-z|^{4-n}}(|y|-\mu),
		\end{equation}
		for $\mu^*\leqslant\mu\leqslant|y|\leqslant\mu^*+1$, $\mu^*+2\leqslant|z|\leqslant R<\infty$. Furthermore, if $R\gg1$ is large, it follows that
		$0<C_{0}\leqslant y\cdot\nabla_{y}\left(|y-z|^{n-4} E(0,y,z,\mu)\right)\leqslant C_{0}<\infty$ for all $|z|\geqslant\mu$, $\mu^*\leqslant\mu\leqslant|y|\leqslant\mu^*+1$.
		Thus, \eqref{upperbound3} holds for $\mu^*\leqslant\mu\leqslant|y|\leqslant\mu^*+1$ and $|z|\geqslant R$. Besides, by the definition of $E(0,y,z,\mu)$, there exists $0<\delta_{3}\leqslant1$ such that 
		\begin{equation}\label{upperbound6}
			{\delta_{3}}{|y-z|^{4-n}}\leqslant E(0,y,z,\mu)\leqslant{|y-z|^{4-n}},
		\end{equation}
		for $|y|\geqslant\mu^*+1$ and $|z|\geqslant\mu^*+2$.
		Therefore, for $k\gg1$, $\mu\leqslant|y|\leqslant\mu^*+1$, we find
		\begin{align*}
			J\left(\mu, w_{ki}, y\right)
			&\geqslant\frac{1}{2}\left(\frac{C_{0}}{u_i\left(\bar{x}_{k}\right)}\right)^{2^{**}-1} \int_{\Omega_{k}\setminus\widetilde{\Omega}_k}{\delta_{1}}{|y-z|^{4-n}}(|y|-\mu)\ud z&\\
			&-C\int_{\mathbb{R}^n\setminus\Omega_{k}}{\delta_{2}}{|y-z|^{4-n}}(|y|-\mu)\left(\frac{\mu}{|z|}\right)^{n+4}\ud z&\\ 
			&\geqslant\frac{1}{C_{1}}(|y|-\mu)u_i\left(\bar{x}_{k}\right)^{-1}-\frac{1}{C_{2}}(|y|-\mu) u_i\left(\bar{x}_{k}\right)^{-2^{**}}&\\ 
			&\geqslant\frac{1}{2 C_{1}}(|y|-\mu) u_i\left(\bar{x}_{k}\right)^{-1}.
		\end{align*}
		Similarly, for $|y|\geqslant\mu^*+1$ and $y\in\widetilde{\Omega}_k$, since $u_i\left(\bar{x}_{k}\right)\rightarrow\infty$ as $k\rightarrow\infty$, there exist $C_{1},C_{2},C_{3},C_{4}>0$ satisfying
		\begin{equation*}
			J\left(\mu, w_{ki}, y\right)\geqslant\frac{1}{C_{3}} u_i\left(\bar{x}_{k}\right)^{-1}-\frac{1}{C_{4}} u_i\left(\bar{x}_{k}\right)^{-2^{**}}\geqslant\frac{1}{2 C_{3}} u_i\left(\bar{x}_{k}\right)^{-1},
		\end{equation*}
		which verifies \eqref{upperbound2}.
		
		Next, by \eqref{upperbound4} and \eqref{upperbound2}, there exists $\varepsilon_{1}\in(0,1/2)$, depending on $k$, such that for $|y|\geqslant\mu^*+1$, it holds
		\begin{equation*}
			w_{ki}(y)-\left(w_{ki}\right)_{0,\mu^*}(y)\geqslant{\varepsilon_{1}}{|y|^{4-n}} \quad \mbox{in} \quad \widetilde{\Omega}_k.
		\end{equation*}
		Using the above inequality and the formula for $\left(w_{ki}\right)_{0,\mu},$ there exists $0<\varepsilon_{2}<\varepsilon_{1}\ll1$ such that for $|y|\geqslant\mu^*+1$, $\mu^*\leqslant \mu\leqslant\mu^*+\varepsilon_{2}$, we get
		\begin{align}\label{upperbound5}
			w_{ki}(y)-\left(w_{ki}\right)_{0,\mu}(y)\geqslant{\varepsilon_{1}}{|y|^{4-n}}+\left[\left(w_{ki}\right)_{0,\mu^*}(y)-\left(w_{ki}\right)_{\mu}(y)\right]\geqslant\frac{\varepsilon_{1}}{2}{|y|^{4-n}}.
		\end{align}
		For $\varepsilon\in\left(0,\varepsilon_{3}\right]$ that will be chosen later, by \eqref{upperbound4} and \eqref{upperbound2} combined with the inequality, we obtain
		\begin{equation*}
			\left|\left(w_{ki}\right)(z)^{2^{**}-1}-\left(w_{ki}\right)_{0,\mu}(z)^{2^{**}-1}\right|\leqslant C(|z|-\mu),
		\end{equation*}
		we get
		\begin{align*}
			w_{ki}(y)-\left(w_{ki}\right)_{0,\mu}(y)
			&\geqslant\int_{\mu\leqslant|z|\leqslant\mu^*+1}E(0,y,z,\mu)\left(w_{ki}(z)^{2^{**}-1}-\left(w_{ki}\right)_{0,\mu}(z)^{2^{**}-1}\right)\ud z&\\ 
			&+\int_{\mu^*+2 \leqslant|z|\leqslant\mu^*+3} E(0,y,z,\mu)\left(w_{ki}(z)^{2^{**}-1}-\left(w_{ki}\right)_{0,\mu}(z)^{2^{**}-1}\right)\ud z&\\ 
			&\geqslant -C\int_{\mu\leqslant|z|\leqslant\mu+\varepsilon}E(0,y,z,\mu)(|z|-\mu)\ud z&\\
			&+\int_{\mu+\varepsilon\leqslant|z|\leqslant\mu^*+1} E(0,y,z,\mu)\left(w_{ki}(z)^{2^{**}-1}-\left(w_{ki}\right)_{0,\mu}(z)^{2^{**}-1}\right)\ud z&\\ 
			&+\int_{\mu^*+2\leqslant|z|\leqslant\mu^*+3} E(0,y,z,\mu)\left(w_{ki}(z)^{2^{**}-1}-\left(w_{ki}\right)_{0,\mu}(z)^{2^{**}-1}\right) \ud z,
		\end{align*}
		for $\mu^*\leqslant\mu\leqslant\mu^*+\varepsilon$ and $\mu\leqslant|y|\leqslant\mu^*+1$.    From \eqref{upperbound5}, there exists $\delta_{5}>0$ such that for $\mu^*+2\leqslant|z|\leqslant\mu^*+3$, it follows 
		\begin{equation*}
			w_{ki}(z)^{2^{**}-1}-\left(w_{ki}\right)_{0,\mu}(z)^{2^{**}-1}\geqslant\delta_{5}.
		\end{equation*}
		Moreover, since $\|w_{ki}\|_{C^{1}\left(B_{2}\right)}\leqslant C$ (independent of $k$), there exists some constant $C>0$, not depending on $\varepsilon$, such that for $\mu^*\leqslant\mu\leqslant\mu^*+\varepsilon$, we get
		\begin{equation*}
			\left|\left(w_{ki}\right)_{0,\mu^*}(z)^{2^{**}-1}-\left(w_{ki}\right)_{0,\mu}(z)^{2^{**}-1}\right|\leqslant C(\mu-\mu^*)\leqslant C\varepsilon, 
		\end{equation*}
		for $\mu\leqslant|z|\leqslant\mu^*+1$. Also, for $\mu\leqslant|y|\leqslant\mu^*+1$, we find 
		\begin{align*}
			\int_{\mu+\varepsilon\leqslant|z| \leqslant\mu^*+1} E(0,y,z,\mu)\ud z 
			&\leqslant\left|\int_{\mu+\varepsilon\leqslant|z|\leqslant\mu^*+1}\left({|y-z|^{4-n}}-|\mathcal{I}_{0,\mu}(y)-z|^{4-n}\right)\ud z\right|&\\ 
			&+\int_{\mu+\varepsilon \leqslant|z| \leqslant\mu^*+1}\left|\left(\frac{\mu}{|y|}\right)^{n-4}-1\right|{\left|\mathcal{I}_{0,\mu}(y)-z\right|^{n-4}}\ud z&\\
			&\leqslant C\left(\varepsilon^{3}+|\ln \varepsilon|+1\right)(|y|-\mu)&
		\end{align*}
		and
		\begin{align*}
			\int_{\mu\leqslant|z|\leqslant\mu+\varepsilon} E(0,y,z,\mu)(|z|-\mu)\ud z 
			&\leqslant \left|\int_{\mu\leqslant|z|\leqslant\mu+\varepsilon}\left(\frac{|z|-\mu}{|y-z|^{n-4}}-\frac{|z|-\mu}{\left|\mathcal{I}_{0,\mu}(y)-z\right|^{n-4}}\right)\ud z\right|&\\
			&+\varepsilon \int_{\mu\leqslant|z|\leqslant\mu+\varepsilon}\left|\left(\frac{\mu}{|y|}\right)^{n-4}-1\right|{\left|\mathcal{I}_{0,\mu}(y)-z\right|^{4-n}}\ud z&\\
			&\leqslant I+C \varepsilon(|y|-\mu),&
		\end{align*}
		where, for $|y|\geqslant\mu+10\varepsilon$, we arrive at
		\begin{align*}
			I&=\left|\int_{\mu\leqslant|z|\leqslant\mu+\varepsilon}\left(\frac{|z|-\mu}{|y-z|^{n-4}}-\frac{|z|-\mu}{\left|\mathcal{I}_{0,\mu}(y)-z\right|^{n-4}}\right)\ud z\right|&\\ 
			&\leqslant C\varepsilon\left(\varepsilon^{3}+|\ln \varepsilon|+1\right)(|y|-\mu).&
		\end{align*}
		On the other hand, for $\mu\leqslant|y|\leqslant\mu+10\varepsilon$, it follows
		\begin{align*}
			I
			&\leqslant\left|\int_{\mu\leqslant|z|\leqslant\mu+10(|y|-\mu)}
			\left(\frac{|z|-\mu}{|y-z|^{n-4}}-\frac{|z|-\mu}{\left|\mathcal{I}_{0,\mu}(y)-z\right|^{n-4}}\right)\ud z
			\right|&\\ 
			&+\left|\int_{\mu+10(|y|-\mu)\leqslant|z|\leqslant \mu+\varepsilon}\left(\frac{|z|-\mu}{|y-z|^{n-4}}-\frac{|z|-\mu}{\left|\mathcal{I}_{0,\mu}(y)-z\right|^{n-4}}\right)\ud z\right|&\\
			&\leqslant C(|y|-\mu)\int_{\mu\leqslant|z|\leqslant \mu+10(|y|-\mu)}\left(\frac{1}{|y-z|^{n-4}}+\frac{1}{\left|\mathcal{I}_{0,\mu}(y)-z\right|^{n-4}}\right)\ud z&\\ 
			&+C\left|y-\mathcal{I}_{0,\mu}(y)\right|\int_{\mu+10(|y|-\mu)\leqslant|z|\leqslant\mu+\varepsilon} \frac{|z|-\mu}{|y-z|^{n-3}}\ud z&\\
			&\leqslant C(|y|-\mu) \sup _{\widetilde{z}\in\mathbb{R}^{n}} \int_{\mu \leqslant|z| \leqslant \mu+100\varepsilon}{|\widetilde{z}-z|^{n-4}}\ud z&\\
			&\leqslant C(|y|-\mu)\varepsilon^{4/n}.&
		\end{align*}
		Finally, using \eqref{upperbound6}, for $\mu<|y|\leqslant\mu^*+1$ and $0<\varepsilon\ll1$ sufficiently small, it holds
		\begin{align*}
			{w_{ki}(y)-\left(w_{ki}\right)_{0,\mu}(y)}
			&\geqslant-C\varepsilon^{\frac{4}{n}}(|y|-\mu)+\delta_{1}\delta_{5}(|y|-\mu)\int_{\mu+2\leqslant|z|\leqslant\mu^*+3}|y-z|^{4-n}\ud z&\\ 
			&\geqslant\left(\delta_{1} \delta_{5} c-C \varepsilon^{\frac{4}{n}}\right)(|y|-\mu) \geqslant0,&
		\end{align*}
		which together with \eqref{upperbound5} contradicts the definition of $\mu^*>0$, if $\mu^*<\mu_{0}$. Hence, the proof of Step 3 is finished, and so Claim 1. Besides, this is also a contradiction with \eqref{buble} that concludes the proof of the proposition.   
	\end{proof}
	
	\begin{corollary}\label{upperbound}
		Let $\mathcal{U}$ be a strongly positive singular solution to \eqref{oursystem}. Then, there exists $C_2>0$ satisfying $|\mathcal{U}(x)|\leqslant C_2|x|^{-\gamma}$ for $0<|x|<1/2$.
	\end{corollary}

	\subsection{Asymptotic radial symmetry and the Harnack inequality}
	Here, we prove the convergence of singular solutions to \eqref{oursystem} to its spherical average, which is defined by $\overline{\mathcal{U}}(x)=\avint_{\partial B_{1}} \mathcal{U}(r \theta)\ud\theta$. Notice that this asymptotic approximation implies that solutions must become radially symmetric close to the origin; this property is called {\it asymptotic radial symmetry}.
	
	\begin{proposition}
		Let $\mathcal{U}$ be a strongly positive singular solution to \eqref{oursystem}. Then, $\mathcal{U}$ is radially symmetric and 
		\begin{equation*}
			|\mathcal{U}(x)|=(1+\mathcal{O}(|x|))|\overline{\mathcal{U}}|(x) \quad {\rm as} \quad x\rightarrow0,
		\end{equation*}
		where $\overline{\mathcal{U}}$ is the spherical average of $\mathcal{U}$.
	\end{proposition}
	
	\begin{proof}
		First, we prove the following claim:
		
		\noindent{\bf Claim 1:} There exists $0<\varepsilon<\min\{1 /10,\bar{r}\}$ such that
		\begin{equation*}
			|\mathcal{U}_{x,\mu}(y)|\leqslant|\mathcal{U}(y)| \quad  \mbox{in} \quad B_{1}(x)\setminus B_{\mu}(x), 
		\end{equation*}
		for $0<\mu<|x|<\varepsilon$, where $\bar{r}$ is such that \eqref{upperbound7} holds for all $0<\mu\leqslant\bar{r}$.
		
		\noindent We divide the proof of the claim into two steps as follows:
		
		\noindent{\bf Step 1:} The critical parameter 
		\begin{equation*}
			\mu^*(x):=\sup\left\{0<\mu\leqslant|x| : |\mathcal{U}_{x,\mu}(y)|\leqslant|\mathcal{U}(y)| \ \mbox{for} \ y\in B_{2}\setminus\left(B_{\mu}(x)\cup\{0\}\right) \ \mbox{and} \ 0<\mu<\mu^*\right\}
		\end{equation*}
		is well-defined and positive.
		
		\noindent Indeed, using Lemma~\ref{movingpsheresineq}, for every $x\in B^*_{1/10}$ one can find $0<r_{x}<|x|$ such
		that for all $0<\mu\leqslant r_{x}$, it follows $|\mathcal{U}_{x,\mu}(y)|\leqslant|\mathcal{U}(y)|$ for $0<\mu<|y-x| \leqslant r_{x}$.
		Moreover, as a consequence of \eqref{rescaledintegralsystem}, we get
		\begin{equation}\label{ineq1}
			|\mathcal{U}(x)|\geqslant 4^{4-n}\int_{B_{2}} |f_i(\mathcal{U})|(y)\ud y:=C_{0}>0,
		\end{equation}
		which implies that there exists $0<\mu_{1}\ll r_{x}$ such that, for every $0<\mu\leqslant\mu_{1}$, it holds
		\begin{equation}\label{ineq2}
			|\mathcal{U}_{x,\mu}(y)|\leqslant|\mathcal{U}(y)| \quad \mbox{for} \quad y\in B_{2}\setminus\left(B_{r_{x}}(x) \cup\{0\}\right).
		\end{equation} 
		Hence, as a combination of \eqref{ineq1} and \eqref{ineq2}, it follows the proof of Step 1.
		
		\noindent{\bf Step 2:} There exists $\varepsilon>0$ such that $\mu^*=|x|$ for all $|x|\leqslant \varepsilon$, $\mu^*\leqslant\mu<|x|\leqslant\bar{r}$, and $y\in B_{1}$, we get 
		\begin{equation*}
			u_i(y)-(u_i)_{x, \mu}(y)\geqslant\int_{B_{1}\setminus B_{\mu}(x)} E(0,y,z,\mu)\left[f_i(\mathcal{U}(z))-f_i(\mathcal{U}_{x, \mu}(z))\right]\ud z+J(\mu, u, y), 
		\end{equation*}
		for any $i\in I$, where
		\begin{align*}
			J(\mu, u_i, y)=& \int_{B_{2}\setminus B_{1}} E(x,y,z,\mu)\left[f_i(\mathcal{U}(z))-f_i(\mathcal{U}_{x, \mu}(z))\right]\ud z-\int_{\mathbb{R}^n\setminus B_2} E(x,y,z,\mu)f_i(\mathcal{U}_{x, \mu}(z))\ud z.
		\end{align*}
		
		\noindent In fact, for $y\in\mathbb{R}^n\setminus B_1$ and $\mu<|x|<\varepsilon<1/10$, we have 
		\begin{equation*}
			\mathcal{I}_{x,\mu}(y)=\left|x+\frac{\mu^{2}(y-x)}{|y-x|^{2}}\right|\geqslant|x|-\frac{10}{9}\mu^{2} \geqslant|x|-\frac{10}{9}|x|^{2}\geqslant\frac{8}{9}|x|.
		\end{equation*}
		Using Proposition~\ref{movingpsheresineq}, there exists $C_0,C_1>0$ such that $|\mathcal{U}|\left(\mathcal{I}_{x,\mu}(y)\right)\leqslant C|x|^{-\gamma}$, which provides for all $y\in\mathbb{R}^n\setminus B_1$,
		\begin{align*}
			|\mathcal{U}_{x,\mu}(y)|=\left(\frac{\mu}{|y-x|}\right)^{n-4} |\mathcal{U}\left(\mathcal{I}_{x,\mu}(y)\right)|\leqslant C\mu^{n-4}|x|^{-\gamma}\leqslant C|x|^{\gamma}\leqslant C\varepsilon^{\gamma}.
		\end{align*}
		By \eqref{ineq1}, we find $|\mathcal{U}_{x, \mu}(y)|<|\mathcal{U}(y)|$ for $y\in B_{2}\setminus B_{1}$.
		In addition, by \eqref{upperbound2}, \eqref{ineq1} and \eqref{ineq2}, there exist $C_{1}>0$, independent of $x$, such that
		\begin{align*}
			&J(\mu, |\mathcal{U}|, y)&\\
			&\geqslant \int_{B_2\setminus B_1}E(x, y, z, \mu)\left(C_{0}^{2^{**}-1}-C^{2^{**}-1} \varepsilon^{n+4}\right)\ud z&\\
			&-C\int_{\mathbb{R}^n\setminus B_2}E(x, y, z, \mu)\left(|x-z|^{4-n}|x|^{\gamma}\right)^{2^{**}-1}\ud z&\\
			&\geqslant\frac{1}{2}C_{0}^{2^{**}-1}\int_{B_2\setminus B_1}E(x, y, z, \mu)\ud z-\varepsilon^{\frac{n+4}{2}}C\int_{\mathbb{R}^n\setminus B_2}E(x, y, z, \mu)|x-z|^{n+4}\ud z&\\
			&\geqslant\frac{1}{2}C_{0}^{2^{**}-1}\int_{B_{19/10}\setminus B_{11/10}}E(0, y-x, z, \mu)\ud z-\varepsilon^{\frac{n+4}{2}}C\int_{\mathbb{R}^n\setminus B_{19/10}}E(0, y-x, z, \mu)|x-z|^{n+4}\ud z&\\
			&\geqslant \frac{1}{C_{1}}(|y-x|-\mu),&
		\end{align*}
		for $y \in B_{1} \backslash\left(B_{\mu}(x) \cup\{0\}\right)$ and  $0<\varepsilon\ll1$. Eventually, if $\mu^*<|x|$, by Lemma~\ref{movingpsheresineq}, we get a contradiction with the definition of $\mu^*$. Hence, Step 2 is proved and so Claim 1.
		
		Finally, for any $i\in I$, choose $0<r_i<\varepsilon^{2}$ and $x_{1i}, x_{2i}\in\partial B_{r_i}$ satisfying
		\begin{equation*}
			u_i\left(x_{1i}\right)=\max_{\partial B_{r_i}} u_i \quad \mbox{and} \quad u_i\left(x_{2i}\right)=\min_{\partial B_{r_i}} u_i.
		\end{equation*}
		Now choosing
		\begin{equation*}
			x_{3i}=x_{1i}+\frac{\varepsilon_i\left(x_{1i}-x_{2i}\right)}{4\left|x_{1i}-x_{2i}\right|} \quad \mbox{and} \quad \mu=\sqrt{\frac{\varepsilon_i}{4}\left(\left|x_{1i}-x_{2i}\right|+\frac{\varepsilon_i}{4}\right)},
		\end{equation*}
		it follows from Claim 1 that
		$(u_i)_{x_{3i},\mu}\left(x_{2i}\right)\leqslant u_i\left(x_{2i}\right)$.
		Furthermore, we have
		\begin{align*}
			(u_i)_{x_{3i},\mu}\left(x_{2i}\right) &=\left(\frac{\mu}{\left|x_{1i}-x_{2i}\right|+{\varepsilon_i}/4}\right)^{n-4} u_i\left(x_{1i}\right)&\\ 
			&=\left(\frac{1}{4\left|x_{1i}-x_{2i}\right|\varepsilon_i^{-1}+1}\right)^{\frac{n-4}{2}} u_i\left(x_{1i}\right)&\\ &\geqslant\left(\frac{1}{8r\varepsilon_i^{-1}+1}\right)^{\gamma} u_i\left(x_{1i}\right),&
		\end{align*}
		which implies 
		\begin{equation*}
			\max_{\partial B_{r_i}} u_i\leqslant\left({8r_i}{\varepsilon^{-1}}+1\right)^{\gamma}\min_{\partial B_{r_i}}u_i,
		\end{equation*}
		and this proves the proposition.
	\end{proof}
	
	As a consequence of the upper estimate, we prove the following Harnack inequality for solutions to \eqref{oursystem}, whose scalar version can be found in \cite[Theorem~3.6]{MR2240050}.
	
	\begin{corollary}\label{harnack}
		Let $\mathcal{U}$ be a strongly positive solution to \eqref{oursystem}. Then, there exists $C>0$ such that 
		\begin{equation*}
			\max_{|x|=r}|\mathcal{U}|\leqslant C\min_{|x|=r}|\mathcal{U}| \quad {\rm for} \quad 0<r<1/4.
		\end{equation*}
		Moreover, $|D^{(j)}\mathcal{U}|\leqslant C|x|^{-j}|\mathcal{U}|$ for $j=1,2,3,4$.
	\end{corollary}
	
	\begin{proof}
		For any $i\in I$, let us define $\widetilde{u}_i(y)=r^{\gamma}u_i(ry)$. Thus, 
		\begin{equation}\label{harnack1}
			\widetilde{u}_i(y)=\int_{B_{2}} |y-z|^{4-n}f_i(\widetilde{\mathcal{U}})\ud z+\widetilde{h}_{i}(y),
		\end{equation}
		where $\widetilde{\psi}_{i}(y)=r^{\gamma} \psi_i(ry)$. By Proposition~\ref{upperbound}, there exists $C_2>0$, such that $\widetilde{u}_i\leqslant C_2$ in $B_{2}\setminus B_{1/10}$. Taking $|x|=1$, let us consider
		\begin{equation*}
			(\varrho_i)_{x}(y)=\int_{B_{2/r}(x)\setminus B_{9/10}(x)}|y-z|^{4-n}f_i(\widetilde{\mathcal{U}})\ud z.
		\end{equation*}    
		Hence, for any $y_{1}, y_{2} \in B_{1 / 2}(x),$ we have
		\begin{align*} 
			(\varrho_i)_{x}(y_{1})\leqslant C\int_{B_{2/r}(x)\setminus B_{9/10}(x)} |y-z|^{4-n}f_i(\widetilde{\mathcal{U}})\ud z\leqslant C (\varrho_i)_{x}(y_{2}),
		\end{align*}
		which implies that $\varrho_i$ satisfies the Harnack inequality in $B_{1/2}(x)$. On the other hand, $\psi_i$ also satisfies the Harnack inequality in $B_{1/2}(x)$ and
		\begin{equation*}
			\widetilde{u}_i(y)=\int_{B_{9/10}(x)}|y-z|^{4-n}f_i(\widetilde{\mathcal{U}})\ud z+(\varrho_i)_{x}(y)+\widetilde{h}_{i}(y) \quad \mbox{in} \quad B_{1/2}(x).
		\end{equation*}
		Now by \cite[Theorem~2.3]{MR2055032} we have that $\sup_{B_{1/2}(x)}\widetilde{u}_i\leqslant C \inf _{B_{1/2}(x)}\widetilde{u}_i$, which by covering argument provides
		\begin{equation*}
			\sup_{1/2\leqslant|y|\leqslant 3/2}\widetilde{u}_i\leqslant C \inf _{1/2\leqslant|y|\leqslant 3/2}\widetilde{u}_i,
		\end{equation*}
		and, by rescaling back to $u$, the proof of the first part follows.
		
		Next, for any fixed $x$ and $i\in I$, let $r=|x|$ and $\widetilde{u}_i(y)=r^{\gamma}u_i(ry)$. Thus, $\widetilde{\mathcal{U}}_i$ satisfies \eqref{harnack1} and, by Proposition~\ref{upperbound}, it holds that $\widetilde{u}_i\leqslant C_2$ in $B_{3/2}\setminus B_{1/2}$. Finally, using the local estimates from \cite[Section~2.1]{MR2055032} and the smoothness of $\psi_i$, one can find $C>0$ satisfying $\left|D^{(j)}\widetilde{u}_i(x)\right|\leqslant C$ for $|x|=1$ and $j=1,2,3,4$. Whence, by scaling back to $u_i$, the proof is concluded.
	\end{proof}
	
	\subsection{Removable singularity classification and the lower bound estimate}
	Next, we use the Pohozaev invariant, the Harnack inequality, and a barrier argument for providing a removable classification result, which implies the lower bound estimate in Proposition~\ref{estimates}.
	
	\begin{lemma}\label{liminflimsup}
		Let $s=2^{**}-1$ and $\mathcal{U}\in C\left(B^*_{2},\mathbb{R}^p\right)\cap L^{2^{**}-1}\left(B_{2},\mathbb{R}^p\right)$ be a strongly positive solution to \eqref{rescaledintegralsystem}. Assume $\psi_i\in C^{\infty}(B_{1})$ for all $i\in I$. If $\limsup _{|x|\rightarrow 0} |\mathcal{U}(x)|=\infty$, then $\liminf _{|x|\rightarrow 0} |\mathcal{U}(x)|=\infty$.
	\end{lemma}
	
	\begin{proof}
		Let us consider $\{x_{k}\}_{k\in\mathbb{N}}\subset B_1$ satisfying $r_{k}=|x_{k}|\rightarrow0$ and $u_i\left(x_{k}\right)\rightarrow\infty$ as $k\rightarrow\infty$. By the Harnack inequality, we have $\inf_{\partial B_{r_{k}}}u_i\rightarrow\infty$. Thus, we obtain $-\Delta(u_i-\psi_i)\geqslant0$ in $B^*_{2}$. Hence, since $\psi_i\in C^{\infty}(B_1)$ for all $i\in I$, it follows that $\min_{B_{r_j}\setminus B_{r_{j+1}}}u_i(x)\rightarrow\infty$ as $j\rightarrow\infty$. Therefore, we conclude
		\begin{equation*}
			\min_{B_{r_{k}}\setminus B_{r_{k+1}}}(u_i-\psi_i)=\min_{\partial B_{r_{k}}\cup\partial B_{r_{k+1}}}(u_i-\psi_i),
		\end{equation*}    
		which proves the lemma.
	\end{proof}
	
	The next lemma is the main result of this subsection.
	
	\begin{lemma}
		Let $s=2^{**}-1$ and $\mathcal{U}\in C\left(B^*_{2},\mathbb{R}^p\right)\cap L^{2^{**}-1}\left(B_{2},\mathbb{R}^p\right)$ be a strongly positive solution to \eqref{rescaledintegralsystem}. If $\lim _{|x| \rightarrow0}|x|^{\gamma} |\mathcal{U}(x)|=0$, then $|\mathcal{U}|$ can be extended as a continuous function to the whole $B_1$.
	\end{lemma}
	
	\begin{proof}
		Let us consider the barrier functions from \cite{MR2055032}. For any $i\in I$ and $\delta>0$, we choose $0<\rho\ll 1$ such that $u_i(x)\leqslant\delta|x|^{-\gamma}$ in $B^*_{\rho}$.
		Fixing $\varepsilon>0$, $\kappa\in\left(0, \gamma\right)$ and $M\gg1$ to be chosen later, we define
		\begin{equation*}
			\varsigma_i(x)=
			\begin{cases}
				{M|x|^{-\kappa}+\varepsilon|x|^{4-n-\kappa},} & \mbox{if} \ {0<|x|<\rho}\\ 
				{u_i(x)},& \mbox{if} \ {\rho<|x|<2}.
			\end{cases}
		\end{equation*}
		Notice that for every $0<\kappa<n-4$ and $0<|x|<2$, one can use a change a variables to find $C>0$ such that
		\begin{align*}
			\int_{\mathbb{R}^{n}}{|x-y|^{4-n}|y|^{-4-\kappa}}\ud y 
			&=|x|^{4-n}\int_{\mathbb{R}^{n}}{\left||x|^{-1}x-|x|^{-1}y\right|^{4-n}|y|^{\-\kappa-4}}\ud y&\\
			&=|x|^{-\kappa+4}\int_{\mathbb{R}^{n}} {\left||x|^{-1}x-z\right|^{4-n}|z|^{\-\kappa-4}}\ud z&\\ 
			&\leqslant C\left(\frac{1}{n-4-\kappa}+\frac{1}{\kappa}+1\right)|x|^{-\kappa},
		\end{align*}
		which yields for $0<|x|<2$ and $0<\delta\ll1$,
		\begin{align*}
			\int_{B_{\rho}}{u_i^{2^{**}-2}(y)\varsigma_i(y)}{|x-y|^{4-n}}\ud y
			&\leqslant \delta^{2^{**}-2} \int_{\mathbb{R}^{n}}{\varsigma_i(y)}{|x-y|^{n-4}|y|^{-4}}\ud y&\\
			&\leqslant C\delta^{2^{**}-2} \varsigma_i(x)&\\ 
			&<\frac{1}{2} \varsigma_i(x).&
		\end{align*}
		Moreover, for $0<|x|<\rho$ and $\bar{x}=\rho x|x|^{-1}$, we get
		\begin{align*}
			\int_{B_{2}\setminus B_{\rho}}{u_i^{2^{**}-2}(y)\varsigma_i(y)}{|x-y|^{4-n}}\ud y
			&=\int_{B_{2}\setminus B_{\rho}}\frac{|\bar{x}-y|^{n-4}}{|x-y|^{n-4}}\frac{u_i^{2^{**}-1}(y)}{|\bar{x}-y|^{n-4}}\ud y&\\
			&\leqslant 2^{n-4}\int_{B_{2}\setminus B_{\rho}}\frac{u_i^{2^{**}-1}(y)}{|\bar{x}-y|^{n-4}}\ud y&\\ 
			&\leqslant 2^{n-4} u_i(\bar{x})&\\
			&\leqslant 2^{n-4}\max_{\partial B_{\rho}} u_i.
		\end{align*}
		The last inequality implies that for $0<|x|<\tau$ and $M\geqslant\max_{\partial B_{\rho}}u_i$,
		\begin{equation*}
			\varsigma_i(x)+\int_{B_{2}}\frac{u_i^{2^{**}-2}(y)\varsigma_i(y)}{|x-y|^{4-n}}\ud y \leqslant \varsigma_i(x)+2^{n-4}\max_{\partial B_{\rho}}u+\frac{1}{2}\varsigma_i(x)<\varsigma_i(x).
		\end{equation*}
		
		In the next claim, we show that $\varsigma_i$ can be taken indeed as a barrier for any $u_i$.
		
		\noindent{\bf Claim 1:} For any $i\in I$, it holds that $u_i(x)\leqslant \varsigma_i(x)$ in $B^*_{\rho}$. 
		
		\noindent Indeed, assume it does not hold. Since $u_i(x)\leqslant\delta|x|^{-\gamma}$ in $B^*_{\rho}$, by the definition of $\varsigma_i$, there exists $\bar{\tau} \in(0, \rho)$, depending on $\varepsilon$, such that $\varsigma_i\geqslant u_i$ in $B^*_{\widetilde{\rho}}$ and $\varsigma_i>u_i$ near $\partial B_{\rho}$. Let us consider, 
		\begin{equation*}
			\bar{\tau}:=\inf\left\{\tau>1 : \tau\psi_{i}>u_i \ \mbox{in} \ B^*_{\rho}\right\}.
		\end{equation*}
		Then, we have that $\bar{\tau}\in(1,\infty)$ and there exists $\bar{x}\in B_{\rho} \setminus\bar{B}_{\widetilde{\tau}}$ such that $\bar{\tau}\varsigma_i(\bar{x})=u_i(\bar{x})$. Furthermore, for    $0<|x|<\tau$, it follows 
		\begin{equation*}
			\bar{\tau}\varsigma_i(x)\geqslant\int_{B_{2}}{u_i^{2^{**}-2}(y)\bar{\tau} \varsigma_i(y)}{|x-y|^{4-n}}\ud y+\bar{\tau}\varsigma_i(x)\geqslant\int_{B_{2}}{u_i^{2^{**}-2}(y)\bar{\tau}\varsigma_i(y)}{|x-y|^{4-n}}\ud y+\varsigma_i(x),
		\end{equation*}
		which gives us,
		\begin{equation*}
			\bar{\tau}\varsigma_i(x)-u_i(x)\geqslant\int_{B_{2}}{u_i^{2^{**}-2}(y)(\bar{\tau} \varsigma_i(y)-u_i(y))}{|x-y|^{4-n}}\ud y.
		\end{equation*}
		Finally, by evaluating the last inequality at $\bar{x}\in B_{\rho}\setminus\bar{B}_{\widetilde{\tau}}$, we get a contradiction and the claim is proved.
		
		As a consequence of the claim, we get that
		$u_i(x)\leqslant\varsigma_i(x)\leqslant M|x|^{-\kappa}+\varepsilon|x|^{4-n-\kappa}$  in $B^*_{\rho}$,
		which, by passing to the limit as $\varepsilon\rightarrow0$, implies that $u_i^{2^{**}-2}\in L^{s}(B^*_{\rho})$ for some $s>{n}/{4}$ and any $i\in I$. 
		Hence, using standard elliptic regularity, the proof of the lemma follows.
	\end{proof}
	
	\begin{lemma}
		Let $s=2^{**}-1$ and $\mathcal{U}\in C\left(B^*_{2}\right)\cap L^{2^{**}-1}\left(B_{2}\right)$ be a strongly positive solution to \eqref{oursystem}. Assume $\psi_i\in C^{\infty}(B_{1})$ is a positive function in $\mathbb{R}^{n}$ satisfying $\Delta^{2} \psi_i=0$ in $B_{2}$ for all $i\in I$. If $\liminf _{|x|\rightarrow 0}|x|^{\gamma} |\mathcal{U}(x)|=0$, then $\lim _{|x|\rightarrow 0}|x|^{\gamma} |\mathcal{U}(x)|=0$.
	\end{lemma}
	
	\begin{proof}
		Assume by contradiction there exists $C>0$ such that $\lim\sup_{x\rightarrow0}|x|^{\gamma}|\mathcal{U}(x)|=C>0$; thus, from Lemma~\ref{liminflimsup}, we get that $\liminf_{|x|\rightarrow 0} |\mathcal{U}(x)|=\infty$. Using the assumption and the Harnack inequality in Lemma~\ref{harnack}, there exists $\{r_{k}\}_{k\in\mathbb{N}}$ such that $r_k\rightarrow 0$ and $r_{k}^{\gamma}\bar{u}_i\left(r_{k}\right)\rightarrow0$ as $k\rightarrow\infty$. As well as, $r_{k}$ is a local minimum point of $r^{\gamma} \bar{u}_i(r)$. Furthermore, let us define
		\begin{equation*}
			\varphi_{ki}(y)=\frac{u_i\left(r_{k}y\right)}{u_i\left(r_{k} {\bf e}_{n}\right)},
		\end{equation*}
		which in combination with \eqref{rescaledintegralsystem} provides,
		\begin{equation*}
			\varphi_{ki}(y)=\int_{B_{2/r_{k}}}{\left(r_{k}^{\gamma} u\left(r_{k}{\bf e}_{n}\right)\right)^{2^{**}-2}\varphi_{ki}(\eta)^{2^{**}-1}}{|y-z|^{4-n}}\ud z+\psi_{ki}(y) \quad \mbox{in} \quad B^*_{2/r_{k}},
		\end{equation*}
		where $\psi_{ki}(y)=u\left(r_{k} {\bf e}_{n}\right)^{-1} \psi_i\left(r_{k} y\right)$.
		
		\noindent{\bf Claim 1:} For any $i\in I$, it follows that $\lim_{k\rightarrow\infty}\varphi_{ki}(y)={1}/{2}(|y|^{4-n}+1)$ in $C_{\loc}^{2}\left(\mathbb{R}^{n} \setminus\{0\}\right)$.
		
		\noindent In fact, since $u_i\left(r_{k} {\bf e}_{n}\right)\rightarrow \infty$, we have that $\psi_{ki}(y)\rightarrow 0$ as $k\rightarrow\infty$ in $C^{n}_{\loc}\left(\mathbb{R}^{n}\right)$. Next, using the Harnack inequality, we obtain that $r_{k}^{\gamma}u_i\left(r_{k} {\bf e}_{n}\right) \rightarrow0$, and $\varphi_{ki}$ is locally uniformly bounded in $B^*_{2/ r_{k}}$. Hence, 
		\begin{equation*}
			\lim_{k\rightarrow\infty}\left(r_{k}^{\gamma}u_i\left(r_{k}{\bf e}_{n}\right)\right)^{2^{**}-2}\varphi_{ki}(y)^{2^{**}-1}=0 \quad \mbox{in} \quad C_{\loc}^{n}\left(\mathbb{R}^{n}\setminus\{0\}\right).
		\end{equation*}
		Thus, for any $\tau>1$, $0<|y|<\tau$ and $0<\varepsilon<|y|/100$, up to subsequences, it follows
		\begin{align*}
			&\lim _{k \rightarrow \infty} \int_{B_{\tau}}{\left(r_{i}^{\gamma} u\left(r_{k}{\bf e}_{n}\right)\right)^{2^{**}-2} \varphi_{ki}(z)^{2^{**}-1}}{|y-z|^{4-n}}\ud z&\\
			&=\lim_{k\rightarrow\infty}\int_{B_{\varepsilon}}{\left(r_{k}^{\gamma} u\left(r_{k} {\bf e}_{n}\right)\right)^{2^{**}-2}\varphi_{ki}(z)^{2^{**}-1}}{|y-z|^{4-n}}\ud z&\\
			&=|y|^{4-n}(1+\mathcal{O}(\varepsilon))\lim_{k\rightarrow\infty} \int_{B_{\varepsilon}}\left(r_{k}^{\gamma}u\left(r_{k}{\bf e}_{n}\right)\right)^{2^{**}-2}\varphi_{ki}(z)^{2^{**}-1}\ud z.&
		\end{align*}
		By sending $\varepsilon \rightarrow 0$, we have
		\begin{equation*}
			\lim _{k \rightarrow \infty} \int_{B_{\tau}}{\left(r_{k}^{\gamma} u\left(r_{k}{\bf e}_{n}\right)\right)^{2^{**}-2} \varphi_{ki}(z)^{2^{**}-1}}{|y-z|^{4-n}}\ud z=A|y|^{4-n},
		\end{equation*}
		for some $A\geqslant 0$. Moreover, since the left-hand side of the last equation is locally uniformly bounded in $C_{\loc}^{n+1}\left(B_{\tau}\right)$, for any $i\in I$, there exists $\varrho_i\in C^{2}\left(B_{\tau}\right)$  satisfying
		\begin{equation*}
			\lim_{k\rightarrow\infty}\int_{B_{2}\setminus B_{\tau}}{\left(r_{k}^{\gamma}u_i\left(r_{k} {\bf e}_{n}\right)\right)^{2^{**}-2}\varphi_{ki}(z)^{2^{**}-1}}{|y-z|^{4-n}}\ud z= \varrho_i(y)\geqslant 0 \quad \mbox{in} \quad C_{\loc}^{n}\left(B_{\tau}\right).
		\end{equation*}
		In addition, for any fixed $R\gg1$ and $y\in B_{\tau}$, we have
		\begin{equation*}
			\lim_{k\rightarrow\infty}\int_{t \leqslant|y| \leqslant R} {\left(r_{k}^{\gamma}u_i\left(r_{k} {\bf e}_{n}\right)\right)^{2^{**}-2}\varphi_{ki}(z)^{2^{**}-1}}{|y-z|^{4-n}}\ud z=0,
		\end{equation*}
		and for any $y_1, y_2\in B_{\tau}$, we obtain
		\begin{align*}
			&\int_{B_{2/r_{k}}\setminus B_{R}}{\left(r_{k}^{\gamma}u_i\left(r_{k} {\bf e}_{n}\right)\right)^{2^{**}-2}\varphi_{ki}(z)^{2^{**}-1}}{|y_1-z|^{4-n}}\ud z&\\
			&\leqslant\left(\frac{R+\tau}{R-\tau}\right)^{n-4}\int_{B_{2/r_{k}}\setminus B_{R}}{\left(r_{k}^{\gamma}u_i\left(r_{k} {\bf e}_{n}\right)\right)^{2^{**}-2}\varphi_{ki}(z)^{2^{**}-1}}{|y_2-z|^{4-n}}\ud z.&
		\end{align*}
		Therefore, it follows 
		\begin{equation*}
			\varrho_i\left(y_1\right)\leqslant\left(\frac{R+\tau}{R-\tau}\right)^{n-4} \varrho_i\left(y_2\right),
		\end{equation*}
		which, by passing to the limit as $R \rightarrow \infty$ and exchanging the roles of $y_1$ and $y_2$, implies $\varrho_i\left(y_2\right)=\varrho_i\left(y_1\right)$. Whence, $\varrho_i(y) \equiv \varrho_i(0)$ for all $y \in B_{\tau}$ and $i\in I$. Since $\varphi_{ki}$ is locally uniformly bounded in $B^*_{2/r_{k}}$, we have that it is locally uniformly bounded in $C^{n+1}(B^*_{2/r_{k}})$. Hence, up to subsequence, it follows that $\varphi_{ki}\rightarrow\varrho_i$ as $k \rightarrow\infty$ in $C_{\loc}^{n}\left(\mathbb{R}^{n}\setminus\{0\}\right)$, for some $\varrho_i$, which yields the following
		\begin{equation*}
			\lim_{k\rightarrow\infty}\varphi_{ki}(y)={A}{|y|^{4-n}}+\psi_i(0) \quad \mbox{in} \quad C_{\loc}^{n}\left(\mathbb{R}^{n}\setminus\{0\}\right).
		\end{equation*}
		Using that $\varphi_{ki}\left({\bf e}_{n}\right)=1$ and
		\begin{equation*}
			\frac{\ud}{\ud r}\left\{r^{\gamma} \bar{\varphi}_{ki}(r)\right\}\big|_{r=1}=r_{k}^{-\gamma+1}u_{i}\left(r_{k} {\bf e}_{n}\right)^{-1}\frac{\ud}{\ud r}\left\{r^{\gamma} \bar{u}_i(r)\right\}\Big|_{r=r_{k}}=0,
		\end{equation*}
		by taking the limit $k\rightarrow\infty$, it follows that $A=\psi_i(0)=1/2$, which proves the claim.
		
		In the next claim, we obtain some information about the limit of the Pohozaev invariant, which is used to generate a contradiction.
		
		\noindent{\bf Claim 2:} $\lim_{k\rightarrow\infty}\mathcal{P}_{\rm sph}\left(r_k,\mathcal{U}\right)=0.$
		
		\noindent In fact, for any $i\in I$, let us consider
		\begin{equation*}
			\widetilde{u}_i(x)=\int_{B_{2}}f_i(\mathcal{U}){|x-y|^{4-n}}\ud y+\psi_i(x) \quad \mbox{in} \quad \mathbb{R}^{n}\setminus\{0\},
		\end{equation*}
		which provides that $\widetilde{u}_i=u_i$ in $B_{2}\setminus\{0\}$, and
		\begin{equation*}
			\widetilde{u}_i(x)=\int_{B_{2}}f_i(\widetilde{\mathcal{U}}){|x-y|^{4-n}}\ud y+\psi_i(x) \quad \mbox{in} \quad \mathbb{R}^{n}\setminus\{0\}.
		\end{equation*}
		Consequently, using that $\Delta^{2} \psi_i=0$ in $B_{2}$ for any $i\in I$, it follows that $\Delta^{2}\widetilde{u}_i=f_i(\mathcal{U})$ in $B^*_{2}$.
		On the other hand, we know that $\mathcal{P}_{\rm sph}\left(r_k,\mathcal{U}\right)$ is a constant on $r$. Moreover, since $\left|D^{(j)}\varphi_{ki}\right|\leqslant C$ near $\partial B_{1}$ and
		$r_{k}^{\gamma} u\left(r_{k} {\bf e}_{n}\right)=o(1)$ as $k\rightarrow\infty$, we have
		$\left|D^{(j)} u_i(x)\right|\leqslant Cr_{k}^{-j} u\left(r_{k} {\bf e}_{n}\right)=o(1)r_{k}^{-\gamma-k}$ for all $|x|=r_{k}$ and $j=1,2,3,4$, which proves the second claim. 
		
		Hence, using Claim 2, it holds that $\mathcal{P}_{\rm sph}\left(r_{k},\mathcal{U}\right)=0$ for $k\in\mathbb{N}$. Thus, by \eqref{pohozaevsphericalautonomous}, we get the following identity
		\begin{equation*}
			\sum_{i=1}^p\int_{\partial B_{1}} q\left(\varphi_{ki}(x), \varphi_{ki}(x)\right)\ud x+\widehat{c}(n)\left(r_{k}^{\gamma} u_i\left(r_{i} e_{n}\right)\right)^{2^{**}-2}\sum_{i=1}^p\int_{\partial B_{1}} |\varphi_{ki}(x)|^{2^{**}}\ud x=0,
		\end{equation*}
		where we recall that $q$ is defined by \eqref{pohozaeverrorfunction}.
		Next, sending $k\rightarrow\infty$, and doing some manipulation, we obtain
		\begin{equation*}
			\int_{\partial B_{1}} q\left({|x|^{4-n}}+1,{|x|^{4-n}}+1\right)\ud x=0.
		\end{equation*}
		On the other hand, by Theorem~\ref{thm:andrade-doo19} (i), we know that  $\mathcal{U}_{0,\mu}(x)=\left(\frac{2\mu}{1+|x|^{2}\mu^{2}}\right)^{\gamma}$ satisfies the limit blow-up system,
		\begin{equation*}
			\Delta^{2} \mathcal{U}_{0,\mu}=c(n) f_i(\mathcal{U}_{0,\mu}) \quad \mbox{in} \quad \mathbb{R}^{n},
		\end{equation*}
		which implies that for any $\mu>0$, we get that $\mathcal{P}_{\rm sph}\left(\mathcal{U}_{0,\mu}, 1\right)=\lim_{r\rightarrow\infty} \mathcal{P}_{\rm sph}\left(\mathcal{U}_{0,\mu}, r\right)=0$.
		Hence, we find
		\begin{equation*}
			\sum_{i=1}^p\int_{\partial B_{1}} q\left(\mu^{-\gamma}(u_i)_{0,\mu}, \mu^{-\gamma}(u_i)_{0,\mu}\right)\ud x+ \widehat{c}(n)\mu^{4-n}\sum_{i=1}^p\int_{\partial B_{1}} (u_i)_{0,\mu}^{2^{**}}\ud x=0,
		\end{equation*}
		which, by taking the limit as $\mu\rightarrow 0$ provides
		\begin{align*}
			0&=\int_{\partial B_{1}} q\left({|x|^{4-n}}+1, {|x|^{4-n}}+1\right)\ud\sigma-\int_{\partial B_{1}} q\left({|x|^{4-n}}, {|x|^{4-n}}\right)\ud\sigma&\\
			&=(n-4)\int_{\partial B_{1}}\partial_{\nu}\Delta\left(|x|^{4-n}\right)\ud\sigma\neq0,& 
		\end{align*}
		which is a contradiction. This concludes the proof of the proposition.
	\end{proof}
	
	Consequently, using the last lemma and the barrier construction, we can present the removable singularity theorem.
	
	\begin{corollary}\label{removable}
		Let $\mathcal{U}$ be a strongly positive solution to \eqref{oursystem}. Then, $\mathcal{P}_{\rm sph}(\mathcal{U})\leqslant0$ and $\mathcal{P}_{\rm sph}(\mathcal{U})=0$, if, and only if, $\mathcal{U}$ has a removable singularity at the origin.
	\end{corollary}
	
	Ultimately, as a by-product of the last lemma and the removable singularity classification, we also have the lower bound estimate.
	
	\begin{corollary}\label{lowerbound}
		Let $\mathcal{U}$ be a strongly positive singular solution to \eqref{oursystem}. Then, there exists $C_1>0$ such that $    C_1|x|^{-\gamma}\leqslant |\mathcal{U}(x)|$ for $0<|x|<1/2$.
	\end{corollary}
	
	\begin{proof}[Proof of Proposition~\ref{estimates}]
		It is a consequence of Corollaries~\ref{upperbound} and \ref{lowerbound}.
	\end{proof}
	
	\section{Local asymptotic behavior}\label{sec:convergence}
	In this section, we present the proof of Proposition~\ref{convergence}, which, as a by-product, provides the proof of Theorem~\ref{theorem1} and Theorem~\ref{theorem1}'. 
	To this end, we use 
	the growth properties of the Jacobi fields in Proposition~\ref{growthpropertiessystem} and the a priori estimates in Proposition~\ref{estimates}. 
	First, inspired by \cite{MR2737708}, we summarize the Simon's convergence technique: 
	
	\noindent (a) There exist $C_1,C_2>0$ such that any solution to \eqref{sphevectfowler} satisfies the uniform estimate 
	\begin{equation*}
		C_1\leqslant |\mathcal{V}(t,\theta)|\leqslant C_2;
	\end{equation*}
	\noindent (b) If $\tau_k\rightarrow\infty$ and $\mathcal{V}_k(t,\theta):=\mathcal{V}(t+\tau_k,\theta)$. Then, the slide back sequence $\{\mathcal{V}_k\}_{k\in\mathbb{N}}$ converges uniformly on compact sets to a bounded solution $\mathcal{V}_{\infty}$ of \eqref{fowlersystem};\\
	\noindent (c) Any angular derivative $|\partial_{\theta}\mathcal{V}(t,\theta)|$ converges to $0$ as $t\rightarrow\infty$;\\
	\noindent (d) There exists $S>0$ such that for any infinitesimal rotation $\partial_{\theta}$ of $\mathbb{S}^{n-1}$, and for any $\tau_k\rightarrow\infty$, if we set $A_k=\sup_{t\geqslant0}|\partial_{\theta}\mathcal{V}(t,\theta)|$, and if $|\partial_{\theta}\mathcal{V}_k(\tau_k,\theta)|=A_k$ for some $(\tau_k,\theta_k)\in \mathcal{C}_0$, then $s_k\leqslant S$;\\
	\noindent (e) $|\partial_{\theta}\mathcal{V}(t,\theta)|$ converges to $0$ exponentially as $t\rightarrow\infty$, as well as $|\mathcal{V}(t,\theta)-\overline{\mathcal{V}}(t)|$, where $\overline{\mathcal{V}}$ is a spherical average of $\mathcal{V}$;\\
	\noindent (f) There exists a bounded solution $\mathcal{V}_{a,T}$ of \eqref{fowlersystem} and $\sigma\geqslant0$ such that $\mathcal{V}(t,\theta)$ converges to $\mathcal{V}(t+\sigma)$ exponentially as $t\rightarrow\infty$;
	
	\begin{remark}
		Using Theorem~\ref{thm:andrade-doo19} and Propositions~\ref{growthpropertiessystem} and \ref{estimates} most of the steps above follows the same lines of \cite{MR1666838,MR4002167}. 
		Regardless, we include all the proofs here for the convenience of the reader. 
		The main difference is the number of Jacobi fields to analyze; for the second order equation, we have two linearly independent Jacobi fields, one that grows unbounded and the other that is exponentially decreasing.
		In contrast, in our case, we have four linearly independent Jacobi fields that behave in the same way; that is, two of them grow unbounded, and the others are exponentially decaying. 
		Surprisingly, this Jacobi field basis turns out to have only two elements on the zero frequency case. 
	\end{remark}
	
	\subsection{Simple convergence}\label{subsec:simpleconvergence}
	Here we prove a result that is equivalent to Proposition~\ref{convergence} written in cylindrical coordinates.
	
	\begin{proposition}
		Let $\mathcal{V}$ be a strongly positive singular solution to \eqref{sphevectfowler} satisfying \eqref{sharpasymp}. Then, there exists $\beta_0^*>0$ and an Emden--Fowler solution $\mathcal{V}_{a,T}$ such that 
		\begin{equation}\label{asymptoticcyl}
			\mathcal{V}(t)=(1+\mathcal{O}(e^{\beta^*_0 t}))\mathcal{V}_{a,T}(t) \quad {\rm as} \quad t\rightarrow\infty.
		\end{equation}
	\end{proposition}
	
	\begin{proof}
		Initially, by Remark~\ref{negativepohozaev}, the origin is a non-removable singularity. Thus, using Corollary~\ref{removable}, we have that $\mathcal{P}_{\rm sph}(\mathcal{U})<0$. Consider $\mathcal{V}=\mathfrak{F}(\mathcal{U})$ and $\{\tau_k\}_{k\in\mathbb{N}}$ such that $\tau_k\rightarrow\infty$ as $k\rightarrow\infty$. Let us define the sequence of translations,
		\begin{equation*}
			\mathcal{V}_k(t,\theta)=\mathcal{V}(t+\tau_k,\theta) \quad \mbox{defined in} \quad \mathcal{C}_{\tau_k}:=(-\tau_k,\infty)\times\mathbb{S}^{n-1}.
		\end{equation*}
		Again, applying the estimates in  Proposition~\ref{estimates}, we get 
		\begin{equation}\label{cylindricalestimates}
			C_1\leqslant |\mathcal{V}_k(t,\theta)|\leqslant C_2.
		\end{equation}
		By \eqref{cylindricalestimates}, we find $\{\mathcal{V}_k\}_{k\in\mathbb{N}}$ is uniformly bounded in $C^{4,\zeta}_{\loc}(\mathcal{C}_0,\mathbb{R}^p)$ for some $\zeta\in(0,1)$. Hence, by standard elliptic regularity, there exists a limit solution $\mathcal{V}_{\infty}\in C^{4,\zeta}_{\loc}(\mathbb{R},\mathbb{R}^p)$ such that, up to subsequence, $\mathcal{V}_k\rightarrow \mathcal{V}_{\infty}$, and $\mathcal{V}_{\infty}$ satisfies \eqref{sphevectfowler}.
		Thus, by Theorem~\ref{thm:andrade-doo19} (ii), $\mathcal{V}_{\infty}$ is an Emden--Fowler solution, that is, there exist $a\in(0,a_0)$ and $T\in(0,T_a)$ such that $\mathcal{V}_{\infty}=\mathcal{V}_{a,T}$ and does not depend on the variable $\theta$. 
		
		\noindent{\bf Claim 1:} The following elliptic estimates hold:\\
		\noindent{(i)} $\mathcal{V}_k(t,\theta)=\overline{\mathcal{V}}_k(t)(1+o(1))$;\\
		\noindent{(ii)} $\nabla \mathcal{V}_k(t,\theta)=-\overline{\mathcal{V}}^{(1)}_k(t)(1+o(1))$;\\
		\noindent{(iii)} $\Delta \mathcal{V}_k(t,\theta)=\overline{\mathcal{V}}^{(2)}_k(t)(1+o(1))$;\\
		\noindent{(iv)} $\nabla^{3/2}\mathcal{V}_k(t,\theta)=-\overline{\mathcal{V}}^{(3)}_k(t)(1+o(1))$.
		
		\noindent Indeed, if (i) is not valid, there would exist $\bar{\varepsilon}>0$ and $\tau_k\rightarrow\infty$, $\theta_k\rightarrow\infty$ such that 
		\begin{equation*}
			\left|\frac{\mathcal{V}_k(\tau_k,\theta_k)}{\overline{\mathcal{V}}_k(\tau_k)}-1\right|\geqslant\bar{\varepsilon},
		\end{equation*}
		which is a contradiction since $\mathcal{V}_k\rightarrow \mathcal{V}_{\infty}$ and $\mathcal{V}_{\infty}$ is radially symmetric. The same argument holds for (ii), (iii) and (iv). This estimate implies (b), that is, any angular derivative $|\partial_{\theta}\mathcal{V}_k|$ converges uniformly to zero.
		
		\noindent{\bf Claim 2:} The necksize of $\mathcal{V}_{\infty}$ does not depend on $k\in\mathbb{N}$.
		
		\noindent In fact, this is a consequence of the following identity
		\begin{equation*}
			\mathcal{P}_{\rm cyl}(\mathcal{V}_{\infty}):={\mathcal{P}}_{\rm cyl}(0,\mathcal{V}_{\infty})=\lim_{k\rightarrow\infty}{\mathcal{P}}_{\rm cyl}(0,\mathcal{V}_k)=\lim_{k\rightarrow\infty}{\mathcal{P}}_{\rm cyl}(\tau_k,\mathcal{V})=\mathcal{P}_{\rm cyl}(\mathcal{V}).
		\end{equation*}
		Consequently, we have that for each $\{\tau_k\}_{k\in\mathbb{N}}$, the correspondent sequence $\{\mathcal{V}_k\}_{k\in\mathbb{N}}$ converges to $\mathcal{V}_{a,T}:=\Lambda^*v_{a,T}$ as $k\rightarrow\infty$, where $T$ does not depend on $k$.
		
		In the next claim, we prove (c), (d), (e), and (f).
		
		\noindent{\bf Claim 3:} There exist $\sigma\in\mathbb{R}$ and $\beta^*_0>0$ such that
		\begin{equation*}
			\left|\mathcal{V}_{\sigma}(t,\theta)-\mathcal{V}_{a,T}(t)\right|\leqslant Ce^{\beta^*_0 t} \quad \mbox{on} \quad \mathcal{C}_0.
		\end{equation*}
		
		\noindent As a matter of fact, we divide the rescaling argument into three steps as follows.
		
		First, let $T_a\in\mathbb{R}$ be the fundamental period of the Emden--Fowler solution $v_{a,T}$ and define $A_{\tau}=\sup_{t\geqslant0}|\partial_{\theta}\mathcal{V}_{\tau}|$.
		Since $|\partial_{\theta}\mathcal{V}_{\tau}|$ converges uniformly to zero as $t\rightarrow\infty$, we have $A_{\tau}<\infty$.
		
		\noindent{\bf Step 1:} For every $c>0$, there exists an integer $N>0$ such that, for any $\tau>0$ either:\\
		\noindent{(i)} $A_{\tau}\leqslant ce^{-2\tau}$, or\\
		\noindent{(ii)} $A_{\tau}$ is attained at some point in $\overline{\mathcal{C}}_{0,I_N}:=I_N\times\mathbb{S}^{n-1}$, where $I_N=[0,NT_a]$.
		
		\noindent Suppose that the claim is not true. Then, there exists $C>0$ and $\tau_k,\theta_k\rightarrow\infty$ such that $|\partial_{\theta}\mathcal{V}_{\tau}(s_k,\theta_k)|=A_{\tau_k}$ and $A_{\tau_k}>Ce^{-2\tau_k}$ as $k\rightarrow\infty$. We define $\widetilde{\mathcal{V}}_k(t,\theta)=\mathcal{V}_k(t+s_k,\theta)$ and $\Phi_k=A^{-1}_{\tau_k}\partial_{\theta}\widetilde{\mathcal{V}}_k$. In addition, we have that $|\Phi_k|\leqslant1$ and satisfy the nonlinear system $\mathcal{N}(\widetilde{\mathcal{V}}_k)=0$, which by differentiation with respect to $\theta$ implies that $\mathcal{L}^a\left(\Phi_k\right)=0$. Now, using standard elliptic regularity, we can extract a subsequence $\{\Phi_k\}_{k\in\mathbb{N}}$ which converges in compact subsets to a nontrivial and bounded Jacobi field satisfying $\mathcal{L}^a\left(\Phi\right)=0$. This is a contradiction since $\Phi$ has no zero eigencomponent relative to $\Delta_{\theta}$ and thus is unbounded. This proves Step 1.
		
		Assuming that $\mathcal{V}_k(t,\theta)$ converges to $\mathcal{V}_a(t+T)$ as $k\rightarrow\infty$, let us define 
		\begin{equation*}
			\mathcal{W}_k(t,\theta)=\mathcal{V}_k(t,\theta)-\mathcal{V}_a(t+T), \quad \eta_k=\varrho\max_{I_N}|\mathcal{W}_k| \quad \mbox{and} \quad \Phi_k=\eta_{k}^{-1}\mathcal{V}_k,
		\end{equation*}
		where $\varrho>0$ will be chosen later and satisfies $|\Phi_k|\leqslant \varrho^{-1}$ in $I_N$. Then, by Theorem~\ref{thm:andrade-doo19}, it follows 
		\begin{equation}\label{princeton}
			\Delta^2_{\rm cyl}\mathcal{W}_k-\left[f_i(\mathcal{V}_k)-\Lambda^*v_{a,T}^{2^{**}-1}\right]=0,
		\end{equation}
		where $\Lambda^*=(\Lambda^*_1,\dots,\Lambda^*_p)$ and
		\begin{equation*}
			|\mathcal{V}_k|^{2^{**}-2}v_{ki}-\Lambda^*v_{a,T}^{2^{**}-1}=|\mathcal{V}_k|^{2^{**}-2}v_{ki}+
			\Lambda_i^*v_{a,T}\frac{|\mathcal{V}_k|^{2^{**}-2}-v_{a,T}^{2^{**}-2}}{|\mathcal{V}_k|^{2}-v_{a,T}^{2}}\sum_{j\in I}\left(v_{kj}+\Lambda^*_iv_{a,T}\right).
		\end{equation*}
		Multiplying \eqref{princeton} by $\eta_k^{-1}$ and taking the limit as $k\rightarrow\infty$ we get $\mathcal{L}^a\left(\Phi^{*}\right)=0$, where $\Phi^{*}=\lim_{k\rightarrow\infty}\Phi_k$.
		
		\noindent{\bf Step 2:} The Jacobi filed $\Phi^{*}$ is bounded for all $t\geqslant0$.
		
		\noindent Using Proposition~\ref{growthpropertiessystem} and the Fourier decomposition \eqref{Fourierdecomposition}, we get
		\begin{equation*}
			\Phi^{*}=b_1\Phi_{a,0}^{+}+b_2\Phi_{a,0}^{-}+\check{\Phi},
		\end{equation*} 
		where $\check{\Phi}$ is the projection on the subspace generated by the eigenfunctions associated to the nonzero eigenvalues of $\Delta_{\theta}$.
		Using Proposition~\ref{growthpropertiessystem}, we claim that $\check{\Phi}$ is bounded. Indeed, we need to verify that $\partial_{\theta}\check{\Phi}=\partial_{\theta}\Phi$ is bounded for $t\geqslant0$. 
		In this fashion, we have that $\partial_{\theta}\Phi=\lim_{k\rightarrow\infty}\eta_k\partial_{\theta}\mathcal{V}_k$. Furthermore, if $\partial_{\theta}\Phi$ is zero the result follows. Then, we suppose that $\partial_{\theta}\Phi$ is nontrivial. In this case, if (i) of Step 1 happens, we get
		\begin{equation*}
			\sup_{t\geqslant0}\left(\eta_k^{-1}|\partial_{\theta}\mathcal{V}_k|\right)\leqslant C\eta_k^{-1}e^{-2\tau_k}\leqslant C.
		\end{equation*}
		On the other hand, if (ii) of Step 1 happens, since $\eta_k^{-1}|\partial_{\theta}\mathcal{V}_k|$ converges in the $C^4$-topology, we have
		\begin{equation*}
			\sup_{t\geqslant0}\left(\eta_k^{-1}|\partial_{\theta}\mathcal{V}_k|\right)\leqslant \sup_{I_N}\left(\eta_k^{-1}|\partial_{\theta}\mathcal{V}_k|\right)\leqslant C.
		\end{equation*}
		The last two inequalities imply the boundedness of $\check{\Phi}$. 
		
		To finish the proof of Step 2, we must show that $b_2=0$. Indeed, since $\Phi_k=\eta_k^{-1}\mathcal{W}_k\rightarrow\Phi$ as $k\rightarrow\infty$, we obtain
		\begin{align*}
			\mathcal{V}_k&=\mathcal{V}_{a,T}+\eta_k\Phi^{*}+o(\eta_k)&\\
			&=\mathcal{V}_{a,T}+\eta_k(b_1\Phi_{a,0}^{+}+b_2\Phi_{a,0}^{-}+\check{\Phi})+o(\eta_k).&
		\end{align*}
		On the other hand, 
		\begin{equation*}
			\mathcal{P}_{\rm cyl}(0,\mathcal{V}_k)=\mathcal{P}_{\rm cyl}(\tau_k,\mathcal{V})=\mathcal{P}_{\rm cyl}(\mathcal{V})+\mathcal{O}(e^{-2\tau_k})=\mathcal{P}_{\rm cyl}(T,\mathcal{V}_a)+\mathcal{O}(e^{-2\tau_k}).
		\end{equation*}
		Then, if $b_2\neq0$, we would have a contradiction, since $\eta_k^{-1}e^{-2\tau_k}=o(1)$ as $k\rightarrow\infty$ and the two sides of the last equality would differ for sufficiently large $k$.
		
		Let us define
		\begin{equation*}
			\mathcal{W}_{\tau}(t,\theta)=\mathcal{V}(t+\tau,\theta)-\mathcal{V}_a(t+T) \quad \mbox{and} \quad \eta(\tau)=\varrho\max_{I_N}|\mathcal{W}_\tau|,
		\end{equation*}
		where $I_N$ is defined in Step 1 and $\varrho>0$ will again be chosen later. For a fixed $C_1>0$, we have the following:
		
		\noindent{\bf Step 3:} Assume that $N,\varrho,\tau\gg1$ and $0<\eta\ll1$. Then, there exists $|\delta|\leqslant C_1\eta(\tau)$ such that
		\begin{equation}\label{correction}
			2\eta(\tau+NT_a+\delta)\leqslant\eta(\tau).
		\end{equation}
		
		\noindent Suppose that \eqref{correction} does not hold. Then, there would exist some $\tau_k\rightarrow\infty$ such that $\eta(\tau_k)\rightarrow0$ and for $s>0$ satisfying $|s|\leqslant C_1\eta(\tau_k)$ we have $\eta(\tau_k+NT_a+s)>1/2\eta(\tau_k)$. Similarly to the previous step, let us define $\Phi_k=\eta(\tau_k)^{-1}\mathcal{W}_{\tau_k}$; thus by Step 2, we can suppose that $\{\Phi_k\}_{k\in\mathbb{N}}$ converges to a bounded Jacobi Field $\Phi^{*}$, which provides
		\begin{equation}\label{havard}
			\Phi^{*}=b_1\Phi_{a,0}^{+}+\check{\Phi},
		\end{equation}
		where $\check{\Phi}$ has exponential decay. Since $|\check{\Phi}|<\varrho^{-1}$ on $I_N$, we get that $b_1$ is uniformly bounded and independent of $\tau_k>0$. Moreover, we know that $\Phi_{a,0}^{+}=\partial_{a}\mathcal{V}_{a,T}$ is bounded and $\Phi_{a,0}^{-}=\partial_{\theta}\mathcal{V}_{a,T}$ is linearly growing.
		Setting $s_k=-\eta(\tau_k)b_1$, we can choose $C_1\gg1$ sufficiently large such that $|s_k|<|C_1\eta(\tau_k)|$. Hence, for $t\in[0,2NT_a]$, we get
		\begin{align*}
			\mathcal{W}_{\tau_k+s_k}(t,\theta)&=\mathcal{V}(t+\tau_k-\eta(\tau_k)b_1,\theta)-\mathcal{V}_{a,T}(t)&\\
			&=\mathcal{V}_{\tau_k}(t-\eta(\tau_k)b_1,\theta)-\mathcal{V}_{a,T}(t-\eta(\tau_k)b_1)-\eta(\tau_k)b_1\frac{\mathcal{V}_{a,T}(t-\eta(\tau_k)b_1)-\mathcal{V}_{a,T}(t)}{-\eta(\tau_k)b_1}&\\
			&=\eta(\tau_k)\Phi_k(t-\eta(\tau_k)b_1,\theta)-\eta(\tau_k)b_1\Phi^{+}_{a,0}+o(\eta(\tau_k))&\\
			&=\mathcal{W}_{\tau_k}(t,\theta)-\eta(\tau_k)b_1\Phi^{+}_{a,0}+o(\eta(\tau_k)).&
		\end{align*}
		Therefore, by \eqref{havard}, we obtain
		\begin{equation*}
			\mathcal{W}_{\tau_k+s_k}=\eta(\tau_k)\check{\Phi}+o(\eta(\tau_k)) \quad \mbox{on} \quad [0,2NT_a],
		\end{equation*}
		which implies
		\begin{equation*}
			\max_{I_N}|\mathcal{W}_{\tau_k+s_k+NT_a}|=\max_{[NT_a,2NT_a]}|\mathcal{W}_{\tau_k+s_k}|\leqslant\eta_{\tau_k}\max_{I_N}(|\check{\Phi})+o(\eta(\tau_k)).
		\end{equation*}
		Then, since $|\check{\Phi}|$ decreases exponentially in a fixed rate, one can choose $N,\varrho\gg1$ sufficiently large satisfying
		$\max_{I_N}|w_{\tau_k+s_k+NT_a}|\leqslant2^{-1}\eta(\tau_k)$, which  implies $\eta(\tau+NT_a+s)\leqslant\eta(\tau)$. 
		This is a contradiction, and this step is proved.
		
		Now we use Step 3 to construct a sequence that converges to the correct translation parameter. 
		
		\noindent{\bf Step 4:} There exists $\sigma>0$ such that $|\mathcal{W}_{\sigma}(t,\theta)|$ converges exponentially to $0$ as $t\rightarrow\infty$.
		
		\noindent First, choose $\tau_0,N\gg1$ such that Step 3 is satisfied and $C_1\eta(\tau_0)\leqslant 2^{-1}NT_a$. Set $s_0=-\eta(\tau_0)b_1$ as above; thus $|s_0|\leqslant C_1\eta(\tau_0)\leqslant 2^{-1}NT_a$. Let us define inductively three sequences:
		\begin{align*}
			&\sigma_k=\tau_0+\sum_{i=0}^{k-1}s_i,&\\
			&\tau_k=\tau_{k-1}+s_{k-1}+NT_a=\sigma_k+kNT_a,&\\
			&s_k=-\eta(\tau_k)b_1.&
		\end{align*}
		Notice that by induction, it follows that $\eta(\tau_k)\leqslant2^{-k}\eta(\tau_0)$ and $|s_k|<2^{k-1}NT_a$. Then, there exists $\sigma=\lim_{k\rightarrow\infty}\sigma_k\leqslant\tau_0+NT_a$ and so $\tau_k\rightarrow\infty$ as $k\rightarrow\infty$. 
		Now choosing $k\in\mathbb{N}$ such that $t=kNT_a+[t]$ with $[t]\in I_N$, we can write
		\begin{align*}
			\mathcal{W}_{\sigma}(t,\theta)&=\mathcal{V}(t+\sigma,\theta)-\Lambda^*v_{a,T}&\\
			&=\mathcal{V}(t+\sigma,\theta)-\mathcal{V}(t+\sigma_k,\theta)+\mathcal{V}(t+\sigma_k,\theta)-\Lambda^*v_{a,T}(t).&
		\end{align*}
		In addition, since $\partial_t\mathcal{V}$ is uniformly bounded, we get
		\begin{equation*}
			\mathcal{V}(t+\sigma,\theta)-\mathcal{V}(t+\sigma_k,\theta)=\partial_{t}\mathcal{V}(t_0)\sum_{i=k}^{\infty}s_i=\mathcal{O}(2^{-k}),
		\end{equation*}
		for some $t_0>0$ and $\mathcal{V}(t+\sigma_k,\theta)-\mathcal{V}_{a,T}(t)=\mathcal{V}(\tau_k+[t],\theta)-\mathcal{V}_{a,T}([t])=\mathcal{W}([t],\theta)$, which provides
		\begin{equation*}
			\mathcal{W}_{\sigma}(t,\theta)=\mathcal{W}_{\tau_k}([t],\theta)+\mathcal{O}(2^{-k}).
		\end{equation*}
		Therefore, using that $b\max_{I_N}|\mathcal{W}_{\tau_k}(t,\theta)|=\eta(\tau_k)\leqslant2^{-k}\eta(\tau_0)$, we obtain that $|\mathcal{W}_{\sigma}(t,\theta)|=\mathcal{O}(2^{-k})$ as $k\rightarrow\infty$, or in terms of $t=kNT_a+[t]$,
		\begin{equation*}
			|\mathcal{W}_{\sigma}(t,\theta)|\leqslant C_1e^{-\frac{\ln 2}{NT_a}t},
		\end{equation*}
		which, by taking $\beta_0^*=-{\ln 2}/{NT_a}$, concludes the proof of Step 4 .
		
		Finally, we observe that Claim 3 directly implies \eqref{asymptoticcyl}.
	\end{proof}
	
	\begin{proof}[Proof of Proposition~\ref{convergence}]
		It follows by using Propositions~\ref{estimates}, and undoing the cylindrical transformation in \eqref{asymptoticcyl}.
	\end{proof}
	
	\subsection{Deformed Emden--Fowler solutions}\label{subsec:defomerdfowler}
	Before we present the proof of Theorem~\ref{theorem1}, we follow \cite{MR1666838} to introduce the family of deformed Emden--Fowler solutions. 
	We construct this $2n$-parameter family of solutions generated by the Emden--Fowler solution.
	
	The idea is to use the pullback of a composition of three conformal transformations, and it is described as follows.
	First, when $p=1$, let us consider an Emden--Fowler solution with $T=0$, given by $u_{a,0}(x)=|x|^{-\gamma}v_a(-\ln|x|)$; thus, using an inversion about the unit sphere, we obtain that $\widetilde{u}_{a,0}(x)=|x|^{-\gamma}v_a(\ln|x|)$. Second, we perform an Euclidean translation about $x_0\in\mathbb{R}^{n}\setminus\{0\}$ to get
	\begin{equation*}
		\widetilde{u}_{a,0,x_0}(x)=|x|x|^{-2}-x_0|^{-\gamma}v_a(\ln|x|x|^{-2}-x_0|).
	\end{equation*}
	Applying another inversion, we find
	\begin{equation*}
		u_{a,0,x_0}(x)=|x|^{-\gamma}|\theta-x_0|x||^{-\gamma}v_a(-\ln|x|+\ln|\theta-x_0x|),
	\end{equation*}
	where $\theta=x|x|^{-1}$. 
	Moreover, in cylindrical coordinates, we have
	\begin{equation}\label{asymptoticsjacobifields1}
		v_{a,0,x_0}(t,\theta)=|\theta-x_0e^{-t}|^{-\gamma}v_a(t+\ln|\theta-x_0e^{-t}|),
	\end{equation}
	which has a singularity at $t=\ln a$ and $\theta=a/|a|$.
	Finally, taking a time translation $T\in(0,T_a)$, we construct the families $u_{a,T,x_0}$ and $v_{a,T,x_0}$.
	The parameters $x_0\in\mathbb{R}^n$ and $T\in(0,T_a)$ correspond to conformal motions, whereas the {\it Fowler parameter} $a\in(0,a_0)$ does not.
	
	\begin{definition}\label{def:deformedfowler}
		For $p>1$, let us define the {\it deformed $($vectorial$)$ Emden--Fowler solution} by 
		\begin{equation*}
			\mathcal{U}_{a,T,x_0}=\Lambda^*u_{a,T,x_0} \quad \mbox{and} \quad \mathcal{V}_{a,T,x_0}=\Lambda^*v_{a,T,x_0},
		\end{equation*} 
		where $\Lambda^*\in\mathbb{S}_{+,*}^{p-1}$ and $u_{a,T,x_0}$ and $v_{a,T,x_0}$ are scalar Emden--Fowler solutions.
	\end{definition}
	
	\begin{proposition}
		For any $a\in(0,a_0)$ and $x_0\in\mathbb{R}^n$, we have 
		\begin{equation*}
			\mathcal{U}_{a,0,x_0}(x)=(1+\mathcal{O}(|x|))\mathcal{U}_{a,0}(|x|) \quad {\rm as} \quad x\rightarrow0.
		\end{equation*}
	\end{proposition}
	
	\begin{proof}
		Initially, we take $p=1$ and calculate the Taylor series of $u_{a,0,x_0}$ near $|x|=0$,
		\begin{equation}\label{bus}
			|x|x|^{-1}-x_0|x||^{\gamma}=1+\gamma(x_0\cdot x)+\mathcal{O}(|x|^2).
		\end{equation}
		Similarly,
		\begin{equation*}
			\ln|x|x|^{-1}-x_0|x||=-(x_0\cdot x)+\mathcal{O}(|x|^2),
		\end{equation*}
		which implies
		\begin{equation}\label{train}
			v_a(-\ln|x|-(x_0\cdot x)+\mathcal{O}(|x|^2))=v_a(-\ln|x|)-v_a^{(1)}(-\ln|x|)(x_0\cdot x)+\mathcal{O}(|x|^2).
		\end{equation}
		By combining \eqref{bus} and \eqref{train}, we obtain
		\begin{align*}
			u_{a,0,x_0}(x)
			&=|x|^{-\gamma}\left[v_a(-\ln|x|)+(x_0\cdot x)\left(v_a^{(1)}(-\ln|x|)+\gamma v_a(-\ln|x|)\right)+\mathcal{O}(|x|^2)\right]&\\\nonumber
			&=u_{a,0}(x)+|x|^{-\gamma}(x_0\cdot x)\left(-v_a^{(1)}+\gamma v_a\right)+\mathcal{O}(|x|^{-\gamma-2}),&
		\end{align*}
		which together with Theorem~\ref{thm:andrade-doo19} implies
		\begin{align}\label{matematica}
			\mathcal{U}_{a,0,x_0}(x)=\mathcal{U}_{a,0}(x)+|x|^{-\gamma}(x_0\cdot x)\left(-\mathcal{V}_a^{(1)}+\gamma \mathcal{V}_a\right)+\mathcal{O}(|x|^{-\gamma-2});
		\end{align}
		this concludes the proof of the proposition.
	\end{proof}
	
	\begin{remark}\label{signjacobifields}
		Since the Jacobi fields and the indicial roots of the linearized operator are not counted with multiplicity, these functions are all equal; thus, it is simpler to write $\Phi^{+}_{a,1}=\cdots=\Phi^{+}_{a,n}$ and \eqref{matematica} can be reformulated as
		\begin{align}\label{lastequation}
			\mathcal{U}_{a,0,x_0}(x)&=|x|^{\gamma}\left[v_a\left(-\ln|x|+|x|\left(\sum_{j=1}^{n}x_j\chi_j(\theta)\Phi_{a,j}\right)+\mathcal{O}(|x|^2)\right)\right]&\\\nonumber
			&=|x|^{\gamma}\left[\mathcal{V}_a\left(-\ln|x|+(x_0\cdot x)\Phi^{+}_{a,1}(-\ln|x|)+\mathcal{O}(|x|^2)\right)\right] \quad {\rm as} \quad x\rightarrow0.&
		\end{align}
		In cylindrical coordinates, we can rewrite
		\begin{equation}\label{asymptoticsjacobifields2}
			\mathcal{V}_{a,0,x_0}(t,\theta)=\mathcal{V}_{a}(t)+e^{-t}\langle\theta, a\rangle\left(-\mathcal{V}_a^{(1)}+\gamma \mathcal{V}_a\right)+\mathcal{O}\left(e^{-2t}\right) \quad {\rm as} \quad t\rightarrow\infty.
		\end{equation} 
		Nevertheless, for the translation $\widetilde{\mathcal{V}}_{a,x_0}(t,\theta)=\mathcal{V}_a(t-t_0,\theta)$ with $t_0=-\ln|x_0|$, we have 
		\begin{equation*}
			\widetilde{\mathcal{V}}_{a,x_0}(t,\theta)=e^t\left(-\mathcal{V}_a^{(1)}+\gamma \mathcal{V}_a\right)+\mathcal{O}(1) \quad {\rm as} \quad t\rightarrow\infty.
		\end{equation*}
		Also, notice that if $\langle a, \theta\rangle> 0$, then $|\mathcal{V}_{a,0,x_0}(t,\theta)|>|\mathcal{V}_{a}(t)|$ and $|\overline{\mathcal{V}}_{a,0,x_0}(t,\theta)|>|\overline{\mathcal{V}}_{a}(t)|$. Moreover, the opposite inequality holds when $\langle a,\theta\rangle< 0$.
	\end{remark}
	
	\subsection{Proof of Theorem~\ref{theorem1}}\label{subsec:improvedconvergence}
	To complete our analysis, we present the proof of Theorem~\ref{theorem1}. 
	Our proof is based on the surjectiveness of the linearized operator stated in Proposition~\ref{cor:surjectiveness}. 
	In the next result, we provide a higher order expansion for solutions to \eqref{oursystem} with $R=\infty$ and $p=1$, which can be stated as
	
	\begin{proposition}\label{prop:imporvedconvergence}
		Let $\mathcal{U}$ be a singular solution to \eqref{oursystem}. Then, for any $x_0\in\mathbb{R}^n$ there exists an Emden--Fowler solution $\mathcal{V}_{a,T}$ such that
		\begin{equation}\label{fisica2}
			\mathcal{U}(x)=|x|^{-\gamma}\left[\mathcal{V}_a(-\ln|x|+T)+(x_0\cdot x)\Phi^{+}_{a,1}(-\ln|x|+T)+\mathcal{O}(|x|^{\beta^*_1})\right] \quad {\rm as} \quad x\rightarrow0,
		\end{equation}
		where $\beta^*_1:=\min\{2,\beta_{a,2}\}>0$.
	\end{proposition}
	
	\begin{proof}
		Again, we start with $p=1$. Using the asymptotics proved in Theorem~\ref{theorem1}, we deduce     
		\begin{equation*}
			u(x)=|x|^{-\gamma}v_a(-\ln|x|)=|x|^{-\gamma}\left[v_a(-\ln|x|+T_a)+w(-\ln|x|)\right],
		\end{equation*}
		where $\phi\in C^{4,\zeta}_{-\beta}(\mathcal{C}_0)$ for some $\zeta\in(0,1)$ and $\beta>0$.
		Moreover, since $v_a$ satisfies \eqref{fowler4order}, that is, $\mathcal{N}(v_a)=0$, we can expand $\mathcal{N}$ in Taylor series around $v_{a,T}$ to get $\mathcal{L}^{a}(\phi)=\psi(\phi)$, where
		\begin{equation*}
			\psi(\phi)=(v_a+\phi)^{2^{**}-1}-v_a^{2^{**}-1}-\widetilde{c}(n)v_a^{2^{**}-2}\phi.
		\end{equation*}
		It is straightforward to see that if $\phi\in C^{m,\zeta}_{-\beta}(\mathcal{C}_0)$, then $\psi(\phi)\in C^{m,\zeta}_{-2\beta}(\mathcal{C}_0)$ for any $m\in\mathbb{N}$. 
		Now, we can run the iterative method. 
		First, assume that $\beta\in (0,1/2)$, then using Claim 1, we obtain $\psi(\phi)\in C^{0,\zeta}_{-2\beta}(\mathcal{C}_0)$. 
		In addition, by (i) of Corollary~\ref{iteration}, we have $\phi\in C^{4,\zeta}_{-2\beta}(\mathcal{C}_0)$ and $\psi(\phi)\in C^{4,\zeta}_{-4\beta}(\mathcal{C}_0)$, which implies $w\in C^{4,\zeta}_{-4\beta}(\mathcal{C}_0)$. 
		After some steps, we conclude that $w\in C^{4,\zeta}_{\beta'}(\mathcal{C}_0)$ for some $\beta'\in(1/2,1)$. Therefore, $\psi(\phi)\in C^{4,\zeta}_{-2\beta'}(\mathcal{C}_0)$ and by (ii) of Corollary~\ref{iteration}, we find that $\phi\in
		C^{4,\zeta}_{-2\beta'}(\mathcal{C}_0)\oplus D_{a,1}(\mathcal{C}_0)$, which provides $\phi\in C^{4,\zeta}_{\beta}(\mathcal{C}_0)$ for $\beta'=\min\{2\beta',\beta_{a,2}\}$. 
		Besides, observe that $\beta'>\beta$ is optimal, which by Theorem~\ref{thm:andrade-doo19} implies \eqref{fisica2}.
	\end{proof}
	
	\begin{proof}[Proof of Theorem~\ref{theorem1}]    
		It follows directly from Proposition~\ref{prop:imporvedconvergence}, because the deformed Emden--Fowler solutions also satisfy the same estimate \eqref{fisica2}, which is a consequence of \eqref{lastequation}.
	\end{proof}
	
	In conclusion, we have the following refined asymptotics:
	
	\begin{corollary}
		Let $\mathcal{V}$ be a strongly positive solution to \eqref{sphevectfowler} and $\mathcal{V}_{a,T}$ an Emden--Fowler solution to \eqref{sphevectfowler}. Then, there exists $\beta^*_1>1$ such that
		\begin{equation*}    
			|\mathcal{V}(t,\theta)-\mathcal{V}_{a,T}(t)-\pi_0\left[\mathcal{V}\right](t,\theta)-\pi_1\left[\mathcal{V}\right](t,\theta)|\leqslant Ce^{-\beta^*_1 t} \quad {\rm for} \quad t>0,
		\end{equation*}
		where $\beta^*_{a,1}=\min\{2,\beta_{a,2}\}$ and $\pi_0,\pi_1$ are the projections on the zero-frequency and low-frequency eigenspaces, respectively.
	\end{corollary}
	
	\section{Final considerations}\label{problemssection}
	In this section, we make some further comments about possible natural extensions for the asymptotics in Theorems~\ref{theorem1} and \ref{theorem1'}\textcolor{blue}{'}.\\
	
	\noindent {\bf Higher order refined asymptotics.} In the same spirit of \cite{arXiv:1909.10131v1,arXiv:1909.07466v1}, we believe that the following higher order asymptotics must hold for solutions to \eqref{sphevectfowler},
	\begin{equation}\label{higherorderasymptotics}
	\left|\mathcal{V}(t,\theta)-\mathcal{V}_{a,T}(t)-\Lambda^*\sum_{\ell=1}^{m}\sum_{j=0}^{m-1}b_{\ell j}(t,\theta)t^je^{-\beta_{a,j}}\right|\leqslant Ct^me^{-\beta_{a,m} t} \quad \mbox{for} \quad t>0,
	\end{equation}
	where $-\beta_{a,m}$ is the $m$-th indicial roots the powers $t^j$ for $j=1,\ldots,m$ correspond to the higher algebraic multiplicity.\\
	
	\noindent {\bf Geometric fourth order systems.} Besides, notice that \eqref{oursystem} is a particular case of the following fourth order geometric system, 
	\begin{equation}\label{geometricsystem}
	P_gu_i=c(n)|\mathcal{U}|^{2^{**}-2}u_i \quad \mbox{in} \quad (B_1^*,g) \quad \mbox{for} \quad i\in I.
	\end{equation}
	Here $B_1^n(0)\subset\mathbb{R}^n$ with $n\geqslant5$, $B_1^*=B_1^n(0)\setminus\{0\}$ is the punctured ball, $g$ is a smooth metric on $B_1^n(0)$, $\mathcal{U}=(u_1,\dots,u_p)$ is a $p$-map with Euclidean norm denoted by $|\mathcal{U}|$, $f_i(\mathcal{U})=c(n)|\mathcal{U}|^{2^{**}-2}u_i$ is the Gross--Pitaevskii nonlinear strong coupling term with $2^{**}=2n/(n-4)$ the critical Sobolev exponent, $c(n)$ is a normalizing constant and $P_g$ is the fourth order operator given by
	\begin{equation*}
	P_gu_i=\Delta_{g}^{2} u_i-\sum_{j=1}^{p}\dive\left(\frac{(n-2)^2+4}{2(n-1)(n-2)} R_{g}g-\frac{4}{n-2}\Ric_g\right)\ud u_j+\frac{n-4}{2}\sum_{j=1}^{p}Q_{ij}(x)u_j,
	\end{equation*}
	where $Q_{ij}(x)\in {\rm Sym}_p(\mathbb{R})$, $\Ric_g$ is the Ricci tensor and $R_g$ is the scalar curvature of $(B_1^*,g)$. Notice that when $p=1$ and $Q_{11}=Q_g$, the $Q$-curvature, the system \eqref{geometricsystem} is the celebrated singular $Q$-curvature equation \eqref{qcurvature}. Then, a natural problem is to extend the asymptotics in Theorem~\ref{theorem1} to the inhomogeneous System \eqref{geometricsystem}. 
	It is appealing to find some hypotheses on $Q_{ij}(x)\in M_{p}^s(\mathbb{R})$ and bounds on the dimension such that \eqref{asymptotics} still holds for solutions to the geometric system \eqref{geometricsystem}. This would be a extension of the results in \cite[Theorem~1]{MR2393072} and \cite[Theorem~1.3]{MR4002167}.\\

	\noindent {\bf Low-dimensional $Q$-curvature equation.} 
	For the sake of completeness, it would be also interesting to study the lower-dimensional cases associated to \eqref{geometricsystem}, at least for the scalar case $p=1$. 
	First, it also makes sense to study the three-dimensional case of the $Q$-curvature problem. The difference is that the nonlinear coupling term on the right-hand-side has a negative power,
	\begin{equation}\label{3qcurvature}
	P_gu=c_3u^{-7} \quad \mbox{in} \quad B_1^*.
	\end{equation}
	The objective is to find the local asymptotic behavior for solutions to \eqref{3qcurvature} near the isolated singularity \cite{MR4029726,MR3669775,MR3465087,MR3962197}.
	
	Second, one can also consider the four-dimensional $Q$-curvature equation. It is well-known that due to the Trudinger--Moser inequality that the critical nonlinearity is of the form,
	\begin{equation}\label{4qcurvature}
	P_gu=c_4e^{4u} \quad \mbox{in} \quad B_1^*,
	\end{equation}
	this is also called the Liouville equation.
	The objective would be to establish qualitative properties for solution to \eqref{4qcurvature}, such as asymptotic radial symmetry and local behavior \cite{MR1611691,MR3604948,MR3615163,arXiv:1804.09261v2,hyder-mancini-martinazzi}.
	Let us remark that the second order case is still open.\\
		
	After this manuscript was finished, we learned that J. Ratzkin \cite{arxiv.2001.07984} had obtained a similar result to Corollary~\ref{theorem2}.\\
	
	\begin{acknowledgement}
		The paper was completed while the first-named author was visiting the Department of Mathematics at Princeton University, whose hospitality he gratefully acknowledges. 
		He would also like to express his gratitude to Professor Fernando Cod\'a Marques.
	\end{acknowledgement}
	

\end{document}